\documentclass{amsart}
\usepackage{amssymb}
\usepackage[mathscr]{eucal}
\usepackage{xspace}

\newcommand{\R}{\mathbb{R}}

\newcommand{\Z}{\mathbb{Z}}
\newcommand{\N}{\mathbb{N}}
\newcommand{\C}{\mathbb{C}}
\newcommand{\LL}{\mathcal{L}}

\renewcommand{\H}{\mathcal{H}}
\newcommand{\CS}{\mathcal{S}}
\newcommand{\CR}{\mathcal{R}}
\newcommand{\CO}{\mathcal{O}}

\newcommand{\E}{{\mathcal E}} 
\newcommand{\F}{{\mathcal F}}

\newcommand{\po}{\partial}

\newcommand{\wto}{\rightharpoonup} 
\newcommand{\ve}{\varepsilon}

\newcommand{\la}{\langle}
\newcommand{\ra}{\rangle}

\newcommand{\vphi}{\varphi}
\newcommand{\loc}{{\text{\rm loc}}}
\newcommand{\X}{\times}

\renewcommand{\d}{\delta}
\renewcommand{\l}{\lambda}

\renewcommand{\a}{\alpha}
\renewcommand{\b}{\beta}

\newcommand{\s}{\sigma}
\newcommand{\g}{\gamma} 
\newcommand{\z}{\zeta}
\renewcommand{\k}{\kappa}
\newcommand{\sgn}{\text{\rm sgn}}

\newcommand{\Om}{\Omega}
\newcommand{\om}{\omega}

\newcommand{\supp}{\text{\rm supp}\,}

\newcommand{\M}{{\mathcal M}}

\renewcommand{\E}{{\mathcal E}}
\newcommand{\DM}{\mathcal D\mathcal M}

\renewcommand{\div}{\text{\rm div}\,}

\renewcommand{\supp}{\text{\rm supp}\,}

\newcommand{\Div}{{\rm div}}

\newcommand{\cO}{{\mathcal O}}
\newcommand{\cU}{{\mathcal U}}
\newcommand{\cV}{{\mathcal V}}
\newcommand{\cE}{{\mathcal E}}
\newcommand{\cL}{{\mathcal L}}

\newcommand{\cP}{{\mathcal P}}
\newcommand{\cB}{{\mathcal B}}
\newcommand{\cA}{{\mathcal A}}
\newcommand{\cR}{{\mathcal R}}

\newcommand{\bfa}{{\mathbf a}}

\newcommand{\esslim}{\operatorname{ess}\!\lim}

\newcommand{\sign}{\operatorname{sign}}

\newcommand{\bbf}{\mathbf{b}}
\newcommand{\abf}{\mathbf{a}}

\newcommand{\Abf}{\mathbf{A}}
\newcommand{\Bbf}{\mathbf{B}}
\newcommand{\Rbf}{\mathbf{R}}
\newcommand{\Sbf}{\mathbf{S}}

\newcommand{\Lip}{\text{\rm Lip}}

\renewcommand{\S}{{\mathcal S}}
\newcommand{\B}{{\mathcal B}}

\newcommand{\bbE}{{\mathbb E}}

\newcommand{\bbP}{{\mathbb P}}
\newcommand{\bbM}{{\mathbb M}}
\newcommand{\bbB}{{\mathbb B}}

\newcommand{\fkU}{{\frak U}}

\newcommand{\dist}{\text{dist\,}}

\renewcommand{\Lip}{\text{Lip\,}}

\newcommand{\un}[1]{\underline{#1}}

\newcommand{\ul}{\underline}
\newcommand{\wh}{\widehat}
\newcommand{\wt}{\widetilde}
\newcommand{\ff}{\mathfrak{f}}

\newcommand{\qq}{\mathfrak{q}}

\newcommand{\fF}{\mathfrak{F}}
\newcommand{\FF}{\mathfrak{F}}

\theoremstyle{plain}
\newtheorem{theorem}{Theorem}[section]
\newtheorem{corollary}{Corollary}[section]

\newtheorem{lemma}{Lemma}[section]
\newtheorem{proposition}{Proposition}[section]
\theoremstyle{definition}
\newtheorem{definition}{Definition}[section]
\theoremstyle{remark}

\newtheorem{remark}{Remark}[section]
\numberwithin{equation}{section}

\begin{document}

\title[Boundary Value Problem for Stochastic Parabolic Equations]
{A Boundary Value Problem for a Class of Anisotropic Stochastic Degenerate Parabolic-Hyperbolic Equations }
\author[H.~Frid]{Hermano Frid} 

\address{Instituto de Matem\'atica Pura e Aplicada - IMPA\\
         Estrada Dona Castorina, 110 \\
         Rio de Janeiro, RJ 22460-320, Brazil}
\email{hermano@impa.br}

\author[Y.~Li]{Yachun Li}

\address{School of Mathematical Sciences, MOE-LSC, and SHL-MAC,  Shanghai Jiao Tong University\\ Shanghai 200240, P.R.~China}

\email{ycli@sjtu.edu.cn}

\author[D.~Marroquin]{Daniel Marroquin}
\thanks{D.~Marroquin thankfully acknowledges the support from CNPq, through grant proc. 150118/2018-0.}

\address{Instituto de Matem\'{a}tica - Universidade Federal do Rio de Janeiro\\
Cidade Universit\'{a}ria, 21945-970, Rio de Janeiro, Brazil}
\email{marroquin@im.ufrj.br}

%
\author[J.F.C.~Nariyoshi] {Jo\~ao F.C.~Nariyoshi}
\thanks{J.F.C.~Nariyoshi appreciatively acknowledges the support from CNPq, through grant proc. 140600/2017-5, and from FAPESP, through grant proc. 2021/01800-7. }

\address{Instituto de Matem\'{a}tica, Estat\'{i}stica e Computa\c{c}\~ao Cient\'{i}fica da Universidade de Campinas - IMECC/Unicamp\\  Rua S\'{e}rgio Buarque de Holanda, 651 \\
 Campinas, SP, 13083-859, Brazil}
\email{jfcnar@unicamp.br}

\author[Z.~Zeng]{Zirong Zeng}

\address{School of Mathematical Sciences, MOE-LSC, and SHL-MAC,  Shanghai Jiao Tong University\\ Shanghai 200240, P.R.~China}

\email{beckzzr@sjtu.edu.cn}

\date{}
\thanks{}

\begin{abstract} We establish the well-posedness of an initial-boundary value problem of mixed type for a stochastic nonlinear parabolic-hyperbolic equation on a space domain $\cO=\cO'\X\cO''$ where a Neumann boundary condition is imposed on $\po\cO'\X\cO''$, the hyperbolic boundary, and a Dirichlet condition is imposed on $\cO'\X\po\cO''$, the parabolic boundary. Among other points to be highlighted in our analysis of this problem we mention  the new strong trace
theorem for the special class of stochastic nonlinear parabolic-hyperbolic equations studied here, which is decisive for the uniqueness of the kinetic solution, and the new averaging lemma for the referred class of equations which is a vital part of the proof of the strong trace property. We also provide a detailed analysis of the approximate nondegenerate problems, which is also made here for the first time, as far as the authors know, whose solutions we prove to converge to the solution of our initial-boundary value problem.    

\end{abstract}

\maketitle 

\tableofcontents

\section{Introduction}\label{S:1}

Let $\cO=\cO'\X \cO''$ be a bounded smooth open subset of $\R^d$, with $\cO'\subset\R^{d'}$, $\cO''\in\R^{d''}$, $d=d'+d''$. We consider the following mixed-type initial-boundary value problem for a quasilinear degenerate parabolic-hyperbolic stochastic partial differential equation 
\begin{align}
& du+ \div(\Abf (u))\,dt=  D_{x''}^2: \Bbf (u)\,dt +\Phi(u)\,dW,\qquad x\in\cO,\, t\in(0,T),\label{e1.1}\\
&u(0,x)=u_0(x),\quad\text{for $x\in\cO$}, \label{e1.2}\\
&u(t,x)= u_b(t),  \quad \text{for $(t,x)\in \cO'\X\po\cO''$},  \label{e1.3}\\
&\Abf(u)\cdot\nu=0,\quad \text{for $(t,x)\in \po\cO'\X\cO''$}, \label{e1.4}
\end{align}
where, $\Abf:\R\to\R^d$ and $\Bbf:\R\to \bbM^{d''}$ are smooth maps,  $\bbM^{d''}$ denote the space of $d''\X d''$ matrices, $\Bbf=(B_{ij})_{i,j=d'+1}^d$, and $D_{x''}^2=(\po_{x_i x_j}^2)_{i,j=d'+1}^{d}$. For $\Rbf=(R_{ij}),\,\Sbf=(S_{ij})\in\bbM^{d''}$ we denote 
$\Rbf:\Sbf= \sum_{i,j}R_{ij}S_{ij}$ and, by extrapolation, $D_{x''}^2:\Bbf=\sum_{i,j=d'+1}^d\po_{x_ix_j}^2 B_{ij}$ . The matrix $\Bbf(u)$ is symmetric and its derivative $\bbf(u)=\frac{d}{du}\Bbf(u)$ is a symmetric nonnegative $d''\X d''$ matrix. 
    $W$ is a cylindrical Wiener process.

\subsection{Hypotheses}\label{SS:1.1}
The flux function $\Abf=(A_1, \cdots, A_d):\R\to\R^d$ is assumed to be of class $C^2$ and we denote its derivative by $\abf=(a_1,\cdots,a_d)$. The diffusion matrix $\bbf=(b_{ij})_{i,j=d'+1}^d:\R\to \bbM^{d''}$ is symmetric and positive semidefinite. Its square-root matrix, also symmetric and positive semidefinite, is denoted by $\s$, which is assumed to be bounded and locally $\g$-H\"older continuous for some $\g>1/2$, that is, 
\begin{equation}\label{e1.4o}
|\s(\xi)-\s(\z)|\le C(R) |\xi-\z|^\g\qquad \text{for all $\xi,\z\in\R$, $|\xi-\z|<1$}.
\end{equation}
Moreover, we assume that, for some $b\in C^1(\R)$, with $\frac{db(u)}{du}>0$, for a.e.\ $u\in\R$,  for some constant $\Lambda>1$  and for all $\xi=(\xi_{d'+1},\cdots,\xi_d)\in\R^{d''}$ we have
\begin{equation}\label{e1.4'}
\frac{db(u)}{du}^2|\xi|^2\le \sum_{i,j=d'+1}^d b_{ij}(u)\xi_i\xi_j\le \Lambda \, \frac{db(u)}{du}^2|\xi|^2.
\end{equation}
As it was observed in \cite{FL},  \eqref{e1.4'} implies that the $B_{ij}$'s are locally Lipschitz  functions of $b(u)$, that is, there exists a locally Lipschitz continuous functions 
$\tilde B_{ij}$ such that   $B_{ij}(u)=\tilde B_{ij}(b(u))$, for all $i,j=d'+1,\cdots,d$.  Relation \eqref{e1.4'} immediately implies 
\begin{equation}\label{e1.4''}
\frac{db(u)}{du}|\xi|^2\le \sum_{i,j=d'+1}^d \s_{ij}(u)\xi_i\xi_j\le \Lambda^{1/2} \, \frac{db(u)}{du}|\xi|^2,
\end{equation}  
and similarly to $\bbf$, we deduce that $\Sigma(u):=\int_0^u\s(\z)\,d\z$ is a locally Lipschitz $d''\X d''$-matrix function of $b(u)$, that is, there is a locally Lipschitz $d''\X d''$-matrix function $\tilde \Sigma$ such that $\Sigma(u)=\tilde \Sigma(b(u))$.

Further, we require a nondegeneracy condition for the symbol $\LL$ associated to the kinetic form of \eqref{e1.1}. In order to have spatial regularity of  kinetic solutions we localize the $\chi$-function associated to such solution and so, for $\ell>0$ sufficiently large,  we may view our localized $\chi$-functions as periodic with period $\ell$. The symbol is defined by
$$
\LL(i\tau,i n, \xi):=i(\tau+\abf(\xi)\cdot n) + {n''}^\top \bbf(\xi)n'',
$$  
where   $n\in\ell\, \Z^d$ and  we write $n=(n',n'')$, for $n\in\Z^d$, where $n'\in\Z^{d'}$ and $n''\in\Z^{d''}$. The reduced symbol is defined by
$$
\LL_0(i\tau,i n, \xi):=i(\tau+\abf'(\xi)\cdot n') + {n''}^\top \bbf(\xi)n'',
$$  
where $\abf'(\xi)=(a_1(\xi),\cdots, a_{d'}(\xi))$,  $n\in\ell\, \Z^d$.
 
For $J,\d>0$ and $\eta\in C_b^\infty(\R)$ nonnegative, let 
\begin{equation*}
\begin{aligned}
\Om_{\LL}^\eta(\tau,\eta;\d)&:=\{\xi\in \supp\eta\,:\, |\LL(i\tau, i n, \xi)|\le \d\},\\
\om_{\LL}^\eta(J;\d) &:= \sup_{\tiny\begin{matrix} \tau\in\R, n\in \ell\,\Z^d\\ |n|\sim J\end{matrix}}|\Om_{\LL}^\eta(\tau,i n;\d)|.
\end{aligned}
\end{equation*}
Let $\LL_{\xi}:=\po_\xi\LL$. We suppose that there exist $\a\in (0,1)$, $\b>0$ and a measurable function  $\vartheta\in L_\loc^\infty(\R;[1,\infty))$ such that
\begin{equation}\label{e1.4'''}
\begin{aligned}
\om_{\cL}^\eta(J;\d) &\lesssim_\eta \left(\frac{\d}{J^\b}\right)^\a,\\
\sup_{\tiny{\begin{matrix}\tau\in\R,n\in\Z^d\\|n|\sim J \end{matrix}}}\sup_{\xi\in\supp \eta}\frac{|\cL_{\xi}(i\tau,in;\xi)|}{\vartheta(\xi)}&\lesssim_\eta J^\b,\qquad \forall \d>0,\, J\gtrsim 1,
\end{aligned}
\end{equation}
where we employ the usual notation $x\lesssim y$, if $x\le Cy$, for some absolute constant $C>0$, and $x\sim y$, if $x\lesssim y$ and $y\lesssim x$. 

The following example in the case where $d=2$,  corresponds to the one  in corollary~4.5 in \cite{TT}  where conditions  \eqref{e1.4''} are verified:
$$
du+\po_{x_1}(\frac1{l+1}u^{l+1})\,dt = \po_{x_2}^2(\frac1{n+1}|u^n| u)\,dt+ \Phi(u)\,dW,
$$
where $l,n\in\N$ satisfy $n\ge 2l$. The same argument as in corollary~4.5 of \cite{TT} applies to the corresponding equation in any space dimension $d$,   replacing $\po_{x_2}^2$ in the above equation by $\Delta_{x''}:=\po_{x_2}^2+\cdots+\po_{x_d}^2$.  Clearly, many other similar examples may be given.

We observe that conditions \eqref{e1.4'''}  imply  the weaker  non-degeneracy condition: For $(\tau,\k)\in\R^{d+1}$, with $\tau^2+|\k|^2=1$,  
\begin{equation}\label{e1.nondeg0}
\left| \left\{\xi \in\supp\eta \,:\, |\tau+\bfa(\xi)\cdot\k|^2+ \left( {\k''}^{\top} \bbf(\xi)\k''\right)^2 =0 \right \} \right|=0,
\end{equation}
and, due to \eqref{e1.4'} and the fact that $\frac{db(u)}{du}>0$, for a.e.\ $u\in\R$,   \eqref{e1.nondeg0} implies
\begin{equation}\label{e1.nondeg}
\left| \left\{\xi \in\supp\eta \,:\, |\tau+\bfa'(\xi)\cdot\k'|^2+ \left( {\k''}^{\top} \bbf(\xi)\k''\right)^2 =0 \right \} \right|=0,
\end{equation}
where $\abf'(\xi):=(a_1(\xi),\cdots,a_{d'}(\xi))$.

We are going to seek  solutions of the initial-boundary value problem \eqref{e1.1}-\eqref{e1.4} which assume values in an interval, say $[u_{\text{min}},u_{\text{max}}]$, such that $u_0(x),u_b(x)\in [u_{\text min},u_{\text max}]$, and
\begin{equation}\label{e1.f}
\Abf'(u_{\text min})=\Abf'(u_{\text max})=0,
\end{equation}
where $\Abf'(u):= (A_1(u),\cdots,A_{d'}(u))$.

As to the stochastic term, we adopt the framework of \cite{FL2}, similar to that in \cite{DV, DHV, GH}.  Let $(\Om,\, \F,\, (\F_t)_{t\ge0},\,$ $ \bbP)$ be a stochastic basis with a complete, right-continuous filtration. Let $\cP$ denote the predictable $\s$-algebra on $\Om\X[0,T]$ associated to $(\F_t)_{t\ge0}$. The initial datum may be random, $\F_0$-measurable, and we assume $u_0\in L^\infty(\Om\X\cO)$.  The process $W$ is a cylindrical Wiener process, 
$$
W(t)=\sum_{k\ge1}\b_k(t) e_k,
$$
with $(\b_k)_{k\ge1}$ being mutually independent real-valued standard Wiener processes relative to $(\F_t)_{t\ge0}$ and $(e_k)_{k\ge1}$ a complete orthonormal system in a separable Hilbert space $\fkU$. For each $u\in L^2(\cO)$, $\Phi(u): {\frak U}\to L^2(\cO)$ is defined by $\Phi(u)e_k=g_k(u(\cdot))$, where $g_k(\cdot)$ is a regular function on $\mathbb{R}$ satisfying the bounds
 \begin{equation}\label{e1.4*}
 |g_k(0)|+|\po_\xi g_k(\xi)|+|\po_\xi^2 g_k(\xi)|\le \a_k,\qquad \forall \xi\in\R,
 \end{equation}
 where $(\a_k)_{k\ge 1}$ is a sequence of positive numbers satisfying $D:=4\sum_{k\ge1}\a_k^2<\infty$.
 Observe that  \eqref{e1.4*} implies
 \begin{align}
 &G^2(u)=\sum_{k\ge1}|g_k(u)|^2\le D(1+|u|^2),\label{e1.4**}\\
 &\sum_{k\ge1}|g_k(u)-g_k(v)|^2\le D(|u-v|^2),\label{e1.5*}
 \end{align}
for all  $x,y\in\cO$, $u,v\in\R$.

The conditions on $\Phi$ imply that $\Phi: L^2(\cO)\to L_2({\frak U}; L^2(\cO))$,  where the latter denotes the space of Hilbert-Schmidt operators from ${\frak U}$ to $L^2(\cO)$. In particular, given a predictable process $u\in L^2(\Om\times [0,T]; L^2(\cO))$, the stochastic integral is a well defined process taking values in $L^2(\cO)$. Indeed, for each $u\in L^2(\cO)$, it follows from \eqref{e1.4*} that
$$
\sum_{k\ge1}\|g_k(u)\|_{L^2(\cO)}^2\le D(1+\|u\|_{L^2(\cO)}^2).
$$ 

In this setting, we can assume without loss of generality  that the $\s$-algebra $\F$ is countably generated and $(\F_t)_{t\ge0}$ is the filtration generated by the Wiener process and the initial condition. 
 
As aforementioned, we will look for bounded solutions of the initial-boundary value problem \eqref{e1.1}-\eqref{e1.3} which assume values in the interval $[u_{\text min},u_{\text max}]$. Accordingly, on top of \eqref{e1.f}, we assume that $[u_{\text min},u_{\text max}]\subset (-L_0,L_0)$, where $L_0$ is as in the nondegeneracy condition \eqref{e1.4'''}. Moreover, we assume that
\begin{equation}\label{e1.8'}
g_k(u_{\rm min})=g_k(u_{\rm max})=0, \quad\text{for all $k=1,2,\cdots$}. 
\end{equation}
We also assume that $u_0\in L^\infty(\cO)$ is deterministic (for simplicity) and satisfies $u_{\rm min} \le u_0\le u_{\rm max}$. Let us point out that $u_0$ may be random, in which case it should be assumed that it is $\mathcal{F}_0$-measurable. The extension to this more general setting is straightforward and follows the arguments below line by line just by adding an expectation where integration in the random parameter takes place. The conditions on $u_b$ are given below (see \eqref{e1.ub}). 

\subsection{Definitions and main result} \label{SS:1.2}

\begin{definition}[Kinetic measure]\label{D:1.1} A mapping $m$ from $\Om$ to $\M_b^+([0,T]\X\cO\X\R)$,the set of nonnegative bounded measures over $[0,T]\X\cO\X\R$, is said to be a kinetic measure if the following holds:
\begin{enumerate}
\item[(i)] $m$ is measurable, in the sense that for each $\psi\in C_0([0,T]\X\cO\X\R)$ the mapping $m(\psi):\Om\to\R$ is measurable, where by $C_0$ we denote the space of continuous functions vanishing at the boundary or when the norm of  the argument goes to infinity. 
\item[(ii)] $m$ vanishes for large $\xi$: if $B_R^c=\{\xi\in\R\,:\, |\xi|\le R\}$, then  
$$
\lim_{R\to\infty} \bbE\, m([0,T]\X\cO\X B_R^c)=0,
$$
\item[(iii)] for any $\psi\in C_0(\cO\X\R)$
$$
\int_{[0,t]\X\cO\X\R}\psi(x,\xi)\,dm(s,x,\xi)\in L^2(\Om\X[0,T])
$$
admits a predictable representative. 
\end{enumerate}
\end{definition}

Concerning the Dirichlet condition in the next definition we make the following comments and further assumptions. First,  we assume the following:
\begin{equation}\label{e1.Dir}
\text{$\Bbf(u)$ is diagonal, that is, $B_{ij}(u)\equiv0$ for $i\ne j$.}
\end{equation}

Second, we introduce the functions
 \begin{equation}\label{e1.10}
 \begin{aligned}
 & F(u,v):=\sgn (u-v) (A(u)-A(v)),\\
 &\bbB(u,v):=(\sgn(u-v)(B_{ij}(u)-B_{ij}(v)))_{i,j=1}^d\\
 & K_{x''}(u,v):=  \nabla_{x''}\cdot \bbB(u,v) -F(u,v),\\
 & H_{x''}(u,v,w):= K_{x''}(u,v)+K_{x''}(u,w)-K_{x''}(w,v),\\
 & \mathbb{G}(u,v):=\sgn (u-v)(\Phi(u)-\Phi(v)),
 \end{aligned}
 \end{equation}
where  $\nabla_{x''}\cdot \bbB(u,v)$ is the $d$-vector with components 
$$
(\nabla_{x''}\cdot \bbB (u,v))_j=\begin{cases}
0, &\text{if }j\le d',\\
\sum\limits_{i=d'+1}^d\po_{x_i}(\sgn(u-v)(B_{ij}(u)-B_{ij}(v))), &\text{if }d'+1\le j\le d,
\end{cases}
$$
and $\mathbb{G}:L^2(\cO)^2\to L_2(\mathfrak{U};L^2(\cO^2))$ is given by $\mathbb{G}(u,v)e_k=\sgn(u-v)(g_k(x,u)-g_k(x,v))$, $k\ge 1$.

Similarly, we define $F_+$, $\mathbb{B}_+$ and $\mathbb{G}_+$ as their counterparts in \eqref{e1.10} with $\sgn(\cdot)_+$ instead of $\sgn(\cdot)$.

We also define
\begin{equation}\label{e1.11}
\cA(u,v,w)=|u-v|+|u-w|-|w-v|.
\end{equation}

 Third, in order to take advantage of the fact that $\po\cO''$  is locally the graph of a $C^2$ function, we introduce a system of balls $\B''$, with the following property. For each $B''=B''(x_0'',r)\in\B''$,  a  ball with center at an arbitrary $x_0''\in\po\cO''$  we have that for some $\gamma\in C^2(\R^{d''-1})$,
\begin{equation}\label{e1.12}
B''\cap \cO''=\{(\bar y'',y_{d})\in B\,:\, y_{d}<\gamma(\bar y''), \bar y''=(y_{d'+1},\cdots,y_{d-1})\in \R^{d''-1}\},
\end{equation}
where the coordinate system $(y_{d'+1},\cdots,y_{d})$ is obtained from the original $(x_{d'+1}, \cdots,$ $x_{d})$ by  relabelling, reorienting and translating.  By relabelling we mean a permutation of the coordinates and by reorienting we mean changing the orientation of one of the coordinate axes. 

Fourth, we assume that $u_b$ is predictable, satisfies $u_{\text min}\le u_b(t,x)\le u_{\text max}$, $(t,x)\in(0,T)\X\po\cO$ and that
\begin{equation}\label{e1.ub}
u_b\in L^2(\Om\times[0,T];X)\cap L^2(\Om;H^1((0,T);L^2(\cO)))\cap L^4(\Om\times[0,T];Y),
\end{equation}
where, $X=L^2(\cO';H^4(\po\cO''))\cap L^2(\po\cO'';H_0^1(\cO')\cap H^2(\cO'))$ and $Y=L^4(\cO';W^{1,4}(\po\cO''))$. Condition \eqref{e1.ub} is intended to ensure that, given $B''\in\mathcal{B}''$ satisfying \eqref{e1.12}, there is an extension $u_{B''}\in L^2(\Om\times[0,T]; L^2(\cO';H^4(\cO)))\cap L^2(\Om\times[0,T];L^2(\cO'';H^2(\cO')))$ of $u_b$ to $\cO$ satisfying (strongly) the following
\begin{align}
&du_{B''} = \Delta u_{B''}dt -\Delta_{x''}^2 u_{B''} dt + \Phi(u_{B''})dW(t),\qquad x\in \cO,\, t\in(0,T),\label{e1.101}\\
&u_{B''}(0)=u_{B''0},\label{e1.102}\\
&u_{B''}(t)\big|_{\cO'\times\po\cO''}= u_b(t),\label{e1.102'} \\
&\frac{\partial u_{B''}}{\partial x_d}(t)\big|_{\cO'\times(\po\cO''\cap B'')} = 0,\label{e1.103}
\end{align}
where, $\Delta_{x''}^2=\sum_{i=d'+1}^d\po_{x_ix_i}^2\sum_{d'+1}^d\po_{x_jx_j}^2$ denotes the bi-Laplacian operator in the parabolic variables and $u_{B''0}$ is a smooth extension of $u_b(0,\cdot)$ to $\cO'\times B''$ such that $\frac{\partial u_{B0}}{\partial x_d}\big|_{\cO'\times(\po\cO''\cap B'')} = 0$. 

\begin{remark}
Actually, we only need to assume that, for each $B''\in\mathcal{B}''$, $u_b|_{\cO'\times (\po\cO''\cap B'')}$ is the restriction to $\cO'\times(\partial \cO''\cap B'')$ of a strong solution of a stochastic equation whose noise term is given by $\Phi(u_{B''})dW$, also satisfying \eqref{e1.103}.  It is possible to show that under the hypothesis \eqref{e1.4*} and assuming \eqref{e1.ub}, then there are strong solutions to \eqref{e1.101}-\eqref{e1.103}, in particular. The proof of this statement goes by the same lines as in  Appendix~A in \cite{FL2}, but using the operator $A=\Delta-\Delta_{x''}^2$ with domain 
\begin{multline*}
D(A) = \{ u \in L^2(\cO' \times \cO''); L^2(\cO''; H^2(\cO')) \cap u \in L^2(\cO'; (H_0^2 \cap H^4)(\cO'))  \\  
 \text{ and, in the sense of traces, } \partial_\nu u = 0 \text{ on } \partial \cO' \times \cO'' \},
\end{multline*}
instead of the bi-Lapacian, and using the results from Appendix~\ref{A4} and Appendix~\ref{B} below. The fact that there is such a function $u_{B''}$ that is a mild solution of \eqref{e1.101}-\eqref{e1.103} can be shown by standard methods via a fixed point argument by assuming only that $u_b\in L^2(\Om\times[0,T];X)\cap L^2(\Om;H^1((0,T);L^2(\cO)))$. The extra regularity assumed in \eqref{e1.ub} guarantees that the solution is in fact strong. In particular, the assumption that $u_b\in L^4(\Om\times[0,T];Y)$ in connexion with \eqref{e1.4*} is used to improve the regularity of the mild solution by adapting some ideas from \cite{Ha2}.  We refer to Appendix~A in \cite{FL2} for the details. 
\end{remark}

\begin{remark}
In the deterministic setting, that is, when $\Phi=0$, $u_{B''}$ may simply be obtained using \eqref{e1.12} by  setting $u_{B''}(\bar{x},x_d)=u_b(\bar{x})$, for $x=(\bar{x},x_d)\in B''\cap \overline{\cO}$ (cf. \cite{FL}). As in \cite{FL2} (cf. \cite{FL,MPT}), the extension of the values of the parabolic boundary to the interior of the domain will be used below to give meaning to the boundary condition \eqref{e1.3}. We will comment further on the extension of the boundary data satisfying \eqref{e1.101}-\eqref{e1.103} in Subsection~\ref{SS:1.4} below.
\end{remark}

%

\begin{definition}[Kinetic solution]\label{D:1.2} A predictable function $u\in L^\infty(\Om\X[0,T]\X \cO)$ is a kinetic solution of \eqref{e1.1}--\eqref{e1.3} if for all $p\ge 1$, 
$$
u\in L^p(\Om\X[0,T], \cP, d\bbP\otimes dt  ;  L^p(\cO)) \cap L^p(\Om; L^\infty(0,T; L^p(\cO))), 
$$  
and it satisfies the following:
\begin{enumerate} 
\item[(i)]{\bf Regularity:}  
\begin{equation}\label{e1.8'''}
\nabla_{x''} b(u) \in L^2(\Om\X[0,T]\X\cO),
\end{equation}
 where $\nabla_{x''}:=(\po_{x_{d'+1}},\cdots,\po_{x_d})$.
In particular, denoting $\Div_{x''}=\sum_{j=d'+1}^d\po_{x_j}$, we have that $\Div_{x''} \int_0^u\sigma(\xi)d\xi \in L^2(\Om\times[0,T]\times\cO)$.
\item[(ii)]{\bf Kinetic equation:} There exists a kinetic measure $m\ge n_1$, $\bbP$-a.s.\ such that the pair $(f={\bf 1}_{u>\xi}, m)$ satisfies, for all $\varphi\in C_c^\infty([0,T)\X\cO\X\R)$, $\bbP$-a.s., 
\begin{multline} \label{e1.9}
\int_0^T\la f(t),\po_t\varphi(t)\ra\, dt + \la f_0,\varphi(0)\ra +\int_0^T\la f(t), \abf\cdot\nabla\varphi(t)\ra\,dt\\
+\int_0^T\la f(t),{\bbf}:D_{x''}^2\varphi(t)\ra\,dt\\
=-\sum_{k\ge1}\int_0^T\int_{\cO}g_k(u(t,x))\varphi(t,x,u(t,x))\,dx\,\b_k(t)\\
-\frac12\int_0^T\int_{\cO} G^2(u(t,x))\po_\xi\varphi(t,x,u(t,x))\,dx\,dt+m(\po_\xi\varphi),
\end{multline}
where, $n_1:\Om\to\M_b^+([0,T]\X \cO\X\R)$ is defined as follows: for any $\varphi\in C_0 ([0,T]\X\cO\X\R)$ 
$$
n_1(\varphi)=\int_0^T\int_{\cO}\int_{\R}\varphi(t,x,\xi)\left|\div_{x''} \int_0^u\s(\z)\,d\z\right|^2\,d\d_{u(t,x)}(\xi)\,dx\,dt.
$$

\item[(iii)]{\bf Neumann condition on $\Gamma':=\po\cO'\X\cO''$:}  For all $\tilde \phi\in C_c^\infty((0,T)\X\R^{d'}\X\Om'')$,
\begin{multline}\label{e1.9N}
\int_0^T\int_{\cO}\{ u\po_t\tilde\phi +\Abf(u)\cdot\nabla \tilde\phi-\nabla_{x''}\cdot \Bbf(u)\cdot\nabla_{x''}\tilde\phi\}\,dx\,dt \\ +\int_0^T\int_{\cO} \tilde\phi\Phi(u)\,dx\, dW(t)=0,
\end{multline}
where $\nabla_{x''}\cdot \Bbf(u)$ is the $d''$-vector whose $j$-th component is $\sum_{i=d'+1}^d\po_{x_i}B_{ij}(u)$, for $d'+1\le j\le d$. Observe that the test function $\tilde\phi$ does not necessarily vanish over $\Gamma'_T:=(0,T)\X\Gamma'$. 
\end{enumerate}
Additionally, among all possible functions that satisfy (i), (ii) and (iii), item (iv) below will characterize uniquely the one that satisfies the boundary condition \eqref{e1.3}.
\begin{enumerate}
\item[(iv)] {\bf Dirichlet condition on $\Gamma'':=\cO'\X\po\cO''$:}  We say that $u$ satisfies \eqref{e1.3} if the following conditions hold.  For each $B''\in\B''$, and some random constant $C_*>0$ with finite expectation, depending only on $\Abf, \Bbf$ and $u_b$, we have a.s. and for all $0\le\tilde \varphi\in C_c^\infty((0,T)\X \cO'\times B'')$ that
\begin{multline}\label{e1.15}
\int_0^T\int_{\cO}\{|u(t,x)-u_B(t,x)|\po_t\tilde\varphi-K_{x''}(u(t,x),u_B(t,x))\cdot\nabla\tilde \varphi\}\,dx\, dt\\ 
+\sum_{k\ge 1}\int_0^T \int_\cO \mathbb{G}_k(u(t,x),u_B(t,x))\tilde \varphi \, dx\, d\beta_k(t)   \ge -C_*\|\tilde \varphi \|_{L^2(\cO\times [0,T])}.
\end{multline}
Also, if $v$ is any other kinetic solution of equation \eqref{e1.1} (possibly with different initial data $v_0$) and $\zeta_\delta''$ is any $\cO''$-boundary layer sequence (for whose precise definition we refer to Section~\ref{S:3}), then for all $0\le \tilde \phi \in C_c^\infty(\cO'\times B''\times\cO)$ and $0\le t\le T$, we have that
\begin{equation}\label{e1.16}
\liminf_{\delta\to 0}\bbE\int_0^t \int_{\cO^2} H_{x''}(u(s,x),v(s,y),u_B(s,x))\cdot\nabla_x\zeta_\delta''(x)\tilde \phi(x,y)\, dx\, dy\, ds\ge 0.
\end{equation}
Moreover,  for all $B''\in\cB''$ and for all $\phi\in C_c^1(B'')$, a.s.\ we have
\begin{multline}\label{e1.17}
\int_{\cO'}b(u(t,x))dx'=\int_{\cO'}b(u_b(t,x))dx' \\
 \text{ on } \po\cO\times (0,T)\text{ in the sense of traces in } L^2(0,T; H^1(B\cap\cO'')).
\end{multline}
%
%
\end{enumerate}
\end{definition}

\begin{remark}
Note that in the deterministic setting, that is, when $g_k\equiv 0$ for all $k\ge 1$, the constants are solutions of equation \eqref{e1.1} and the condition \eqref{e1.16} is only necessary to hold for constant solutions $v(t,y)=k$ (cf. \cite{MPT,FL}). In the present setting, like in \cite{FL2}, this is no longer the case.
\end{remark}

\begin{remark}[Chain rule]
Since, according to \eqref{e1.4''}, $\Sigma(u)=\int_0^u\s(\z)\,d\z$ is a locally Lipschitz function of $b(u)$, condition \eqref{e1.8'''} implies that $\nabla_{x''} \Sigma(u) \in L^2(\Om\times[0,T]\times\cO)$. Consequently, for any $0\le \vartheta\in C_b(\mathbb{R})$ the following chain rule formula holds in $L^2(\Om\times[0,T]\times\cO)$ (see the appendix in \cite{CP})
\begin{equation}\label{e1.c-r}
\Div_{x''} \int_0^u\vartheta(\xi)\sigma(\xi)d\xi = \vartheta(u)\Div_{x''} \int_0^u\sigma(\xi)d\xi, \qquad \text{in }\mathcal{D}'(\cO),\text{ a.e. }(\omega,t).
\end{equation}
\end{remark}

The main goal of this paper is to prove the following theorem. 

\begin{theorem}\label{T:1.1} Let $u_0\in L^\infty(\Om\X\cO)$ with $u_0,u_b\in[a,b]\subset(-M,M)$ and assume that \eqref{e1.4*} holds. Then, there is a unique kinetic solution of \eqref{e1.1}--\eqref{e1.4}, and it has almost surely continuous trajectories in $L^p(\cO)$, for all $p\in[1,\infty)$.
Moreover, if $u$ and $v$ are two kinetic solutions with initial data $u_0,v_0$ and  Dirichlet data $u_b,v_b$, we have
\begin{equation}\label{e1.18}
\bbE \int_{ \cO}|u(t,x)-v(t,x)|\,dx\le \bbE \int_{\cO}|u_0(x)-v_0(x)|\,dx, 
\end{equation}
for some $C>0$ depending only on the data of the problem. 
\end{theorem}

It is important to have also at hand the notion of entropy solution.

\begin{definition}\label{D:1.3} A bounded measurable function $u\in L^\infty(\Om\X[0,T]\X \cO)$ is a weak entropy solution of \eqref{e1.1}--\eqref{e1.4} if for all $p\ge 1$, 
$$
u\in L^p(\Om\X[0,T], \cP, d\bbP\otimes dt  ;  L^p(\cO)) 
$$  
and it satisfies conditions (i), (iii) and (iv) of Definition~\ref{D:1.2} and
\begin{enumerate}
\item[(ii')]  There exists a kinetic  measure $m:\Om\to \M_b^+([0,T]\X\cO\X\R)$ satisfying 
$$
m \ge \d(\xi-u(t,x)) \sum_{k=d'+1}^d\left(\sum_{i=d'+1}^d\po_{x_i}\int^u \s_{ik}(\xi)\,d\xi \right)^2,
$$
a.s., where $\d(\xi)$ denotes the Dirac measure concentrated at 0, such that,  for all $\eta\in C^2(\R)$, with $\Abf_\eta, \Bbf_\eta$ such that $\frac{d}{du}\Abf_\eta(u)=\frac{d}{du}\eta(u)\frac{d}{du}\Abf(u)$ and $\frac{d}{du}\Bbf_\eta(u)=
\frac{d}{du}\eta(u) \frac{d}{du}\Bbf(u)$, and for all
$\varphi\in C_c^\infty([0,T)\X\cO)$, 
\begin{multline}\label{e1.8E-1}
\int_0^T\la\eta(u(t)),\po_t\varphi\ra\,dt +\la \eta(u_0),\varphi(0)\ra+ \int_0^T\la \Abf_\eta(u(t)),\nabla \varphi\ra\,dt \\ +\int_0^T \la \div_{x''}(\Bbf_\eta(u(t))),\nabla_{x''}\varphi\ra\,dt
 =\int_{[0,T]\X\cO\X\R} \frac{d^2\eta}{d\xi^2}(\xi)\varphi \,dm(t,x,\xi) \\
 -\sum_{k\ge1}\int_0^T\la g_k(u(t))\frac{d\eta}{du}(u(t)), \varphi\ra\,d\b_k(t)-\frac12\int_0^T\la G^2(u(t))\frac{d^2\eta}{du^2}(u(t)),\varphi\ra\,dt,
 \end{multline}
 a.s.\ where  $\la\cdot,\cdot\ra$ represents the inner product of  $L^2(\cO)$ or $L^2(\cO;\R^d)$. 
\end{enumerate}
\end{definition}

The following proposition establishes the equivalence between the notions of kinetic and weak entropy solutions, in the context of $L^\infty$ solutions.

\begin{proposition}  \label{P:1.1}  For a bounded measurable function $u:\Om\X[0,T]\X\cO\to\R$ it is equivalent to be a kinetic solution of \eqref{e1.1}--\eqref{e1.4} and a weak entropy solution of \eqref{e1.1}--\eqref{e1.4}.
\end{proposition}

Since $u$ is bounded, the proof follows from the same arguments in the proof of the corresponding result
in \cite{CP}.

\subsection{Outline of the content}\label{SS:1.4} This paper extends to stochastic equations the results in \cite{FL}. This extension is far from trivial since the study of degenerate parabolic-hyperbolic stochastic equations in bounded domains requires the combined use of many deep 
results in the frontier of the research in mathematical analysis and probability theory as indicated by the recent articles \cite{DHV} and \cite{GH}. In particular, since we do not impose restrictions on the spatial support of the noise, we are forced to extend the strong trace theorem in \cite{FL} to the present stochastic context, establishing then a new strong trace property for degenerate parabolic-hyperbolic stochastic equations. Moreover, a new averaging lemma is stated and proved (see Lemma~\ref{L:5.1} below), which was a missing point  in \cite{FL},  as a decisive step in the proof of the new strong trace theorem in Section~\ref{S:4}.  

The Dirichlet boundary condition on the parabolic boundary also poses a challenging problem, specially in the proof of uniqueness of solutions, due to the presence of the noise. However, in \cite{FL2} the authors developed several techniques that enable the usage of the existence of normal weak traces for divergence measure fields in the present stochastic setting allowing for a delicate analysis of the solutions near the boundary, which can be reproduced in our present context with slight adaptations. Said analysis uses the extension $u_{B''}$ of the prescribed values on the parabolic boundary to the interior of the domain in order to control the values of the solution near the boundary. On the other hand, in order to deduce the consistency of the definition of kinetic solutions to \eqref{e1.1}--\eqref{e1.3} with limits obtained from the vanishing viscosity method, it is necessary to impose that $u_B$ be a strong solution to a stochastic equation whose noise term is given by $\Phi(u_B)dW$, in order to avoid infinite quadratic variation in the limit when comparing a solution with the boundary data near the boundary. This imposition precludes us from using the trivial extension given by $\tilde u_B(\bar{x},x_d)=u_b(\bar{x})$ considered in the deterministic case treated in \cite{MPT, FL}.

The theory for degenerate parabolic-hyperbolic stochastic equations is an extension of the theory for stochastic conservation laws, which in turn 
have a recent yet intense history. For the sake of examples, we mention Kim \cite{Kim} for the first result of existence and uniqueness of entropy solutions of the Cauchy problem for a one-dimensional stochastic conservation law, in the additive case, that is, $\Phi$ does not depend on $u$. Feng and Nualart \cite{FN}, where a notion of strong entropy solution is introduced, which is more restrictive than that of entropy solution, and  for which the uniqueness is established in the class of entropy solutions in any space dimension, in the multiplicative case, i.e., $\Phi$ depending on $u$; existence of such strong entropy solutions is proven only in the one-dimensional case.  Chen, Ding and Karlsen \cite{CDK}, where the result in \cite{FN} was improved and existence in any dimension was proven in the context of the functions of bounded variation.    Debussche and Vovelle in  \cite{DV}, where a major step  in the development of this theory was made with the extension of the concept of kinetic solution, originally introduced by Lions, Perthame and Tadmor in \cite{LPT},  for deterministic conservation laws,  to the context of stochastic conservation laws, for which the  well-posedness of the Cauchy problem was established in the  periodic setting in any space dimension.  Bauzet, Vallet and Wittbold \cite{BVW}, where the existence and uniqueness of entropy solutions for the general Cauchy problem  was proved in any space dimension (see also,   \cite{KSt}). Concerning boundary value problems,  Vallet and  Wittbold \cite{VW}, in the additive case,  and Bauzet, Vallet and Wittbold \cite{BVW1}, in the multiplicative case,  obtain existence and uniqueness of entropy solutions to  the homogeneous Dirichlet  problem, i.e., null boundary condition.  The methods and results introduced in \cite{DV}  were later extended to degenerated parabolic problems by  Debussche, Hofmanov\'a and Vovelle \cite{DHV} and Gess and Hofmanov\'a \cite{GH}, and we refer to these papers for other relevant references on this subject.

\section{Doubling of Variables and Kruzhkov Inequality}\label{S:2} 

 We start recalling the result establishing the existence of left- and right-continuous  representatives of a kinetic solution proved in \cite{DV,DHV}.  The same property holds here also and the proof is exactly the same as in  \cite{DV,DHV} to which we refer. 
 
\begin{proposition}[Left- and right-continuous representatives]\label{P;2.1}  Let $u$ be a kinetic solution to \eqref{e1.1}--\eqref{e1.4}. Then $f=1_{u>\xi}$ admits representatives $f^-$ and 
$f^+$  which are almost surely left- and right-continuous, respectively, at all points $t^*\in[0,T]$ in the sense of distributions over $\cO\X\R$. More precisely, for all $t^*\in[0,T]$ there exist kinet functions $f^{*,\pm}$ on $\Om\X\cO\X\R$ such that setting $f^\pm(t^*)=f^{*,\pm}$ yields $f^{\pm}=f$ almost everywhere and
$$
\la f^\pm(t^*\pm\ve),\psi\ra\to \la f^\pm(t^*),\psi\ra,\qquad \ve\downarrow0,\,\forall\psi\in C_c^2(\cO\X\R),\quad\text{$\bbP$-a.s.}
$$
Moreover, $f^+=f^-$ for all $t^*\in[0,T]$ except for some at most countable set.
\end{proposition}

The following result is a key step in the proof of uniqueness. The proof is similar to that of the corresponding result in \cite{FL2} with slight adaptations and so we omit it here. 

\begin{theorem}[Doubling of variables]\label{T2.100} Let $u$ and $v$ be kinetic solutions of \eqref{e1.1}--\eqref{e1.3} with initial data $u_0$ and $v_0$, respectively. Denote $\nabla_{x+y}=\nabla_x + \nabla_y$. Then, for $0\le t\le T$ and nonnegative test functions $\theta\in C_c^\infty([0,T))$ and $\phi\in C_c^\infty(\cO^2)$, we have a.s. that
  \begin{multline}\label{e4.1*}
  -\bbE\int_0^T\int_{\cO^2}\left(u^{\pm}(s,x)-v^{\pm}(s,y)\right)_+\theta'(s)\phi(x,y)\, dx\, dy dt \\
  \le  \bbE\int_{\cO^2}\left(u_0(x)-v_0(y)\right)_+\theta(0)\phi(x,y)\, dx\, dy\\
  +\bbE\int_0^T\int_{\cO^2}F_+(u(s,x),v(s,y))\cdot \theta(s)\nabla_{x+y}\phi(x,y)\, dx\, dy\, ds \\
  - \bbE\int_0^T \int_{\cO^2}\nabla_{x''+y''}\cdot\mathbb{B}_+(u(s,x),v(s,y))\theta(s)\nabla_{x''+y''}\phi(x,y)\, dx\, dy\, ds.
 \end{multline}   
\end{theorem}

As a consequence of Theorem~\ref{T2.100} we have the following Kruzhkov inequality.

\begin{theorem}\label{T:2.1}
Let $u$ and $v$ be kinetic solutions of \eqref{e1.1}--\eqref{e1.3} with initial data $u_0$ and $v_0$, respectively. Then, for any nonnegative test functions $\theta\in C_c^\infty((0,T))$ and $\psi\in C_c^\infty(\cO'\times\mathbb{R}^{d''})$, we have a.s. that
\begin{multline}\label{e3.unique1}
  -\bbE\int_0^T\int_{\cO}\left|u^{\pm}(s,x)-v^{\pm}(s,x)\right|\theta'(s)\psi(x)\, dx\, dy \\
  \le  \bbE\int_0^T\int_{\cO}K_{x''}(u(s,x),v(s,x))\cdot \nabla\psi_1(x)\, dx\, dy\, ds. 
 \end{multline}
\end{theorem}

The proof of this result can be carried out line by line as the proof of Theorem~3.3 in \cite{FL2}, where the authors prove a comparison inequality for solutions of the Dirichlet problem for quasilinear degenerate parabolic stochastic partial differential equations on a bounded domain. Note that in the parabolic case, that is, when $d'=0$, this result automatically yields uniqueness as we are allowed to take $\psi\equiv 1$ as a test function in \eqref{e3.unique1}, which is the case considered in \cite{FL2}. In our present case, we need the strong trace property from Section~\ref{S:4} below in order to obtain the uniqueness of solutions.

\begin{proof}
We only give a sketch of the proof. First, we fix some ball $B\in\mathcal{B}$ centered at some point of the boundary $\po\cO$ satisfying \eqref{e1.12} and take smooth functions $\psi_1,\psi_2\in C_c^\infty(B)$ with $0\le \psi_i\le 1$, $i=1,2$, such that $\psi_2(x)=1$ for $x\in \supp\psi_1$. Then, we define $\psi(x,y)=\psi_1(x)\psi_2(y)$.

Next, as in \cite{FL2} (cf. \cite{FL,MPT}), we consider our coordinates $x=(\bar x, x_{d''})\in \mathbb{R}^{d-1}\times \mathbb{R}$ already relabelled so that $\po \cO\cap B=\{ x_{d''}=\lambda(\bar x)\}$ and we take, by approximation, the following test function in \eqref{e3.unique1}
\[
\varphi(x,y)=\zeta_\delta''(x'')\zeta_\eta''(y'')\rho(x-y)\psi(x,y),
\]
where $\zeta_\delta$ and $\zeta_\eta$ are $\cO''$-canonical boundary layer sequences (see Section~\ref{S:3} below), $\rho=\rho_{m,n}$ is given by $\rho_{m,n}=\rho_m(\bar x - \bar y')\rho_n(x_{d''}-y_{d''})$ and $\rho_m$ and $\rho_n$ are sequences of symmetric mollifiers in $\mathbb{R}^{d-1}$ and in $\mathbb{R}$, respectively.

Let us point out that, as in \cite{FL2}, assumption \eqref{e1.15} yields the existence of normal weak traces for the fields $K_{x''}(u,u_{B''})$ and $K_x(v,u_{B''})$, which is key to deal with the terms that involve derivatives of the boundary layer sequences.

Then, after adding and subtracting a few terms and performing a delicate analysis of the resulting inequality, where we use the normal weak traces for the fields $K_{x''}(u,u_{B''})$ and $K_x(v,u_{B''})$ to control the values of the solutions near the boundary, we are able to take the limit as $\delta,\eta\to 0$ first and then as $m,n\to \infty$ to obtain the following
\begin{multline}\label{e3.unique2}
  -\bbE\int_0^T\int_{\cO}\left|u^{\pm}(s,x)-v^{\pm}(s,x)\right|\theta'(s)\psi_1(x)\, dx\, dy \\
  \le  \bbE\int_0^T\int_{\cO}F(u(s,x),v(s,x))\cdot \nabla\psi_1(x)\, dx\, dy\, ds \\
   -\bbE\int_0^T \int_{\cO}\nabla_{x''}\cdot\mathbb{B}(u(s,x),v(s,x))\nabla_{x''}\psi_1(x)\, dx\, dy\, ds.
 \end{multline}
We omit the details as they follow line by line the proof of Theorem~3.3 in \cite{FL2} with slight adaptations.

To conclude, we see that we may take a covering $\{B_j''\}_{j=0}^N$ of $\overline{\cO''}$, where $B_j''\in \mathcal{B}''$ for $1\le j\le N$ and $B_0\subset\subset\cO''$ and a partition of unity $\{\tilde \psi_{j}\}_{j=0}^N$ subordinated, so that we have the inequality \eqref{e3.unique2} with $\psi_1=\tilde \psi_j$, for each $j=1,...,N$. Regarding $\tilde \psi_0$, we see that \eqref{e3.unique2} may also be deduced to hold with $\psi_1=\tilde \psi_0$ much more easily as there is no boundary analysis in this case. Thus, adding the inequalities \eqref{e3.unique2} corresponding to each $\tilde \psi_j$ we obtain \eqref{e3.unique1}. 
\end{proof}

\section{Divergence-Measure Fields and Normal Traces for Kinetic Solutions}\label{S:3}
  
As a preparation for our subsequent discussion about the strong trace property,  in this section we recall some facts in the theory of divergence-measure fields that will be used in this paper. We also comment on how these results can also be used to deduce the existence of normal weak traces for certain fields in this stochastic context.
    
 \begin{definition}\label{D:3.1} Let $U\subset\R^N$ be open. 
For $F\in L^p(U;\R^N)$, $1\le p\le\infty$, or $F\in\M(U;\R^N)$,
set
\begin{equation}\label{e1}
|\div F|(U):=\sup\{\,\int_U\nabla\vphi\cdot F\,\,: \,\,\vphi\in C_0^1(U),
\,\, |\vphi(x)|\le1,\ x\in U\,\}.
\end{equation}
For $1\le p\le\infty$, we say that $F$ is 
an $L^p$-divergence-measure field over $U$, i.e., $F\in\DM^p(U)$,
if $F\in L^p(U;\R^N)$ and
\begin{equation}\label{norm1}
\|F\|_{\DM^p(U)}:=\|F\|_{L^p(U;\R^N)}+|\div F|(U)<\infty.
\end{equation}
We say that $F$ is an extended divergence-measure field over $U$,
i.e., $F\in\DM^{ext}(U)$, if $F\in\M(U;\R^N)$ and
\begin{equation}\label{norm2}
\|F\|_{\DM^{ext}(U)}:=|F|(U)+|\div F|(U)<\infty.
\end{equation}
If $F\in \DM^*(U)$ for any open set $U\Subset\R^N$, 
then we say $F\in \DM^*_{loc}(\R^N)$.
\end{definition}

Here, we will be concerned only with bounded domains $U\subset\R^N$, and fields that are $L^p$ vector functions, so it will suffice to consider divergence-measure fields in 
$\DM^1(U)$.  We recall the Gauss-Green formula for general $\DM^1$-fields, first proved in \cite{CF2,CF3} and extended by Silhavy in \cite{S1}.

\begin{theorem}[Chen \& Frid \cite{CF2,CF3}, Silhav\'y \cite{S1}] \label{T:3}  If $F\in\DM^1(U)$ then there exists a linear functional $F\cdot\nu: \operatorname{Lip}(\po U)\to\R$ such that
\begin{equation}\label{e8}
F\cdot\nu(g|\po U)=\int_U \nabla g\cdot F+\int_U g\,\div F,
\end{equation}
for every $g\in \operatorname{Lip}(\R^N)\cap L^\infty(\R^N)$.  
Moreover,
\begin{equation}\label{e8'}
|F\cdot\nu(h)|\le |F|_{\DM(U)}|h|_{\Lip(\po U)},
\end{equation}
for all $h\in\Lip(\po U)$, 
 where we use the notation 
 $$
 |g|_{\Lip(C)}:= \sup_{x\in C}|g(x)|+\Lip_C(g).
 $$
 
Furthermore,  let $m\,:\, \R^N\to\R$ be a nonnegative Lipschitz function with $\supp m\subset \bar U$ which is strictly positive on $U$, and for each $\ve>0$ let $L_\ve=\{ x\in U\,:\,0<m(x)<\ve\}$. Then:
\begin{enumerate}

\item[(i)] {\rm({\em cf.} \cite{CF2,CF3} and \cite{S1})} If $g\in \Lip(\R^N)\cap L^\infty(\R^N)$, we have
\begin{equation}\label{e11}
   F\cdot\nu(g|\partial U)=-\lim_{\ve\to0}\ve^{-1}\int_{L_\ve}g \,\nabla m\cdot F\,dx;
\end{equation}

\item[(ii)] {\rm({\em cf.} \cite{S1})} If 
\begin{equation}\label{e12}
\liminf_{\ve\to0} \ve^{-1} \int_{L_\ve}|\nabla m\cdot F|\,dx<\infty,
\end{equation}
then $F\cdot\nu$ is a measure over $\partial U$. 
\end{enumerate} 
\end{theorem}
  
A typical example of  such  $m$ is provided by $m(x)=\dist (x,\partial U)$, for $x\in U$, and $m(x)=0$, for $x\in\R^N\setminus U$. Another example of a function $m$ for which \eqref{e11} holds is given by a {\em level set boundary layer sequence}, provided that the domain has a {\em Lipschitz deformable boundary}; concepts whose definitions we recall subsequently.

\begin{definition}\label{D:2.1}  Let $U\subset\R^{N+1}$ be  an open set. We say that  $\po U$ is a {\em   Lipschitz deformable boundary}  if the following hold:

\begin{enumerate} 

\item[(i)] For each $x\in\po U$, there exist $r>0$ and a Lipschitz mapping $\g :\R^{N}\to\R$  such that, upon relabeling, reorienting and translation,    
$$
U\cap Q(x,r) = \{\,y\in\R^{N+1}\,:\, \g(y_1,\cdots,y_{N})<y_0\,\}\cap Q(x,r),
$$
where $Q(x,r)=\{\,y\in\R^{N+1}\,:\, |y_i-x_i|\le r,\ i=1,\cdots,N+1\,\}$. We denote by $\hat \g$ the map $ \hat y\mapsto (\g(\hat y),\hat y)$, $\hat y=(y_1,\cdots,y_{N})$.

\item[(ii)]  There exists a map $\Psi:[0,1]\X\po U\to \bar U$ such that $\Psi$ is a bi-Lipschitz homeomorphism over its image and $\Psi(0,x)=x$, for all $x\in\po U$.
For $s\in[0,1]$, we denote by $\Psi_s$ the mapping from $\po U$ to $\bar U$ given by $\Psi_s(x)= \Psi(s,x)$, and set $\po U_s:=\Psi_s(\po U)$. We call such map a Lipschitz deformation for $\po U$. 

\end{enumerate}
The {\em level set function associated with the deformation} $\Psi$ is the function $h:\bar U\to [0,1]$, defined by $h(x)=s$, if $x\in\po U_s$, $0\le s\le 1$, and $h(x)=1$, if $x\in U\setminus\Psi((0,1)\X\po U)$.  
 \end{definition}
 
 \begin{definition} \label{D:2.2}  Let $U\subset\R^{N+1}$ be an open set with a  Lipschitz deformable boundary $\po U$, and $\Psi:[0,1]\X\po U \to \bar \Om$ a Lipschitz deformation.
 \begin{enumerate}
 \item  The Lipschitz deformation is said to be {\em regular over} $\Gamma\subset\po U$,  if $D\Psi_s\to \operatorname{Id}$, as $s\to0$, in $L^1(\Gamma, \H^{N})$; 
 \item The Lipschitz deformation is said to be {\em strongly regular over} $\Gamma\subset\po U$ if it is regular over $\Gamma$ and  $J[\Psi_s]\to 1$ in $\Lip(\Gamma)$, as $s\to0$, that is,  given any Lipschitz diffeomorphism $\hat \g: \Om\subset\R^{N}\to \Gamma$, we have  $D\Psi_s\circ\hat\g\to D\hat\g$ in $L^1(\Om)$, as $s\to0$, and  $J[\Psi_s\circ\hat\g]/ J[\hat\g]\to 1$ in $\Lip(\Om)$, as $s\to0$ . Here, for a Lipschitz function $\a:\R^k\to\R^m$ we denote by $J[\a]$ the Jacobian of the map $\a$ (see, e.g., \cite{EG}). Observe that we do not need to require more regularity on $\Gamma$, it suffices that $J[\Psi_s]\in\Lip(\Gamma)$.  
 \end{enumerate}
  \end{definition}

The following two results have been proved  in \cite{Fr}; we refer to the latter or  \cite{FL} for the proofs.

\begin{theorem}[{\em cf.} \cite{Fr}]\label{T:2.3} Let $U\subset\R^{N+1}$ be a bounded open set with a deformable Lipschitz boundary and  $F\in\DM^1(U)$. Let $\Psi:\po U\X[0,1]\to\bar U$ be a Lipschitz deformation of $\po U$. Then, for almost all $s\in[0,1]$, and all $\phi\in C_0^\infty(\R^{N+1})$, 
\begin{equation}\label{e2.2'}
\int_{U_s}\phi\,\div F = \int_{\po U_s}\phi(\om) F(\om)\cdot\nu_s(\om)\,d\H^{N}(\om)-\int_{U_s}F(x)\cdot\nabla\phi(x)\,dx,
\end{equation}
where $\nu_s$ is the unit outward normal field defined $\H^{N}$-almost everywhere in $\po U_s$, and $U_s$ is the open subset of $U$ bounded by $\po U_s$.
\end{theorem}

\begin{theorem}[{\em cf.} \cite{Fr}] \label{T:2.3'} Let $F\in\DM^1(U)$, where $U\subset\R^{N+1}$ is a bounded open set with a Lipschitz deformable boundary and Lipschitz   deformation $\Psi:[0,1]\X\po U\to\bar U$.
Denoting by $F\cdot\nu|_{\po U}$ the continuous linear functional $\Lip(\po U)\to\R$ given by the normal trace of $F$ at $\po U$, we have the formula
\begin{equation}\label{e2.3'}
F\cdot\nu|_{\po U}=\esslim_{s\to0} F\circ\Psi_s(\cdot)\cdot\nu_s(\Psi_s(\cdot))J[\Psi_s] ,
\end{equation}
with equality in the sense of $(\Lip(\po U))^*$, where on the right-hand side the  functionals are given by ordinary functions in $L^1(\po U)$. In particular, if $\Psi$ is strongly regular  over $\Gamma\subset\po U$   then, for all $\varphi\in \Lip(\po U)$ with $\supp\varphi\subset\Gamma$, we have
 \begin{equation}\label{e2.3''}
\la F\cdot\nu|_{\po U},\varphi\ra=\esslim_{s\to0} \int_{\Gamma} F\circ\Psi_s(\om)\cdot\nu_s(\Psi_s(\om))\varphi(\om)\,d\H^{N}(\om).
\end{equation}
\end{theorem}

   \medskip

We also recall the following definition of boundary layer sequence ({\em cf.} \cite{MPT, FL}).

\begin{definition}
Let $\cU\subset \mathbb{R}^d$ be a smooth open set. We say that $\zeta_\delta$ is a boundary layer sequence if for each $\delta>0$, $\zeta_\delta\in \Lip(\cU)$, $0\le \zeta_\delta \le 1$, $\zeta_\delta(x) \to 1$ for every $x\in\cU$, as $\delta\to 0$ and $\zeta_\delta=0$ on $\po\cU$.
\end{definition}

Let us also recall  that if $\cO$ has a Lipschitz deformable boundary and given a Lipschitz deformation for $\po \cO$, $\Psi:[0,1]\times \po \CO\to \cO$, the associated level set function $h:\cO\to [0,1]$ is  given by $h(x)=s$ for $x\in \Psi(s,\po\cO)$ and $h(x)=1$ for $x\in \cO\setminus \Psi([0,1]\times \po\cO)$. Then we can also define an associated boundary layer sequence by
\[
\zeta_\delta(x)=\frac{1}{\delta}\min\{ \delta,h(x) \},\qquad 0<\delta<1,
\]
which we call the {\em level set boundary layer sequence} associated with the deformation $\Psi$. In this case, as in \cite{FL}, we note that if $\Psi$ is of class $C^{1,1}$, we have that
\begin{equation}\label{e2.102}
\begin{aligned}
\nabla \zeta_\delta(x) &= -\frac{1}{\delta}\chi_{\{ 0<\zeta_\delta(x) <1 \}}(x)N(x),\\
D^2\zeta_\delta(x)&=-\frac{N(x)\otimes\nu(x)}{\delta}d\mathcal{H}^{d-1}(x)\big|\Psi(\delta,\po\cO)\\
&\qquad -\frac{1}{\delta}\chi_{0<\zeta_\delta(x)<1}\nabla N(x),
\end{aligned}
\end{equation}
where $N(x)=\lambda(x)\nu(x)$, $\nu$ is the outward unit normal vector to $\Psi(\delta,\po\cO)$, $\lambda(x)$ is a positive Lipschitz function and $\mathcal{H}^{d-1}(x)\big|\Psi(\delta,\po\cO)$ denotes the $(d-1)$-dimensional Hausdorff measure restricted to the hyper-surface $\Psi(\delta,\po\cO)$. 
   
\begin{remark}\label{R:2.1}
Let $U$ be endowed with a Lipschitz boundary  and $\Gamma\subset\po U$ be an open  piece of $\po U$. Let $\om\in\Gamma$ and $\CR_\om:\R^{N+1}\to\R^{N+1}$ be a rigid motion in $\R^{N+1}$ such that, for some Lipschitz function  $\g:\R^{N}\to\R$, denoting $y=\CR_\om x$, $\hat y=(y_1,\cdots,y_{N})$,  and defining $\hat\g(\hat y):=(\g(\hat y),\hat y)$,  we have  that $\hat\g(\Om)=\Gamma$, for some open set  $\Om\subset\R^{N}$. Let also $\tilde U=\{y\in\R^{N+1}\,:\, \g(\hat y)<y_0\,\}$ and suppose $U\cap \tilde U\ne\emptyset$. If $F\in\DM^1(U)\cap\DM^1(\tilde U)$, it is immediate to check, using the Gauss-Green formula, that $F\cdot\nu|_{\po U}$ and $F\cdot\nu|_{\po\tilde U}$ coincide over $\Gamma$, that is,
$$
\la F\cdot\nu|_{\po U},\varphi\ra=\la F\cdot\nu|_{\po\tilde U},\varphi\ra,
$$
for all $\varphi\in \Lip_c(\Gamma)$, where $\Lip_c(\Gamma)$ denotes the subspace of  functions in $\Lip(\Gamma)$ with compact support in $\Gamma$.  Also recall that, to define 
$F\cdot\nu|_{\po U}$,    it is not necessary that $\po U$ be Lipschitz deformable.  In such cases, restricted to functions $\varphi\in \Lip_c(\Gamma)$ we will always view $\la F\cdot \nu|_{\po U},\varphi\ra$ as obtained, after translation, relabeling and reorienting coordinates, 
 through \eqref{e2.3''} by using the {\em canonical  deformation} of $\po \tilde U$ , defined as  $(s, (\gamma(\hat y),\hat y))\mapsto (\gamma(\hat y)+s, \hat y)$, evidently strongly regular over $\Gamma$, which is legitimate for $\tilde U$; we call that a {\em local canonical deformation of $\po U$}.

\end{remark}

\begin{remark}\label{R:2.2} Since the normal trace $F\cdot\nu|_{\po U}$, of a divergence measure field over an open set $U$, restricted to some open piece $\Gamma\subset \po U$, 
does not depend on $U$, but just on $\Gamma$, for the purpose of defining $F\cdot\nu$ over $\Gamma$, we may refer to a deformation  $\Psi: [0,1]\X\Gamma\to\bar U$, defined just on $[0,1]\X\Gamma$, which may be the restriction over $[0,1]\X\Gamma$ of a deformation $\tilde\Psi:[0,1]\X\po\tilde U\to\overline{\tilde U}$ such that $\Gamma\subset\po\tilde U$,
$U\cap\tilde U\ne\emptyset$ and $F$ may be extended somehow from $U\cap\tilde U$ to $\tilde U$ so as to be viewed as  a divergence-measure field over $\tilde U$.  
\end{remark}


 
Coming back to the main subject of the paper, we note that Theorem~\ref{T:3} can be used to deal with the boundary values of kinetic solutions of equations \eqref{e1.1}-\eqref{e1.4}, which is essential to guarantee their uniqueness. Indeed, this has ben successfully done in \cite{FL2} where the authors deal with the Dirichlet problem for a quasilinear degenerate parabolic stochastic partial differential equation on a bounded domain with multiplicative noise. The idea is to combine Theorem~\ref{T:3} with \eqref{e1.15} in order to guarantee the existence of the normal traces on $\cO\times\po\cO''$ for the field $K_{x''}(u(t,x),u_{B''}(t,x))$, where $u$ is a kinetic solution of equations \eqref{e1.1}-\eqref{e1.4}.  We refer to \cite{FL2} for the details.
  
\section{Strong trace property}\label{S:4}

In this section we establish  the strong trace property for stochastic parabolic-hyperbolic equations, which is a decisive point in the proof of the uniqueness of kinetic solutions to the problem \eqref{e1.1}-\eqref{e1.4}. Recall that $\Gamma':=\po\cO'\X\cO''$ and $\Gamma_T':=\Gamma'\X(0,T)$.  

 \begin{theorem}\label{T:5.1}  Let $u\in L^\infty(\Om\X(0,T)\X \cO)$ and for all $p\ge 1$, 
$$
u\in L^p(\Om\X(0,T), \cP, d\bbP\otimes dt  ;  L^p(\cO)). 
$$  
Assume that $u$ satisfies  $\nabla_{x''}b(u)\in L^2(\Om\X(0,T)\X\cO)$   and (ii') of Definition~\ref{D:1.3}. 
Then, there exists $u^\tau\in L^\infty(\Om\X\Gamma_T')$  such that, for any  deformation of $\po\cO'$,  $\Psi': [0,1]\X \po\cO'\to \bar\cO'$, strongly regular over $\po\cO'$,  if  $\Psi: [0,1]\X\Gamma'_T \to \bar \cO_T$, is defined by $\Psi(s,t,x',x'')=(t,\Psi'(s,x'),x'')$, we have 
\begin{equation}\label{eT5.1-1}
\esslim_{s\to0} \bbE \int_0^T\int_{\Gamma'} |u(\Psi(s,t,\wh x))-u^\tau(t,\wh x)|\, d \H^{d-1}(\wh x)\,dt=0.
\end{equation}
Moreover, we also have
\begin{equation}\label{eT5.1-2}  
\esslim_{s\to0}  \int_0^T\int_{\Gamma'} |u(\Psi(s,t,\wh x))-u^\tau(t,\wh x)|\, d \H^{d-1}(\wh x)\,dt=0,
\end{equation}
almost surely. 
\end{theorem} 
  
\begin{proof}   By \eqref{e1.8E-1}, given $\eta\in C^2(\R)$,  we have, in the sense of distributions on $(0,T)\X\cO$,
\begin{multline}\label{est.1}
\po_t\left(\eta(u) -\int_0^t\sum_{k\ge1} g_k(x,u)\frac{d\eta}{du}(u)\,d\beta_k(s)\right)+\div \Abf_\eta (u)-D_{x''}^2: \Bbf_\eta(u)\\
= -\int_{\R}\frac{d^2\eta}{d\xi^2}(\xi)\,dm(t,x,\xi) +\frac12 G^2(\cdot,u)\frac{d^2\eta}{du^2}(u).
\end{multline}
Defining
 $$
 \ff(t,x,\xi)=\chi(\xi; u(x,t)),\quad \text{where}\quad \chi(\xi;u):=\begin{cases} -1, \quad u\le \xi< 0,\\ 1,\quad  0<\xi \le u,\\ 0,\quad |\xi|>|u|,\end{cases}
 $$
 assuming that $\eta(0)=0$,  $\Abf_\eta(0)=0$, $\Bbf_\eta(0)=0$, using 
  \begin{gather*}
 \eta(u)=\int_{\R}\frac{d\eta}{d\xi}(\xi)\chi(\xi;u)\,d\xi,\quad \Abf_{\eta}(u)=\int_{\R}\frac{d\eta}{d\xi}(\xi)\abf(\xi)\chi(\xi;u)\,d\xi,\\
  \Bbf_{\eta}(u)=\int_{\R}\frac{d\eta}{d\xi}(\xi)\bbf (\xi)\chi(\xi;u)\,d\xi,\\
  \sum_{k\ge1} g_k(\cdot,u)\frac{d\eta}{du}(u)\,\dot\beta_k(t) = \int_{\R} \sum_{k\ge1} g_k(\cdot,\xi) \frac{d\eta}{d\xi}(\xi)\dot\beta_k(t)\,\d_u(\xi),\\
  \frac12 G^2(\cdot,u)\frac{d\eta}{du}(u)=\int_{\R} G^2(\cdot,\xi)\frac{d\eta}{d\xi}(\xi)\,\d_u(\xi),
   \end{gather*}
where $\d_u(\xi)$ is the Dirac measure concentrated at $u$, from \eqref{est.1} we get the following kinetic equation, valid for a.e.\ $\om\in\Om$, in the sense of the distributions on $(0,T)\X\cO\X\R$,
\begin{multline}\label{est.2}
\po_t \ff + \abf(\xi)\cdot \nabla_x \ff + \bbf(\xi):D_{x''}^2 \ff \\= \left(m-\frac12G^2(x,\xi)\,\d_{u(t,x)}(\xi)\right)_\xi 
+\sum_{k=1}^\infty g_k(x,\xi) \dot\beta_k(t)\, \d_{u(t,x)}(\xi)\\
=q_\xi-\sum_{k=1}^\infty g_k(x,\xi)(\po_\xi\ff) \dot\beta_k(t) +\sum_{k=1}^\infty\d_0(\xi) g_k(x,\xi)\dot\beta_k(t),
\end{multline}
where $q=m-\frac12G^2(x,\xi)\,\d_{u(t,x)}(\xi)$. In particular, $|q|\le m+\frac12G^2(x,\xi)\,\d_{u(t,x)}(\xi)$, where $|q|$ denotes the total variation measure associated with the measure $q$. 
Integrating \eqref{est.2} with respect to $\xi$ we may obtain an explicit formula for $m$ which show that, if $L>\|u\|_{L^\infty(\Om\X(0,T)\X\cO)}$, then $m$ is a.s.\ supported  on $(0,T)\X\cO\X(-L,L)$.
Besides, by using the explicit formula for $m$ obtained from \eqref{est.2} and applying it to suitable test functions we get, in a standard way, the following fact, whose detailed proof we omit.   

\begin{lemma}\label{L:3.01} For any $\cV''$ such that $\overline\cV''\subset\cO''$  and  $1\le p <\infty$, it holds
\begin{equation}\label{e03.5}
	\bbE\Vert m \Vert_{\M}^p = \bbE m\big( (0,T)\X \cO'\X\cV''\X[-L,L] \big)^p \leq C(p).	
	\end{equation}
\end{lemma}

We observe that if we can obtain a strong trace $u_{\cV''}^\tau$ on $(0,T)\X\po\cO'\X\cV''$, satisfying \eqref{eT5.1-1}  and \eqref{eT5.1-2} with $\po\cO'\X\cV''$ instead of $\Gamma'$, for any $\cV''$ with $\overline\cV''\subset\cO''$, 
then they must be consistent in the sense that, if $\cV_1''\subset\cV_2''$, then $u_{\cV_2''}^\tau=u_{\cV_1''}^\tau$, a.e.\ on $(0,T)\X\po\cO'\X\cV_1''$, for a.e.\ $\om\in\Om$. Therefore, it suffices to prove the result for an arbitrary $\cV''$, with 
$\overline\cV''\subset\cO''$. Hence, for simplicity,  we may assume that \eqref{e03.5} holds with $\cO''$ instead of $\cV''$, which we will do henceforth.

\bigskip
\centerline{\em \#1. The existence of weak traces.}
\bigskip

Since the boundary $\po\cO'$ is locally the graph of a Lipschitz function, we may fix an open subset of $\po\cO'$, $\S_0'$, which is the graph of a Lipschitz function. More specifically, we fix a neighborhood $\cO_0'$, of a point in $\po\cO'$, so that,
after translating, relabeling and possibly changing orientation (i.e., $x_i\mapsto -x_i$), $\cO'_0$, may be expressed by
\begin{equation}\label{est.3}
\cO'_0:=\{ x'=(\wh x',x_{d'})\in (-r,r)^{d'-1}\X(-r,r)\,:\, x_{d'}> \g_0(\wh x')\},
\end{equation}
where $r>0$ and $\g_0:(-r,r)^{d'-1}\to \R$ is a Lipschitz function satisfying $-r<\g_0(\wh x')<r$, everywhere. Hence, the part of the boundary of $\cO'$ that we are interested in is
\begin{equation}\label{eS0'}
\S_0'=\{x'=(\wh x',x_{d'})\in(-r,r)^{d'-1}\X(-r,r)\:\, x_{d'}=\g_0(\wh x')\}.
\end{equation}
So, let us define
$$
\cO_0= \cO'_0\X\cO'',
$$
and
$$
\Gamma_0=\S_0'\X\cO''.
$$  
 For any $\S_0'$-strongly regular Lipschitz deformation $\psi:[0,1]\X\S_0'\to \overline{\cO'}$, we may write
 \begin{equation}\label{eS0'2}
 \ff_\psi(t,\wh x',s, x'')=\ff(t,\psi(s,\wh x'), x''), \quad\text{for every $\wh x'\in (-r,r)^{d'-1}$}.
 \end{equation}
 Let us set
 $$
 \begin{cases}
 Q:=(0,T)\X(-r,r)^{d'}\X\cO'' ,\\
 \Sigma:=(0,T)\X(-r,r)^{d'-1}\X\cO''.
 \end{cases}
 $$
 Though in order to have $\cO_0'$ written in the form \eqref{est.3} we had to perform affine transformations such as translating, relabelling, and possibly a change  of orientation, the function $\ff$ in the new
 coordinates will satisfy an equation similar to \eqref{est.2}, only with ${\bf a}'(\xi)$ replaced with $\cR {\bf a}'(\xi)$, where, with some abuse of notation,  we write   $\bfa(\xi)=(\bfa'(\xi),\bfa''(\xi))$, 
 $\bfa':=\pi_{d'}\abf=(a_1,\cdots,a_{d'})$,  $\bfa'':=\pi_{d''}\abf=(a_{d'+1},\cdots,a_d)$, and $\pi_{d'},\pi_{d''}$ are the projections of $\R^d$ on $\R^{d'}$ and $\R^{d''}$, respectively, when viewing $\R^d$ as $\R^{d'}\X\R^{d''}$.   Also, we represent as $\cR$ the linear part of the affine transformation necessary to write $\cO'_0$ in the form \eqref{est.3}. For simplicity, we assume that $\ff$ in the new coordinates still satisfies \eqref{est.2}. 
 
The following result is the first step in the  scheme of the proof of the strong trace property in \cite{Va} and it is also the first step in this proof concerning the same property for the stochastic parabolic-hyperbolic equation \eqref{e1.1}. The proof is similar to the one for stochastic conservation laws given in \cite{FL1}. 

\begin{lemma}[Existence of the weak traces] \label{L:3.02}
	There exists a unique function $\ff^\tau \in L^\infty(\Omega\X\Sigma\X(-L,L))$ such that, for any $\Gamma_0$-strongly regular Lipschitz deformation, for a.e.\ $\om\in\Om$, we have
	\begin{equation}
	\esslim_{s \rightarrow 0} \ff_\psi(\om, \cdot, s, \cdot) = \ff^\tau(\om,\cdot,\cdot) \text{ in the weak-$\star$ topology in } L^\infty(\Sigma\X(-L,L)). \label{e03.7'}
	\end{equation}
	
Moreover,
\begin{equation}\label{e03.7''}
\operatornamewithlimits{ess\,lim}_{s \to 0} \ff_\psi(\cdot, s, \cdot) = \ff^\tau(\cdot,\cdot) \text{ in the weak-$\star$ topology of $L^\infty(\Omega\X\Sigma\X(-L,L))$}
\end{equation}
\end{lemma}

\begin{proof}  {\em Step 1:} Let $(h_n)_{n\in\N} \subset C_c^1((-L,L))$ be dense in $L^1((-L,L))$. We consider representatives in $L^\infty(\Om\X(0,T)\X\cO)$ for all functions of the forms 
$$
\int_{-L}^L h_n(\xi) \ff(t,x,\xi)\,d\xi,\quad \int_{-L}^L h_n(\xi)  \ff(t,x,\xi)\abf(\xi)\,d\xi,\quad   h_n'(u(t,x))G^2(x,u(t,x)),
$$
and in $L^2(\Om\X(0,T)\X\cO)$ for functions of the form
$$
  \nabla_{x''}\cdot \-\int_{-L}^L h_n(\xi)\ff(t,x,\xi)\bbf(\xi) \,d\xi,\quad         \sum_{k=1}^\infty  \int_0^t h_n(u(t,x))g_k(x,u(s,x)) \,d\b_k(s) .
$$
Observe that  the $x''$-divergence of the matrix function in the first of the two expressions above actually belongs to $L^2(\Om\X(0,T)\X\cO)$  because of the hypothesis that $\nabla_{x''} b(u)\in L^2( \Om\X(0,T)\X\cO)$, \eqref{e1.4'}, and the fact that  the function resulting 
from the corresponding integral is $\bbf_{h_n}(u(t,x)):=\int_0^{u(t,x)} h_n(\xi)\bbf(\xi)\,d\xi$, and 
$\bbf_{h_n}(u)$ is a Lipschitz function of $b(u)$ due to \eqref{e1.4'}.   
 Let $\Om_0\subset\Om$ be a subset of total measure such that for all $\om\in\Om_0$, the corresponding paths of these functions, viewed as Banach 
 space-valued stochastic processes,  are well defined functions 
 in $L^\infty((0,T)\X\cO_0)$ and $L^2((0,T)\X\cO_0)$, respectively.  We also assume that for all $\om\in\Om_0$ there exists $C(\om)>0$ such that 
 \begin{align*}
m \big((0,T)\X\cO \X [-L,L]\big) \leq C(\om), 
\end{align*}
which is possible by \eqref{e03.5}.  So, let us fix for the moment  $\om\in\Om_0$. 

Let us consider the vector fields $F_{n}$ given by
\begin{multline}\label{e03.6}
F_{n}(t,x) = \Bigl( \int_{-L}^L h_n(\xi) \ff(t,x,\xi) \, d\xi - \sum_{k=1}^\infty \int_0^t h_n(u(s,x)) g_k(x, u(s,x)) \, d\beta_k(s),\Bigr.  \\
 \int_{-L}^L h_n(\xi) \ff(t,x,\xi) \abf'(\xi) \, d\xi,  \\ \Bigl. \int_{-L}^L h_n(\xi) \ff(t,x,\xi) \abf''(\xi) \, d\xi - \nabla_{x''}\cdot \-\int_{-L}^L h_n(\xi)\ff(t,x,\xi)\bbf(\xi) \,d\xi   \Bigr) . 
\end{multline}
We see that  $F_{n} \in L^2((0,T)\X\cO_0)\X  L^\infty( (0,T)\X\cO_0; \R^{d'}) \X  L^2 ((0,T)\X\cO_0; \R^{d''})$. Moreover, from \eqref{est.2},  it is not hard to check that
$$
\operatorname{div}_{t,x} F_{n} = -  \int_{-L}^L h_n'(\xi) \qq(t,x,\xi)\,d \xi  \in \M((0,T)\X\cO_0).
$$
Now, $\Gamma_0=\S_0'\X\cO''$ and $\S_0'$ is a strongly regular deformable Lipschitz boundary. For  any strongly regular deformation 
$\psi:[0,1]\X\po\cO'\to \overline\cO'$, let $\Psi:[0,1]\X(0,T)\X\Gamma_0\to \overline\cO$ be given by $\Psi(s,t,x',x'')=(t,\psi(s,x'),x'')$. We have that, for almost all $s\in[0,1]$, $F_n\circ\Psi(s,\cdot) \cdot \nu_s=F_n^{1'}\circ\Psi(s,\cdot)\cdot\nu_s\in L^\infty((0,T)\X\Gamma_0)$, where $\nu_s$ is the outward unit normal to $\psi(s,\po\cO')$ and $F_n^{1'}$ is the component of $F_n$ corresponding to the space of the $x'$-coordinates. Therefore,  using the parametrization of $\S_0'$ given by \eqref{eS0'2}, Theorem~\ref{T:2.3'} implies that there exists a set $\S_{n} \subset [0,1]$ of total measure and some $F_{n}^{1',b}\cdot \nu \in L^\infty((0,T)\X(-r,r)^{d-1})$, which does not depend on $\psi$, such that
\begin{align}
F_{n}^{1'}(\cdot,\psi(s,\cdot), x'')\cdot\nu_s(\cdot) \stackrel{\star}{\rightharpoonup} F_{n}^{1',b}\cdot\nu &\text{ $\star$-weakly in $L^\infty(\Sigma)$} \nonumber \\ &\text{ as $s \rightarrow 0$ along $s \in \S_{n}$}. \label{e03.7}
\end{align}
Write $\S = \cap_{n=1}^\infty \S_{n}$ so that $\S$ also has total measure in $[0,1]$. From this point the proof follows the lines of the proof of the corresponding result in \cite{FL1} and we include these lines for the sake of convenience of the reader.  

Let us now check that $F_{n}^{1',b}\cdot\nu$ depends linearly on  $h_n$. For any integer $M \geq 1$ and $\varphi_p \in L^1(\Sigma)$, $1 \leq p \leq M$, 
the relations  \eqref{e03.6} and \eqref{e03.7}, the latter taking $s \in \S$,  say that
\begin{align*}
\Big|\int_\Sigma&  \sum_{n,p=1}^M  (F_{n}^{1',b} . \nu)(t,\widehat{x}) \varphi_p(t,\widehat x) \,dt d\widehat{x}    \Big| \\ &\leq \Vert \abf' \Vert_{L^\infty(-L,L)} \int_\Sigma \int_{-L}^L \Bigg|\sum_{n,p=1}^M    h_n(\xi) \varphi_p(t,x) \Bigg| \, d\xi d\widehat{x} dt \\
						 &= \text{(const.)} \Bigg\Vert \sum_{n,p=1}^M  h_n \otimes \varphi_p \Bigg\Vert_{L^1(\Sigma\X(-L,L))}.
\end{align*}
Thus, as $(L^1)^* = L^\infty$, there exists some $H \cdot\nu \in L^\infty(\Sigma\X(-L,L))$ such that, for  all $h \in L^1(-L,L)$ and all $\varphi \in C_c^\infty(\Sigma),$
\begin{align}
 \int_\Sigma \int_{-L}^L h(\xi) &\varphi(t,\widehat{x}) \ff(t,\widehat{x},\psi(\widehat{x},s)) \abf(\xi)\cdot \nu_s(\widehat{x}) \, d\xi d\widehat{x}\, dt \nonumber \\
&\rightarrow \int_\Sigma \int_{-L}^L h(\xi) \varphi(t,x) (H\cdot\nu)(t,\widehat{x}) \, d\xi d\widehat{x}\, dt  \label{e03.8}
\end{align}
as $s \rightarrow 0$ along $s \in \S$. Notice that $H\cdot\nu$ does not depend on $\psi$.

{\em Step 2:} To conclude, let us observe that, since $\Vert \ff_\psi(\cdot,s,\cdot) \Vert_{L^\infty} \leq 1$, the Banach--Alaoglu theorem asserts that, for every strongly regular Lipschitz deformation $\psi$ and every sequence $s_n$ in $\S$ converging to $0$, there exists a subsequence $s_{n_k}$ and some $\ff_\psi^\tau \in L^\infty(\Sigma\X(-L,L))$  such that
$$
\ff_\psi(\cdot,s_{n_k},\cdot) \stackrel{\star}{\rightharpoonup} \ff_\psi^\tau \text{ $\star$-weakly in $L^\infty(\Sigma\X(-L,L))$ as $k \rightarrow \infty$}.
$$
Thus, by \eqref{e03.8}, we deduce from the fact that $\nu_s \rightarrow \nu$ in $L^1((-r,r)^{d-1};\R^d)$ that
\begin{align*}
 \int_\Sigma \int_{-L}^L h(\xi) &\varphi(t,\widehat{x}) \ff_\psi^\tau (t,\widehat{x}) \abf(\xi)\cdot \nu(\widehat{x}) \, d\xi\, d\widehat{x}\, dt    \nonumber \\
&=  \int_\Sigma \int_{-L}^L h(\xi) \varphi(t,x) (H\cdot\nu)(t,\widehat{x}) \, d\xi\, d\widehat{x}\, dt, 
\end{align*}
for every  $h \in L^1(-L,L)$ and $\varphi \in C_c^\infty(\Sigma)$. Since the right-hand term is independent of $\psi$ and $s_n$, so is $\abf(\xi)\cdot\nu(\widehat{x}) \ff_\psi^\tau(t,\widehat{x},\xi)$. On the other hand, remembering that $\abf(\xi)\cdot \nu=\abf'(\xi)\cdot\nu'$,
if $\nu$ is the outward normal to $\Gamma'$,  because the nondegeneracy condition implies that
$$
\LL^1\big\{ \xi \in (-L,L)\,:\, \abf(\xi) \cdot \nu(\widehat{x}) = 0 \big\} = 0,
$$
we conclude that $\ff_\psi^\tau$ also does not depend on $\psi$ and $s_n$, hence we may denote it by $\ff^\tau$. 


{\em Step 3:} Arguing as before, but considering now the vector fields
	\begin{multline*}
	F_{m,n}(t,x) = \bbE \Big[ X_m \, \Big( \int_{-L}^L h_n(\xi) \ff(t,x,\xi) \, d\xi - \sum_{k=1}^\infty \int_0^t h_n(u(t,x)) g_k(x, u(s,x)) \, d\beta_k(s) , 
	 \\ \int_{-L}^L h_n(\xi) \ff(t,x,\xi) \abf'(\xi) \, d\xi  ,\\
	\Bigl.\Bigl. \int_{-L}^L h_n(\xi) \ff(t,x,\xi) \abf''(\xi) \, d\xi - \nabla_{x''}\cdot \int_{-L}^L h_n(\xi)\ff(t,x,\xi)\bbf(\xi) \,d\xi   \Bigr) \Bigr],
	\end{multline*}
	where $(X_m)_{m \in \N}$ is a sequence in $L^\infty(\Omega)$ that is dense in $L^1(\Omega)$ (notice that we can always suppose that $\Omega$ is countably generated), we can deduce the existence of some $\ff^b \in L^\infty(\Omega \X \Sigma \X (-L,L))$ such that 
	$$\operatornamewithlimits{ess\,lim}_{s \to 0} \ff_\psi(\cdot, s, \cdot) = \ff^b \text{ in the $\star$-weak topology of $L^\infty(\Omega\X\Sigma\X(-L,L))$.}$$
	
{\em Step 4:} It remains to show that $\ff^b(\omega, \cdot, \cdot) = \ff^\tau(\om, \cdot, \cdot, \cdot)$ for almost all $\om \in \Omega$ in the $L^1$-sense. This, however, can be seen from the fact that both are the weak-$\star$ limit of $\frac{1}{s} \int_0^s \ff_\psi(\cdot, \s, \cdot)\, d\s$ in $L^\infty(\Om \X \Sigma \X (-L,L))$ as $s \to 0$. Observe that this also shows that $\ff^\tau$ is measurable and $\ff^\tau\in L^\infty(\Om\X\Sigma\X(-L,L))$.

\end{proof}

Next, we need to convert the weak-$\star$ convergence in $L^\infty(\Sigma\X(-L,L))$ in the statement of the previous lemma to a strong convergence in $L^1(\Sigma\X(-L,L))$. 
For that we recall the following criterion from \cite{Va}, to which we refer for the proof. We first recall that, for a measure space $S$,  a function $z\in L^\infty(S\X\R)$ is a $\chi$-function if for almost all
$x\in S$, there exists $a(x)\in \R$ such that 
$$
z(x,\xi)= 1_{(-\infty, a(x))}(\xi)-1_{(-\infty,0)}(\xi).
$$ 
Note that, in this case, $a(x)=\int_{-\infty}^\infty z(x,\xi)d\xi$. Furthermore, $\| a\|_{L^\infty(S)}\leq L$ if and only if the corresponding $\chi$-function satisfies $z(x,\xi)=0$, for a.e. $x\in S$ and $|\xi|\geq L$. In this case, we may simply consider $z$ as an element of $L^\infty(S\times (-L,L))$.

\begin{lemma}\label{2.criterion0}
Let $S$ be a finite measure space and let $g_n\in L^\infty(S\times (-L,L))$ be a sequence of $\chi$-functions converging weakly to $g\in L^\infty(S\times (-L,L))$. Define $v_n(\cdot)=\int_{-L}^Lg_n(\cdot,\xi)d\xi$ and $v=\int_{-L}^Lg(\cdot,\xi)d\xi$. Then, the three following propositions are equivalent:
\begin{itemize}
\item[(i)] $g_n$ converges strongly to $g$ in $L^1(S\times \R)$,
\item[(ii)] $v_n$ converges strongly to $v$ in $L^1(S)$,
\item[(iii)] $g$ is a $\chi$-function.
\end{itemize}
\end{lemma}

As a direct consequence of Lemma~\ref{2.criterion0} we have the following.

\begin{lemma} \label{2.criterion}
	For  every  regular Lipschitz deformation $\psi$,
	$$
	\esslim_{s \rightarrow 0} \ff_\psi(\cdot,\cdot, s,\cdot) = \ff^\tau(\cdot,\cdot,\cdot) \text{ strongly in $L^1(\Om\X\Sigma\X(-L,L))$},
	$$
	if, and only if, $\ff^\tau(\cdot,\cdot,\cdot)$ is a $\chi$-function a.e.\ in $\Om\X\Sigma\X(-L,L)$.  
\end{lemma}

We now pass to the verification that $\ff^\tau$ is indeed a $\chi$-function. 

\bigskip

\centerline{\em \# 2. The blow-up procedure.}

\bigskip

Let us keep $\cO_0$ fixed. Since $\ff^\tau$ is independent on the $\S_0'$-strongly regular Lipschitz deformation, we may choose the special deformation $\psi(s,\widehat{x'}) = (\widehat{x'}, \gamma(\widehat{x'}) + s)$, which is trivially strongly regular. 
Identifying $y_{d'} = s$ 
and $\widehat{y'} = \widehat{x'}$,  $y''=x''$,  define
$$
\widetilde{\ff}(t,y,\xi) = \ff_{\psi}(t,\widehat{y'}, y_{d'},y'', \xi) = \ff(t, \psi(y_{d'},\widehat{y'}),y'', \xi).
$$
For simplicity, let us write $\psi(y')=\psi(y_{d'},\wh{y'})$. Notice that there exists an $r_0>0$ such that $\psi(y') \in \cO_0$ provided that $(\widehat{y'}, y_{d'}) \in (-r,r)^{d'-1} \X (0, r_0)$. As a result, we see from \eqref{est.2} that $\widetilde{\ff}$ is a solution to
\begin{multline}\label{3.eqftil}
\frac{\partial \widetilde \ff}{\partial t} + \widehat{\abf'}(\xi) \cdot \nabla_{\widehat{y'}} \wt \ff + 
\widetilde{a}_{d'}(\widehat{y'},\xi) \frac{\partial \widetilde{\ff}}{\partial y_{d'}} +
\bbf(\xi):\nabla_{y''}^2\wt \ff= \\
=-\abf''(\xi)\cdot\nabla_{y''}\wt\ff+ \qq_\xi-\sum_{k=1}^\infty g_k(x,\xi)(\po_\xi\ff) \dot\beta_k(t) +\sum_{k=1}^\infty\d_0(\xi) g_k(x,\xi)\dot\beta_k(t)
\\ \quad\text{a.s.\ in the sense of the distributions in $(0,T)\X(-r,r)^{d'-1}\X(0,r_0)\X(-r,r)^{d''}$}.
\end{multline}
In the equation above, we have denoted $\abf(\xi) = (\widehat \abf'(\xi), a_{d'}(\xi),\abf''(\xi))$ and
\begin{equation}
\widetilde{a}_{d'}(\widehat{y'}, \xi) = a_{d'}(\xi) - \nabla \gamma_0(\widehat{y'})\cdot\widehat{\abf'}(\xi) = \lambda(\widehat{y'})\abf'(\xi) \cdot \nu'(\widehat{y'}), \label{3.eqa}
\end{equation}
for some $\lambda(\widehat{y'}) < 0$, and $\nu'$ is the outward unit normal to $\po\cO'$, due to \eqref{est.3}. 

Moreover, we have also written $\widetilde{\qq} = \widetilde{m} - \frac{1}{2} \widetilde{G}^2 \delta_{\xi = \widetilde u(t,y)}$, where
\begin{equation*}
\begin{aligned}
&\widetilde u (t,y) = u(t,\psi(y'),y'') = \int_{-L}^L \widetilde{\ff}(t,y,\xi)\, d\xi, \\
&\widetilde{m}(t,y,\xi) = m(t, \psi(y'),y'', \xi), \\
&\widetilde{g}_k(y, z)  = g_k(\psi(y'),y'', z) \text{ for all $k \geq 1$}, \text{ and}\\
&\widetilde{G}^2(y,z)   = \sum_{k=1}^\infty \widetilde{g}_k(y,z)^2.
\end{aligned}
\end{equation*}

\bigskip
Before we rescale $\widetilde{\ff}$, let us recall some lemmas in \cite{Va} slightly adapted to our setting, to which we refer for the proofs. Let $\Om_0$ be as in the proof of Lemma~\ref{L:3.02} and 
$$
S(\ve):\R\X\R^{d'}\X(0,\infty)\X\R^{d''}\to \R\X\R^{d'}\X(0,\infty)\X\R^{d''}
$$
be defined by 
$$
S(\ve)(\ul t, \ul y_1,\cdots,\ul y_{d'}, \ul y_{d'+1},\cdots,\ul y_d)
=(\ve \ul t, \ve \ul y_1,\cdots,\ve \ul y_{d'},\sqrt{\ve}\ul y_{d'+1},\cdots,\sqrt{\ve} \ul y_d),
$$
and let us denote $S_{(t_*, y_*)}(\ve)(\ul t,\ul y):=(t_*,y_*)+S(\ve)(\ul t,\ul y)$. 
 
\begin{lemma} \label{3.lemma1}
        There exists  a sequence    $0<\ve_n\to 0$   and a set of total measure $\E \subset \Sigma$  such that, 
	for every $(t_*,\wh{y}_*):=(t_*,\widehat{y'}_*, y_*'')\in \E$, 
	 and every $R>0$, denoting $y_*:=(\wh y'_*,0,y_*'')$, 
	\begin{align}
	&\lim\limits_{n \rightarrow \infty} \bbE \frac{1}{\ve_n^{d'+\frac{d''}2}} \widetilde{m} \Big(   
	S_{(t_*,y_*)}(\ve_n)\left[(-R, R)^{d'} \X (0, R)\X(-R,R)^{d''}\right] \X [-L,L] \Big) = 0,  \label{3.limm}\\
	&\lim\limits_{n \rightarrow \infty} \bbE \frac{1}{\ve_n^{d'+\frac{d''}2}} 
	\iint_{S_{(t_*,y_*)}(\ve_n)\left[(-R, R)^{d'} \X (0, R)\X(-R,R)^{d''}\right] } \frac{1}{2} \widetilde{G^2}(y,\widetilde{u}(t,y))\, dy\, dt = 0. \label{3.limG}
	\end{align}
	Consequently, for every  $(t_*,\widehat{y}_*) \in \E$  and every $R>0$,
	$$
	\lim\limits_{n \rightarrow \infty} \bbE \frac{1}{\ve_n^{d'+\frac{d''}2}} |\widetilde{\qq} | \Big(  
	S_{(t_*,y_*)}(\ve_n)\left[(-R, R)^{d'}\X(0,R)\X(-R,R)^{d''}\right]\X [-L,L] \Big) = 0,
	$$
	where, as usual, $|\tilde\qq|(A)$ denotes the total variation of $\tilde\qq$ on the set $A$.  
	
	Therefore,  given $(t_*,\wh y_*)\in\E$, there exists a subsequence of $\ve_n$, still denoted $\ve_n=\ve_n(t_*,\wh y_*)$, and
	a subset of total measure $\Om_1=\Om_1(t_*, \wh{y}_*)\subset\Om_0$, such that, for all $\om\in\Om_1$,
	  \begin{align}
	&\lim\limits_{n \rightarrow \infty}  \frac{1}{\ve_n^{d'+\frac{d''}2}} 
	\widetilde{m} \Big( S_{(t_*,y_*)}(\ve_n)\left[(-R, R)^{d'} \X (0, R)\X(-R,R)^{d''}\right]\X [-L,L] \Big) = 0, \label{3.limm'}\\
	&\lim\limits_{n \rightarrow \infty}  \frac{1}{\ve_n^{d'+\frac{d''}2}} 
	\iint_{ S_{(t_*,y_*)}(\ve_n)\left[(-R, R)^{d}\X (0, R)\X(-R,R)^{d''}\right]} \frac{1}{2} \widetilde{G^2}(y,\widetilde{u}(t,y))
	\, dy\, dt = 0. \label{3.limG'}\\
	&\lim\limits_{n \rightarrow \infty}  \frac{1}{\ve_n^{d'+\frac{d''}2}} |\widetilde{\qq} 
	| \Big( S_{(t_*,y_*)}(\ve_n)\left[(-R, R)^{d'}\X (0, R)\X(-R,R)^{d''}\right]\X [-L,L] \Big) = 0.\label{3.qq}
	\end{align}
\end{lemma}

\begin{lemma} \label{3.lemma2} 
	There exists a subsequence of $\ve_n$, still denoted by $\ve_n$, and  a subset  of  $\E\subset \Sigma$,  also of total measure and still denoted by $\E$,  such that, for every $(t_*, \widehat{y}_*) \in \E$,
	 every $R > 0$, and every $1 \leq p < \infty$,
	\begin{align}
	&  \int_{(-R,R)^{d-1}} \int_{-L}^L |\wt a_{d'}(\widehat{y'}_*, \xi) - \wt a_{d'}(\widehat{y'}_* + \ve_n\widehat{\ul y'}, \xi) |^p \, d\xi \, d\widehat{\ul y'}  \to 0 \label{3.lima}  \\
	& \bbE \int_{-R}^R \int_{(-R,R)^{d-1}} \int_{-L}^L |\ff^\tau(\wh S_{(t_*,\wh{y}_*)}(\ve_n)(\ul t,\wh{\ul y'},\ul y''), \xi) - \ff^\tau(t_*,\widehat{y}_*, \xi) | \, d\xi \,d\widehat{\ul y}\, d\ul t \to 0,\label{3.limfb}
	\end{align}
	as $n \to \infty$, where $\wh S_{(t_*,\wh{y}_*)}(\ve_n)(\ul t,\wh{\ul y'},\ul y'')=(t_*,\wh{ y'}_*,y_*'')+(\ve_n\ul t, \ve_n \wh{\ul y}', \sqrt{\ve_n}\ul y'')$. 
	
	Again, it follows that, given $(t_*,\wh {y}_*)\in\E$ there exists a subsequence of $\ve_n(t_*, \wh{y}_*)$ also denoted 
	$\ve_n=\ve_n(t_*, \wh{y}_*)$, and a subset of $\Om_1(t_*, \wh{y}_*)$, also of total measure, and
	also denoted $\Om_1=\Om_1(t_*,\wh{y}_*)$, such that, for all $\om\in\Om_1$,
	\begin{equation}
	   \int_{-R}^R \int_{(-R,R)^{d-1}} \int_{-L}^L |\ff^\tau(\wh S_{(t_*,\wh{y}_*)}(\ve_n)(\ul t,\wh{y'},\ul y''),  \xi) - \ff^\tau(t_*,\widehat{y}_*, \xi) | \, d\xi \,d\widehat{\ul y'}\, d\ul y'' d\ul t \to 0. \label{3.limfb'}
	\end{equation}
\end{lemma}

\bigskip

 Let   $(t_*, \widehat{y}_*) \in \E$, which will be kept fixed until the end of the proof.  Our goal now is to show that $\ff^\tau(t_*,\wh y_*, \xi)$ is a $\chi$-function. 
 
 Let  $R = R(t_*,\widehat{y}_*)$ be the least number between $r$, $r_0$, $T - t_*$ and $t_*$. As before we denote 
 $y_*=(\widehat{y'}_*,0,y_*'')$, $\wh y_*=(\widehat{y'}_*, y_*'')$. For any $\ve > 0$, consider
\begin{equation}
\widetilde{\ff}_\ve(\underline t, \underline y, \xi) = \widetilde{\ff}\left(S_{(t_*,y_*)}(\ve)(\ul t,\ul y), \xi\right) \label{trace.chive}
\end{equation}
for $\omega \in \Omega$, $-L < \xi < L$, and  
\begin{multline*}
(\underline t, \underline y) = (\underline t, \underline{\widehat{y'}}, \underline{y}_{d'}, \ul{y}'') \\
 \in (-R/ \ve, R/ \ve ) \X (-R/\ve, R/\ve)^{d'-1} \X 
(0,R/\ve)\X(-R/\sqrt{\ve},R/\sqrt{\ve})^{d''}\stackrel{\text{def}}{=} Q_\ve .
\end{multline*}
Cearly, $\widetilde{\ff}_\ve$ depends on $(t_*, \widehat{y}_*)$, but, since this point will be fixed henceforth,  we will omit this dependence.

Each $\widetilde{\ff}_\ve$ is still a $\chi$-function, and, in the sense of weak traces, 
\begin{equation}
\widetilde{\ff}_\ve(\ul{t}, \wh {\ul y}, 0, \xi) = \ff^\tau( \wh S_{(t_*,\wh y_*)}(\ve)(\ul t, \wh{\ul y}), \xi), \label{3.chive2} 
\end{equation} 
for $-L<\xi<L$ and 
$$
(\ul t, \ul{\wh y}) \in (-R/\ve, R/\ve)\X(-R/\ve , R/\ve)^{d'-1}\X(-R/\sqrt{\ve},R/\sqrt{\ve})^{d''} \stackrel{\text{def}}{=} \Sigma_\ve,
$$ 
where $\ul{\wh y}:=(\ul y_1,\cdots,\ul y_{d'-1},\ul y_{d'+1},\cdots,\ul y_d)$. 

If $X(t)$ is a predictable  stochastic process, we denote $\int_a^b X(t)\, d\b_k(t):=\int_0^b X(t)\, d\b_k(t)-\int_0^a X(t)\,d\b_k(t)$, for all $k\in\N$. Observe that $\int_a^b X(t)\,d\b_k(t)=-\int_b^a X(t)\,d\b_k(t)$.

{}From  \eqref{3.eqftil} we get that $\wt \ff_\ve$ satisfies  the equation
\begin{multline}\label{3.eqfve0}
\frac{\partial \widetilde \ff_\ve}{\partial \ul t} + \widehat{\abf'}(\xi) \cdot \nabla_{\widehat{\ul y'}} \widetilde \ff_\ve + 
\widetilde{a_{d'}}(\widehat{y}'_* + \ve \ul {\wh y}', \xi) \frac{\partial \widetilde \ff_\ve}{\partial \ul y_{d'}} -\bbf(\xi):\nabla_{\ul{y}''}\wt \ff_\ve= 
-\sqrt{\ve}\abf''(\xi)\cdot\nabla_{\ul y''}\wt \ff_\ve+ \frac{\partial \widetilde{\qq}_\ve}{\partial \xi} \\
- \sum_{k=1}^\infty  \frac{\po}{\po \ul t} \left(\int_{t_*}^{t_*+\ve\ul t} \frac{\po \wt \ff}{\po \xi}(t,S_{y_*}^1(\ve)(\ul y), \xi)\widetilde{g}_{k,\ve}(\ul y,\xi)\, d\b_k(t)\right) 
+ \sum_{k=1}^\infty \d_{\xi=0}  \frac{\po}{\po \ul t}   \left(\int_{t_*}^{t_*+\ve\ul t}  \widetilde{g}_{k, \ve} \,d\b_k(t)\right), 
\end{multline}
where $S_{y_*}^1(\ve)(\ul y)=( y_*'+\ve\ul y', y_*''+\sqrt{\ve}\ul y'')$, and we use the notations 
$$
\begin{aligned}
&\wt g_{k,\ve}(\ul y,\xi):= \wt g_k\left(S_{y_*}^1(\ul y), \xi\right),  \\
&\widetilde{\qq}_\ve :=\widetilde{m}_\ve - \widetilde{G}_\ve^2(\ul y,\xi)\d_{\wt u_\ve(\ul t,\ul y) },\\
&\wt u_\ve(\ul t, \ul y) := \int_{-L}^L \wt \ff_\ve(\ul t, \ul y, \xi) \, d\xi = \wt u\left(S_{(t_*,y_*)}(\ve)(\ul t,\ul y)\right), \\
&\widetilde{G}^2_{\ve}(\ul y, \xi) := \sum_{k=1}^\infty \wt g_{k,\ve}^2( \ul y, \xi)=  \wt G^2(S^1_{y_*}(\ve) (\ul y) , \xi).\\
\end{aligned}
$$
Regarding  $\widetilde{m}_\ve$,  it is, almost surely, the measure such that, for every $a_0<b_0$, $\ldots$, $a_d < b_d$ and $L_1 < L_2$,
\begin{multline*}
\widetilde{m}_\ve \Big( \operatornamewithlimits{\Pi}_{j=0}^d [a_j, b_j] \X [L_1, L_2]  \Big) \\ =
 \frac{1}{\ve^{d'+\frac{d''}2}} \widetilde{m} \Big( S_{(t_*,y_*)}(\ve) \Big[ [ a_0, b_0] \X \cdots\X [a_{d},  b_{d}]\Big] \X [L_1, L_2]  \Big).
\end{multline*}
Therefore, also, almost surely,  for every $a_0<b_0$, $\ldots$, $a_d < b_d$ and $L_1 < L_2$, we have
\begin{multline*}
\widetilde{\qq}_\ve \Big( \operatornamewithlimits{\Pi}_{j=0}^d [a_j, b_j] \X [L_1, L_2]  \Big) \\ = 
\frac{1}{\ve^{d'+\frac{d''}2}} \widetilde{\qq} \Big(S_{(t_*,y_*)}(\ve) \Big[ [ a_0  , b_0] \X\cdots\X [ a_d, b_d]\Big] \X [L_1, L_2]  \Big).
\end{multline*}

Equation \eqref{3.eqfve0} can also be written in the following sometimes more convenient  form 
\begin{multline}\label{3.eqfve}
\frac{\partial \widetilde \ff_\ve}{\partial \ul t} + \widehat{\abf}'(\xi) \cdot \nabla_{\widehat{\ul y}'} \widetilde \ff_\ve + \widetilde{a}_{d'}(\widehat{y}'_*, \xi) \frac{\partial \widetilde \ff_\ve}{\partial \ul y_{d'}} -\bbf(\xi):\nabla_{\ul y''}^2\wt \ff_\ve
= -\sqrt{\ve}\abf''(\xi)\cdot\nabla_{\ul y''}\wt\ff_\ve\\
+\frac{\partial}{\partial \ul y_{d'}} \Big( \big(\widetilde{a_d}(\widehat{y}'_*, \xi) - \widetilde{a_d}(\widehat{y}'_* 
+ \ve \widehat{\ul y}', \xi) \big) \widetilde \ff_\ve \Big) \\
 +  \frac{\partial \widetilde{\qq}_\ve}{\partial \xi} 
- \sum_{k=1}^\infty \frac{\po^2}{\po \ul t\po \xi} \left(\int_{t_*}^{t_*+\ve\ul t} 
 \widetilde{g}_{k,\ve}(\ul y,\xi) \wt \ff(t,S_{y_*}^1(\ve)(\ul y), \xi ) \, d{\beta}_{k}(t)\right)
\\ + \sum_{k=1}^\infty \frac{\po}{\po \ul t} \left(\int_{t_*}^{t_*+\ve\ul t} 
\frac{\partial \widetilde{g}_{k,\ve}}{\po \xi}(\ul y,\xi) \wt \ff(t, S_{y_*}^1(\ve)(\ul y), \xi ) \,d\beta_{k}(t)\right)
\\
+ \sum_{k=1}^\infty  \d_{\xi=0} \frac{\po}{\po \ul t}  \left(\int_{t_*}^{t_*+\ve\ul t}  \widetilde{g}_{k, \ve} \,d\b_k(t)\right).
\end{multline}

In what follows, motivated by \cite{Va}, we are going to prove that $\ff^\tau$ is a $\chi$-function  by proving that, along a suitable subsequence,    
$\wt \ff_\ve\to \ff^\tau(t_*,\wh{y}_*, \cdot )$ in $L_\loc^1(\R^d\X(0,\infty)\X\R)$, for all $\om$ in a subset of total measure of $\Om$.  Here the subsequence and the subset of $\Om$ will depend, in general, on $(t_*,\wh y_*)$, 
as opposed to the deterministic case  in \cite{Va}, where the subsequence does not depend on $(t_*,\wh y_*)$. Nevertheless, this dependence does not have any effect in the conclusion. 
More specifically, keeping the notation in Lemma~\ref{3.lemma1} and Lemma~\ref{3.lemma2}, we will  first obtain a  set of total measure $\Om_2(t_*,\wh y_*)\subset \Om_1(t_*,\wh y_*)$ and a subsequence of $\ve_n(t_*,\wh y_*)$, also denoted $\ve_n(t_*,\wh y_*)$,  so that for each $\om\in\Om_2$, 
$\wt \ff_{\ve_n}\wto \ff^\tau(t_*,\wh y_*, \cdot )$
in the sense of the distributions on $\R^{d'}\X(0,\infty)\X\R^{d''}\X\R$. Then, second, we will obtain another subset of total measure $\Om_3(t_*,\wh y_*)\subset  \Om_2(t_*,\wh y_*)$ and a subsequence $\ve_{n_k}(t_*,\wh y_*)$ of 
$\ve_n(t_*,\wh y_*)$ 
so that, for any $\om\in\Om_3(t_*,\wh y_*)$,      $\wt \ff_{\ve_{n_k}}\to \ff^\tau(t_*,\wh y_*, \cdot)$ in $L_\loc^1(\R^{d'}\X(0,\infty)\X\R^{d''}\X\R)$.
For simplicity, henceforth we will denote $\R^d\X(0,\infty)$ or $\R_+^d$ instead of $\R^{d'}\X(0,\infty)\X\R^{d''}$. 

\begin{lemma}\label{L:3.03} There exists a subset of total measure $\Om_2(t_*,\wh y_*)\subset \Om_1(t_*, \wh y_*)$ and a sequence $\ve_n=\ve_n(t_*,\wh y_*)\to0$, such that for all $\om\in\Om_2(t_*,\wh y_*)$, $\wt \ff_{\ve_n}\wto \ff^\tau(t_*,\wh y_*,\cdot )$ in the sense of the distributions on $\R^d\X(0,\infty)\X(-L,L)$. 
\end{lemma} 

\begin{proof} Let  $\frak D\subset  C_c^\infty(\R^d\X(0,\infty)\X(-L,L))$ be countable and dense in $C_c^2(\R^d\X(0,\infty)\X(-L,L))$.  Let $\varphi\in \frak D$,  and let 
$\Lambda_{\ve}$ denote the distribution corresponding to the stochastic Wiener processes 
in \eqref{3.eqfve0},  That is,
\begin{multline*}
\Lambda_{\ve}:=- \sum_{k=1}^\infty  \frac{\po}{\po \ul t} \left(\int_{t_*}^{t_*+\ve\ul t} \frac{\po \wt \ff}{\po \xi}(t, S_{y_*}^1(\ve)(\ul y), \xi)\wt g_{k,\ve}(\ul y,\xi)\, d\b_k(t)\right) \\
+ \sum_{k=1}^\infty  \d_{\xi=0} \frac{\po}{\po \ul t}   \left(\int_{t_*}^{t_*+\ve\ul t}  \widetilde{g}_{k, \ve} \,d\b_k(t)\right).
\end{multline*}
It is not hard to verify that, for $\ve$ sufficiently small,
$$
\la \Lambda_{\ve} , \varphi\ra= -\sum_{k=1}^\infty \int_\R \int_{\R_+^{d}}\int_{t_*}^{t_*+\ve \ul t} \wt{g}_{k,\ve}(\ul y,\wt u(t, S_{y_*}^1(\ve)(\ul y)))\varphi_{\ul t} (\ul t,\ul y, \wt u(t, S_{y_*}^1(\ve)( \ul y)))d\b_k(t) \, d\ul y\,d \ul t,
$$
where we denote $\R_+^d=\R^{d-1}\X(0,\infty)$.  We have
\begin{multline*}
\bbE |\la \Lambda_{t_0,\wh y_0,\ve}, \varphi\ra| \\
\le \bbE \int_\R \int_{\R_+^{d}}\left| \sum_{k=1}^\infty  \int_{t_0}^{t_0+\ve \ul t} \wt{g}_{k,\ve}(\ul y,\wt u(t, S_{y_*}^1(\ve)(\ul y)))\varphi_{\ul t} (\ul t,\ul y, \wt u(t,S_{y_*}^1(\ve)(\ul y)))d\b_k(t)\right|  \, d\ul y\,d \ul t\\
\le \bbE \int_\R \int_{\R_+^{d}}\left\{ \sum_{k=1}^\infty \left|\int_{t_*}^{t_*+\ve \ul t} \big|\wt{g}_{k,\ve}(\ul y,\wt u(t, S_{y_*}^1(\ve)(\ul y)))\varphi_{\ul t} (\ul t,\ul y, \wt u(t,S_{y_*}^1(\ve)(\ul y))) \big|^2\,dt \right|\right\}^{1/2}   \, d\ul y\,d \ul t\\
\le \bbE \int_\R \int_{\R_+^{d}}\sup_\xi |\varphi_{\ul t}(\ul t,\ul y,\xi)| \left\{ \sum_{k=1}^\infty \left| \int_{t_*}^{t_*+\ve \ul t} 
\big|\wt{g}_{k,\ve}(\ul y,\wt u(t, S_{y_*}^1(\ve)(\ul y))) \big|^2\,dt\right| \right\}^{1/2}   \, d\ul y\,d \ul t\\
\le  \bbE \int_\R \int_{\R_+^{d}}\sup_\xi |\varphi_{\ul t}(\ul t,\ul y,\xi)| \left\{ D \left| \int_{t_*}^{t_*+\ve \ul t} (1+ | \wt u(t, S_{y_*}^1(\ve)(\ul y))|^2)\,dt\right| \right\}^{1/2}   \, d\ul y\,d \ul t\\
\le D^{1/2} \text{diam}\,\{\supp \varphi_{\ul t} \}^{d+1} \|\varphi_{\ul t}\|_\infty (1+\|u\|_\infty^2)^{1/2} \ve^{1/2}\to 0\quad\text{as $\ve\to0$},
\end{multline*} 
where we have used Burkholder inequality (see, e.g., \cite{O}) in the second inequality above, and \eqref{e1.4} in the fourth inequality  above. 
Therefore, using a diagonal process, we can obtain a subsequence of $\ve_n(t_*,\wh y_*)$, obtained from Lemma~\ref{3.lemma1} and Lemma~\ref{3.lemma2},
 which we still denote $\ve_n(t_*,\wh y_*)$ and  a set of total measure $\Om_2(t_*,\wh y_*)\subset \Om_1(t_*,\wh y_*)$ such that, for all $\om\in\Om_2$,  
$$
|\la \Lambda_{\ve_n}, \varphi\ra| \to 0\quad \text{for all $\varphi\in C_c^\infty(\R^d\X(0,\infty)\X(-L,L))$}.
$$ 

Now, let us fix $\om\in\Om_2(t_*,\wh y_*)$. Let $\varphi\in C_c^\infty(\R^d\X(0,\infty)\X(-L,L))$ be of the form $\varphi=\rho_h(\ul y_{d'})\tilde\varphi$, with $\tilde \varphi\in C_c^\infty(\R^d\X[0,\infty)\X(-L,L))$  and
$\rho_h(\ul y_{d'})=\int_0^{\ul y_{d'}}\z_h(s)\,ds$, where $\z_h(s)=h^{-1}\z(h^{-1}s)$ and $\z\in C_c^\infty((0,1))$, with $\z\ge0$ and $\int_0^1 \z(s)\,ds=1$.  Applying \eqref{3.eqfve0} to $\varphi$
of this form, we get, after letting $h\to0$, using Lemma~\ref{L:3.02},
\begin{multline}\label{3.eqfve2}
\int_{\R}\int_0^\infty\int_{\R^{d-1}}\int_\R  \Bigl\{\widetilde \ff_\ve \frac{\po \wt \varphi}{\partial \ul t} +  \widetilde{\ff}_\ve\, \wh {\abf'}(\xi) \cdot  \nabla_{\widehat{\ul y'}}\wt \varphi
+\wt\ff_\ve \,  \widetilde{a_{d'}}(\widehat{y'}_* + \ve \ul {\wh y'}, \xi) \frac{\po\wt \varphi}{\po\ul y_{d'}}\\
-\wt\ff_\ve\,\bbf(\xi):\nabla_{\ul y''}\wt\varphi \Bigr\} \,d\ul t\, d\wh {\ul y}\,d\ul y_{d'}\,d\xi \\ 
+\int_{\R}\int_{\R^{d-1}}\int_\R   \widetilde{a_d}(\widehat{y'} + \ve \ul {\wh y'}, \xi) \ff^\tau\left(\wh S_{(t_*,\wh y_*)}(\ve)(\ul {\wh y}), \xi\right) 
\wt \varphi (\ul t, \wh{\ul y},0,\xi)\, d\ul t\,d\wh{\ul y}\, d\xi\\
=-\sqrt{\ve}\int_{\R}\int_0^\infty\int_{\R^{d-1}}\int_\R \abf''(\xi)\cdot\nabla_{\ul y''} \wt \varphi\, d\ul t\, d\wh {\ul y}\,d\ul y_{d'}\,d\xi    \\
+ \la \wt \qq_\ve,\frac{\po\wt\varphi}{\po \xi}\ra-\la\Lambda_\ve, \wt\varphi\ra.
\end{multline}
Taking $\ve=\ve_n(t_*,\wh y_*)$. Passing to a subsequence of $\ve_n(t_*,\wh y_*)$ if necessary so that $\wt \ff_{\ve_n}\wto \wt\ff$ in the weak-$\star$ topology of $L^\infty(\R^d\X(0,\infty)\X(-L,L)))$, 
for some $\wt \ff \in L^\infty(\R^d\X(0,\infty)\X(-L,L))$, and making $\ve_n(t_*,\wh y_*)\to0$ we get for $\wt\ff$,
\begin{multline}\label{3.eqfve3} 
\int_{(-L,L)} \int_{\R_+^d}\int_\R  \Big\{\widetilde \ff \,\frac{\po \wt \varphi}{\partial \ul t} +  \widetilde{\ff}\, \wh {\abf}(\xi) \cdot  \nabla_{\widehat{\ul y}}\wt \varphi
+\wt\ff \,  \widetilde{a_{d'}}(\widehat{y'}_*, \xi) \frac{\po\wt \varphi}{\po\ul y_{d'} } \\
-\wt\ff\,\bbf(\xi):\nabla_{\ul y''}\wt\varphi \Big\} \, d\ul t \,d\ul y \,d\xi  \\
+\int_{\R} \int_{\R^{d-1}} \int_\R  \widetilde{a_{d'}}(\widehat{y}_*, \xi)\, \ff^\tau(t_*,\wh y_*, \xi) \wt \varphi (\ul t, \wh{ \ul y},0,\xi)\,d\ul t \,d\wh{\ul y}\, d\xi=0. 
\end{multline}
Now, since  $\wt \ff$ and $\ff^\tau$ vanish for $\xi\notin (-L,L)$ and $a_{d'}(\wh y_*',\xi)\neq 0$, for a.e.\ $\xi\in(-L,L)$, by choosing $\tilde \varphi(\ul t, \ul y, \xi)=\rho(\xi)\tilde \phi(\ul t,\ul y)$, with $\rho\in C_c^\infty(\R)$, $ \tilde\phi\in C_c^\infty(\R\X\R^{d-1}\X[0,\infty))$,
we get that for almost every $\xi$, $\wt \ff(\cdot,\cdot,\xi)$ satisfies
 \begin{multline}\label{3.eqfve4} 
 \int_{\R_+^d}\int_\R  \widetilde \ff \,\frac{\po \wt \varphi}{\partial \ul t} +  \widetilde{\ff}\, \wh {\abf}'(\xi) \cdot  \nabla_{\widehat{\ul y}'}\wt \varphi
+\wt\ff \,  \widetilde{a}_{d'}(\widehat{y}_*', \xi) \frac{\po\wt \varphi}{\po\ul y_{d'} } \\
-\wt\ff\,\bbf(\xi):\nabla_{\ul y''}\wt\varphi  \,d\ul t \,d\ul y \\ 
+ \int_{\R^{d-1}} \int_\R  \widetilde{a}_{d'}(\widehat{y}_*', \xi)\, \ff^\tau(t_*,\wh y_*, \xi) \wt \varphi (\ul t, \wh{ \ul y},0,\xi)\,d\ul t \,d\wh{\ul y} =0. 
\end{multline}  
Now,  we make the change of coordinates 
$$
\ul y_{d'}= a_{d'}(\wh y_*',\xi)\, \ul z_d,\quad  \ul t  =\ul \tau+ \ul y_{d'},\quad \wh{\ul y'}= \wh{\ul z'} + \ul z_{d'}\, \wh a(\xi), \quad \ul y''=\ul z''.
$$   
We assume for the moment that $a_{d'}(\wh y_*',\xi)>0$. We get that, in this new system of coordinates, $\wt \ff$ satisfies 
\begin{equation}\label{e.endL}
\int_{\R_+^d} \int_\R \wt \ff \frac{\po \wt \phi}{\po \ul z_d} -\wt \ff\bbf(\xi):\nabla_{\ul z''}\wt\phi\, d\ul \tau\, d\wh{\ul z}\, d\ul z_{d'} + \int_{\R^{d-1}} \int_\R   \ff^\tau(t_*,\wh y_*, \xi)\wt \phi (\ul t, \wh{ \ul z},0)\,d\ul \tau \,d\wh{\ul z} =0.
\end{equation}
for all $\tilde \phi\in C_c^\infty(\R^d\X[0,\infty))$, and we denote $\wh{\ul z}=(\wh{\ul z'},\ul z'')$. Using a test function of the form 
$\wt \phi(\tau,\wt{\ul z'}, \ul z_{d'}, \ul z'')= \phi_1(\tau, \wh{\ul z'})\phi (\ul z_{d'}, \ul z'')$, we then see that, for a.e.\ $(\tau,\wh{\ul z'})\in \R^{d'}$,
 $\wt \ff$ satisfies 
 \begin{equation}\label{e.endL2}
\int_0^\infty \int_{\R_+^{d''}}  \wt \ff \frac{\po \phi}{\po \ul z_{d'}} -\wt \ff\bbf(\xi):\nabla_{\ul z''} \phi\, d\ul z'' \, d\ul z_{d'} + \int_{\R^{d''}}   \ff^\tau(t_*,\wh y_*, \xi) \phi (0,\ul z'')\, d\wh{\ul z''} =0.
\end{equation}       
Let us fix $(\ul \tau, \wh {\ul z'})$ for which \eqref{e.endL2} holds and let us see $\wt \ff$ as a function of $(\ul z_{d'},\ul z'')$ only, for simplicity. Since $\wt \ff\in L^\infty((0,\infty)\X\R^{d''})$, for a.e.\ $\ul z_{d'}\in(0,\infty)$, $\wt \ff(\ul z_{d'}, \cdot)$ is a tempered distribution and so we may use as test function a $\phi$ of the form $\phi( \ul z_{d'},\ul z'')= \z(\ul z_{d'})[ \FF_{\ul z''}\theta](\ul z'')$, where $\z\in C^\infty([0,\infty))$ and 
$\theta \in \S(\R^{d''})$, where the latter is the Schwarz space of fast smooth decaying functions.  Therefore we conclude that  
$\FF_{\ul z''}[\wt \ff](\cdot, \ul z_{d'})$ satisfies the following ODE with prescribed initial value
$$
\begin{cases}    \frac{d \FF_{\ul z''}[\wt \ff]}{d\ul z_{d'}}= {\k''}^\top \bbf(\xi){\k''}\FF_{\ul z''}[\wt \ff], \\
                             \FF_{\ul z''}(0)=\ff^\tau(t_*,\wh y_*, \xi) \d_0(\k''), 
                             \end{cases}
$$                              
where $\d_0(\k'')$ is the Dirac measure concentrated at 0 in the space of frequencies $\k''\in \R^{d''}$.
Hence, we conclude that 
$$
\wt \ff(\ul z_d, \ul z'')= \ff^\tau(t_*, \wh y_*,\xi), \qquad\text{for a.e.\ $(\ul z_{d'}, \ul z'') \in (0,\infty)\X\R^{d''}$}.
$$  
The case where $a_{d'}(\wh y_0,\xi)<0$ is treated exactly in the same way. Hence, bringing  back  the variables $(\ul t,\ul y, \xi)$,  we finally 
arrive at the  desired conclusion. 
\end{proof}

\begin{lemma}\label{L:03.4} There exists a subset of total measure $\Om_3(t_*,\wh y_*)\subset \Om_2(t_*, \wh y_*)$ and a subsequence 
of $\ve_n(t_*,\wh y_*)$, still denoted by $\ve_n\to0$,  such that for all $\om\in\Om_2(t_*,\wh y_*)$, $\tilde \ff_{\ve_n}$ converges  strongly to $\tilde \ff$ in  $L_\loc^1(\R^d\X(0,\infty)\X(-L,L))$  which satisfies \eqref{3.eqfve3}. Consequently, $ \ff^\tau$ is a $\chi$-function.
Here we keep denoting $\R^d\X(0,\infty)=\R^{d'}\X(0,\infty)\X\R^{d''}$. 
\end{lemma}
 
\begin{proof}  

First, we localize the equation \eqref{3.eqfve} by multiplying it by a bump function $\Phi(\ul y,\xi)=\Phi_1(\ul y) \Phi_2(\xi)$, with 
 $\Phi_1\in C_c^\infty(\R^d\X(0,\infty))$, $\Phi_2\in C_c^\infty((-L,L))$,  with $\Phi_1\equiv 1$, for $(\ul t,\wh{ \ul y}', \ul y_{d'},\ul y'')\in (-R,R)^{d'}\X(1/R,R)\X(-R,R)^{d''}$, $\Phi_1\equiv0$,  for $(\ul t,\wh{ \ul y}', \ul y_{d'},\ul y'')\notin (-2R,2R)^{d'}\X(1/(2R),2R)\X
 (-2R,2R)^{d''}$,  $\Phi_2\equiv 1$, for $\xi\in (-L_0,L_0)$.   Let us denote  
 $\tilde \ff_\ve^\Phi=\Phi \tilde \ff_\ve$.  We then get the following equation for $\tilde \ff_\ve^\Phi$,
\begin{multline}\label{e5.11''}
\po_{\un{t}} \tilde \ff_\ve^\Phi+\hat \abf'(\xi)\cdot\nabla_{\un{\hat y'}}\tilde \ff_\ve^\Phi+\tilde a_{d'}(\hat y_*' ,\xi)\po_{\un{y}_{d'}} \tilde \ff_\ve^\Phi  -\bbf(\xi):\nabla_{\un{y}''}^2\tilde \ff_\ve^\Phi=\\
-\ve^{1/2} \Phi \abf''(\xi)\cdot\nabla_{\un{y}''}\tilde \ff_\ve  
+\frac{\partial}{\partial \ul y_{d'}} \Big( \big(\widetilde{a}_{d'}(\widehat{y}'_*, \xi) - \widetilde{a}_{d'}(\widehat{y}'_* 
+ \ve \widehat{\ul y}', \xi) \big) \widetilde \ff_\ve^\phi \Big) 
+\po_\xi(\Phi \tilde m_\ve)\\
+ \tilde \ff_\ve \tilde a_{d'}(\hat y_*',\xi)  \po_{\ul{y}_{d'}}\Phi+ \tilde \ff_\ve \hat \abf'(\xi)\cdot\nabla_{\un{\hat y}'}\Phi
-\tilde \ff_\ve\bbf(\xi):\nabla_{\un{\hat y}''}^2\Phi \\
+\s(\xi)\nabla_{\ul{y}''}\Phi\otimes \s(\xi)\nabla_{\ul{y}''}\tilde \ff_\ve\\
-  \left((\tilde a_{d'}(\hat y_*',\xi)- \tilde a_{d'}(\widehat{y}'_* 
+ \ve \widehat{\ul y}',\xi)) \tilde \ff_\ve \right)\po_{\un{y}_{d'}}\Phi -\tilde m_\ve\po_\xi \Phi+ \Phi \Lambda_\ve,
\end{multline} 
 where we denote again by $\Lambda_\ve$ the distribution composed with the three stochastic integrals in \eqref{3.eqfve}.  Concerning equation \eqref{e5.11''} we observe first observe that 
 \begin{equation}\label{e5.11''a}
 -\ve^{1/2} \Phi \abf''(\xi)\cdot\nabla_{\un{y}''}\tilde \ff_\ve  
+\frac{\partial}{\partial \ul y_{d'}} \Big( \big(\widetilde{a}_{d'}(\widehat{y}'_*, \xi) - \widetilde{a}_{d'}(\widehat{y}'_* 
+ \ve \widehat{\ul y}', \xi) \big) \widetilde \ff_\ve^\phi \Big) 
  \end{equation}
  clearly converges to zero in $W^{-1,2}(\R^d\X(0,\infty)\X(-L,L))$, for all $\om\in \Om_2(t_*,\wh y_*)$. Moreover, 
 except for \eqref{e5.11''a}, which comprises the first two terms on the second line in \eqref{e5.11''},   and  the stochastic integrals, $ \Phi \Lambda_\ve$,
 all other terms on the right-hand side of this equation may be rendered as
 $\po_\xi \mu_\ve$ for some measure $\mu_\ve\in \M(\R^d\X(0,\infty)\X(-L,L)))$ for all $\om\in \Om_2(t_*,\wh y_*)$. Indeed, 
 this is obviously the case for the last term in the second line of \eqref{e5.11''}. Also, 
  the terms  in the third line of \eqref{e5.11''} altogether can clearly be put in the form $\po_\xi\mu_\ve$ for $\mu_\ve\in \M(\R^d\X(0,\infty)\X(-L,L))$  weakly-$\star$ converging to 
  \begin{multline*}
 \int_0^{\xi}\Bigl \{ \ff^\tau(t_*, \wh y_*,\wt\xi) \tilde a_{d'}(\hat y_*',\wt\xi)  \po_{\ul{y}_{d'}}\Phi+ \ff^\tau(t_*,\wh y_*,\wt\xi)  \hat \abf'(\wt \xi)\cdot\nabla_{\un{\hat y}'}\Phi \\
- \ff^\tau(t_*,\wh y_*, \wt\xi)\bbf(\wt \xi):\nabla_{\un{\hat y}''}^2\Phi \Bigr\}\, d\wt \xi,
\end{multline*}
as $\ve\to0$, along a suitable subsequence $\ve_n(t_*,\wh y_*)$, for all $\om\in\Om_2(t_*,\wh y_*)$,  due to Lemma~\ref{L:3.03}.  

 As for the term in the fourth line of \eqref{e5.11''},  we first observe that, formally, we have
  \begin{multline*} 
   \s(\xi)\nabla_{\ul{y}''}\tilde \ff_\ve=\s(\xi)\nabla_{\ul{y}''}\tilde u_\ve \d_{\tilde u_\ve(\ul y)}= \nabla_{\ul{y}''}\int_0^{\tilde u_\ve(\ul{y}'')}\s(\z)\,d\z\,\d_{\tilde u_\ve(\ul{y}'')}(\xi) \\
   =-\ve^{1/2}\nabla_{x''}\int_0^{\tilde u_\ve(\ul{y}'')}\s(\z)\,d\z\,\po_{\xi} 1_{(0, \tilde u_\ve(\ul{y}''))}(\xi)\\
   =-\po_{\xi}\left(\ve^{1/2}\nabla_{x''}\int_0^{\tilde u_\ve(\ul{y}'')}\s(\z)\,d\z\, 1_{(0, \tilde u_\ve(\ul{y}''))}(\xi) \right).
 \end{multline*}
 Now the primitive $\Sigma(u)=\int_0^u\s(\z)\,d\z$ is a Lipschitz function of the function $b(u)$, that is, $\Sigma (u)=\tilde \Sigma(b(u))$ for some Lipschitz $\tilde \Sigma$ as follows from \eqref{e1.4''}.
 The formal calculation may be easily made rigorous. Using these facts  and other trivial rearrangements we may at last also render the  term 
 in the fourth line of \eqref{e5.11''} in the form $\po_\xi\mu_\ve$ 
 for some measure $\mu_\ve$ converging to zero as $\ve\to0$. Further, it is immediate that the first two terms in the last line of \eqref{e5.11''} can also be put in the form $\po_\xi \mu_\ve$ for some measure $\mu_\ve\wto0$ as $\ve\to0$.

Concerning the term  $\Phi \Lambda_\ve$ in \eqref{e5.11''}, let us denote
$$
\begin{aligned}
&\ell_\ve^{(1)} (\ul t,\ul y,\xi):= \sum_{k=1}^\infty \int_{t_0}^{t_0+\ve \ul t} \wt g_{k,\ve}(\ul y,\xi)\wt \ff(t,S_{y_*}^1(\ve)(\ul y), \xi )\, d\b_k(t),\\
& \ell_\ve^{(2)} (\ul t,\ul y,\xi):= \sum_{k=1}^\infty \int_{t_0}^{t_0+\ve \ul t} \frac{\po\wt g_{k,\ve}}{\po \xi}(\ul y,\xi)\wt \ff(t,S_{y_*}^1(\ve)(\ul y), \xi )\, d\b_k(t),\\
& \ell_\ve^{(3)}(\ul t,\ul y) := \sum_{k=1}^\infty \int_{t_0}^{t_0+\ve \ul t} \wt g_{k,\ve}(\ul y,0)\, d\b_k(t).
\end{aligned}
$$
Clearly, $\Lambda_\ve=-\po_\xi\po_{\ul t}\ell_\ve^{(1)}+\po_{\ul t}\ell_\ve^{(2)}+\d_{\xi=0}\po_{\ul t}\ell_\ve^{(3)}$. Let $\Om_2(t_*,\wh y_*)$ and $\ve_n(t_*,\wh y_*)$ be the set of total measure and the subsequence obtained in Lemma~\ref{L:3.03}. 

We claim that, for any bounded open set $V\subset\subset \R^d\X(0,\infty)$, there is a set of total measure 
$\Om_3(t_*,\wh y_*)\subset \Om_2(t_*,\wh y_*)$ and a subsequence of $\ve_n((t_*,\wh y_*)$, also denoted $\ve_n(t_*,\wh y_*)$,  such that, for all  $\om\in\Om_3$, $\ell_{\ve_n}^{(1)}, \ell_{\ve_n}^{(2)}\in L^2(V\X(-L,L))$,
	 $\ell_{\ve_n}^{(3)}\in L^2(V)$, and $\ell_{\ve_n}^{(1)}, \ell_{\ve_n}^{(2)}\to  0$ in $L^2(V\X(-L,L))$ and $\ell_{\ve_n}^{(3)}\to 0$ in $L^2(V)$ as $n\to\infty$.   
	 Moreover, by a standard diagonal argument, we can find a set of total measure $\Om_3(t_*,\wh y_*)\subset\Om_2(t_*,\wh y_*)$ and a subsequence of $\ve_n(t_*,\wh y_*)$, also denoted $\ve_n(t_*,\wh y_*)$ such that the assertion is true for any $V\subset\subset \R^d\X(0,\infty)$.
	 
	 Indeed, it suffices to prove the assertion for $\ell_\ve^{(1)}$ since the proof for the others is similar. By It\^o isometry, we have
	  \begin{multline*}
	  \bbE \int_{-L}^L\int_V |\ell_{\ve}^{(1)}|^2 \,d\ul t\,d\ul y\, d\xi \\ 
	  =\bbE\int_{-L}^L\int_V\sum_{k=1}^\infty\left| \int_{t_*}^{t_*+\ve\ul t} | \wt g_{k,\ve}(\ul y,\xi)|^2 |\wt \ff (t, S_{y_*}^1(\ve)(\ul y),\xi)|^2\, dt\right| \,d\ul t\,d\ul y\,d\xi\\
	  \le \bbE\int_{-L}^L\int_V \left| \int_{t_*}^{t_*+\ve\ul t} D(1+|\xi|^2)\, dt\right| \,d\ul t\,d\ul y\,d\xi\\
	  \le C(V,L, D)\ve\to 0,\quad \text{as $\ve\to0$}. 
	 \end{multline*}
	 Therefore, making $\ve=\ve_n(t_*,\wh y_*)$, we deduce that we can obtain a set of total measure $\Om_3(t_*,\wh y_*)\subset \Om_2(t_*,\wh y_*)$ , and a subsequence of $\ve_n(t_*,\wh y_*)$ also
	 denoted $\ve_n(t_*,\wh y_*)$ such that the claim for $\ell_\ve^{(1)}$ holds. The proof for $\ell_\ve^{(2)}, \ell_\ve^{(3)}$ follows the same lines and that the claim holds for all $V\subset\subset \R_+^d$ follows trivially
	 by a standard diagonal argument. 

As a consequence of the claim just proven, for all $\om\in\Om_3(t_*,\wh y_*)$, we may write the term  $\Phi \Lambda_\ve$ in \eqref{e5.11''} in the form $(1-\po_\xi^2)(1-\Delta_{\ul t,\ul y})^{1/2} k_\ve$, for some $k_\ve$ compact in $L^2(\R^d\X(0,\infty)\X\R)$. Also observe that, if the reduced symbol $\LL_0(\tau, \k,\xi)$ satisfies \eqref{e1.nondeg}, then the reduced symbol
 $$
\tilde \LL_0(\tau,\k, \xi):=\tau + \wt\abf'(\xi) \cdot \k' 
-\bbf(\xi): (\k''\otimes\k''),
$$
where we set $\wt a_{d'}(\xi)=\widetilde{a_{d'}}(\widehat{y}'_*, \xi)$  satisfies: for all $(\tau,\k',\k'')\in\R^{d+1}$, with $\tau^2+|\k'|^2+|\k''|^2=1$, we have
 \begin{equation}\label{e5.nondeg0}
 \operatorname{meas}\left\{\xi\in[-L_0,L_0]\,:\, |\tau+ \wh \abf'(\xi)\cdot \k'|^2+\left(\bbf(\xi):(\k''\otimes\k'')\right)^2=0\right\}=0,
 \end{equation}
 as it is easy to verify. We also observe that, in view of the compactness of the embedding of the space of signed measures $\M_\loc(\R^d\X(0,\infty)\X\R)$ in $W_\loc^{-1,p}(\R^d\X(0,\infty)\X\R)$, for some $1<p<2$, the terms that can be written in the form $\po_\xi\mu_\ve$ can be cast in the form $(1-\po_\xi^2)(1-\Delta_{\ul t,\ul y})^{1/2} \tilde k_\ve$, with $\tilde k_\ve$ in a compact in $L^p(\R^d\X(0,\infty)\X\R)$.    Therefore, the lemma will follow from the following new averaging lemma, that we state and prove subsequently. This result can also be obtained as a consequence of a more general 
 result recently established by the fourth author in \cite{Na}.
 \end{proof}
 
 \begin{lemma}\label{L:5.1} Let $N,N',N''$ be positive integers with $N=1+ N'+N''$, $f_n(y,\xi)$ be a  bounded sequence in $L^\infty(\R^N\X\R)$,  such that $f_n(y,\xi)=0$, if $(y,\xi)\notin K\X[-L,L]$, where $K\subset \R^{N}$ is compact.  
 Let  $h_n$  be compact in $L^p(\R^N\X\R)$, $1<p<2$.   For $y\in\R^N$ we write $y=(y_0,y',y'')$, $y_0\in\R$, $y'\in\R^{N'}$, $y''\in\R^{N''}$. 
 Let  $\CS^{N''\X N''}$ denote the space of the $N''\X N''$ symmetric matrices, and let $(\a_0,\a')\in C^2([-L_0,L_0];\R^{1+N'})$, $\b\in C^2([-L_0,L_0];\CS^{N''\X N''})$, for some $L_0>L$.  Assume
  \begin{equation}\label{e5.14}
 \a_0(\xi)\po_{y_0}f_n+   \a'(\xi)\cdot\nabla_{y'}f_n-\b(\xi): \nabla_{y''}^2f_n= (1-\po_\xi^2)(1-\Delta_y)^{1/2} h_n, 
 \end{equation}
where  $\a_0(\xi)\ne0$ and $\b(\xi)>0$, for a.e.\ $\xi\in [-L,L]$, and  the symbol $\LL(\tau,\k',\k'',\xi):=i( \a_0(\xi) \tau+\a'(\xi)\cdot \k')+\b (\xi)(\k'',\k'')$ satisfies: for all $\k:=(\tau,\k',\k'')\in\R^N$, with $\tau^2+|\k'|^2+|\k''|^2=1$, we have
\begin{equation}\label{e5.nondeg}
\operatorname{meas}\{\xi\in[-L,L]\,:\, |\a_0(\xi)\tau+\a'(\xi)\cdot\k'|^2+(\k'' \b(\xi)\k'')^2=0\}=0.
\end{equation} 
Then, the average $u_n(y)=\int_{\R} f_n(y,\xi)\,d\xi$ is relatively compact in $L_\loc^1(\R^N)$. 
\end{lemma}

 \begin{proof}  Since $f_n$ is uniformly bounded with compact support, we may, with no loss of generality, assume that $f_n\wto 0$, weakly-star in $L^\infty(\R^N\X\R)$. Similarly, since $h_n$ is compact in $L^p(\R^N\X\R)$, $1<p<2$, we may assume that $h_n\to0$ in $L^p(\R^N\X\R)$.     
 Let $\z\in C_c^\infty(\C)$ be radially symmetric and such that $\z(z)=1$ for $|z|<1$ and $\z(z)=0$, for $|z|>2$, and $\psi(z)=1-\z(z)$, $z\in\C$. Let $\k\in \R^N$, $\k=(\tau, \k',\k'')$, $\tau\in\R$, $\k'\in\R^{N'}$, $\k''\in\R^{N''}$. Let us denote  $\tilde \a'(\xi)=(\a_0(\xi),\a'(\xi))$
 and $\tilde \k'=(\tau,\k')$. 
 For $\d>0$ and $\g>0$, let us denote
 \begin{equation*}
 \begin{aligned}
 &\z^{(1)}(\k):=\z\left(\frac{|\k|}{\g}\right),\\
 &\z^{(2)}( \k,\xi):=\z \left(\frac{i\tilde \a'(\xi)\cdot \tilde\k'}{\d|\tilde \k'|}\right),\\
 &\z^{(3)}(\k,\xi):=\z\left(\frac{\b(\xi)(\k'',\k'')}{\d |\k''|^2}\right),\\
 &\psi^{(1)}(\k):=1-\z^{(1)}(\k),\\
 &\psi^{(i)}(\k,\xi)=1-\z^{(i)}(\k,\xi),\qquad i=2,3.
 \end{aligned}
 \end{equation*}
 Let  $\fF$ denote the Fourier transform in $\R^N$. We have 
 \begin{equation} \label{e5.16'}
 \begin{aligned}
 \fF f_n&=\z^{(1)} \fF f_n+\psi^{(1)}\fF f_n\\
            &=\z^{(1)} \fF f_n+\psi^{(1)} \z^{(2)}\fF f_n+\psi^{(1)}\psi^{(2)} \fF f_n\\
            &=\z^{(1)}\fF f_n+  \psi^{(1)} \z^{(2)}\fF f_n+\psi^{(1)}\psi^{(2)} \z^{(3)} \fF f_n+ \psi^{(1)}\psi^{(2)} \psi^{(3)} \fF f_n\\
            &=: \fF f_n^{(1)} +\fF f_n^{(2)}+ \fF f_n^{(3)} +\fF f_n^{(4)}. 
    \end{aligned}
\end{equation}
Observe that $\fF f_n(\k,\xi)=0$, is $\xi\notin [-L,L]$, for all $\k\in\R^N$.  Let us also denote
\begin{equation} \label{e5.16''}
\begin{aligned}
&\int_{-L}^L  f_n^{(1)}\,d\xi=: v^{(1)},\\
&\int_{-L}^L  f_n^{(2)}\,d\xi=: v^{(2)},\\
&\int_{-L}^L  f_n^{(3)}\,d\xi=: v^{(3)},\\
&\int_{-L}^L  f_n^{(4)}\,d\xi=: v^{(4)}.
\end{aligned}
\end{equation}
Thus, 
\begin{equation}\label{e5.16'''}
\int_{-L}^L f_n\, d\xi= v^{(1)}+  v^{(2)}  +  v^{(3)} +  v^{(4)}.
\end{equation}

 Concerning $v^{(1)}$, we first observe that, since $ \sup_{\xi\in\R}  \|f_n(\cdot,\xi)\|_{L^1(\R^N)}\le C$, for some constant $C>0$ independent of $n$. Therefore, it follows that $ \sup_{\xi\in\R}\|\fF f_n(\cdot,\xi)\|_{L^\infty(\R^N)}\le C$. Thus, by Cauchy-Schwarz inequality,   we have
 \begin{equation*}
 \begin{aligned}
 \int_{\R^N} |v_n^{(1)} (x)|^2\,dx &\le 2L\int_{-L}^L\int_{\R^N} |\z^{(1)}(\k)|^2 |\fF f_n(\k,\xi)|^2\,d\k\,d\xi\\
                                           &\le 4L^2  C \int_{\R^N}  |\z^{(1)}(\k)|^2   \,d\k \\
                                           &\le  \tilde C \g^{N},\\
                                           \end{aligned}
  \end{equation*}
 for some constant $\tilde C>0$ independent of $n$ and $\g$. 
 
  Concerning $v^{(2)}$, we have
  \begin{equation*}
  \begin{aligned}
  \fF v_n^{(2)} (\k) &= \int_{\R}  \fF f_n^{(2)}(\k,\xi)\,d\xi\\
  &=\int_{\R}  \psi^{(1)}(\k)\z^{(2)}(\k,\xi) \fF f_n (\k,\xi)\,d\xi.
  \end{aligned}
  \end{equation*}
  By Plancherel identity and Cauchy-Schwarz 
  inequality, we have 
 \begin{equation*}
 \begin{aligned}
 \|v_n^{(2)}\|_{L^2(\R^N)}^2 &=\|\fF v_n^{(2)}\|_{L^2(\R^N)}^2\\
 &\le \int_{\R^N}\left(\int_{-L}^L  1_{\{|\tilde\a'(v)\cdot \tilde\k'|\le 2\d|\tilde\k'| \}}(\xi)\,d\xi\right)\\
 &\qquad\qquad \left(\int_\R \left|\psi^{(1)}(\k) \z^{(2)}(\k,\xi)\fF f_n(\k,\xi)\right|^2\,d\xi\right)\, d\k\\
 &\le \sup_{|\tilde\k'|=1}\left|\{\xi\in[-L,L] \,:\,|\tilde\a'(\xi)\cdot\tilde\k'| \le  2\d\}\right| \|f_n\|_{L^2(\R^N\X\R)}^2,
 \end{aligned}
 \end{equation*}
 where we denote by $|\{\cdots\}|$ the Lebesgue measure of $\{\cdots\}$.  Now, define the functions $h_\d: \mathbb{S}^{d'}:=
 \{|\tilde\k'|=1\}\to \R$  by
 $$
 h_\d(\tilde\k')= \left|\{\xi\in [-L,L]  \,:\,|\tilde\a'(\xi)\cdot\tilde\k'| \le  2\d\}\right|.
 $$ 
It is easy to check that $h_\d$ is continuous on $\mathbb{S}^{d'}$ and that $h_{\d_1}(\tilde\k')\le h_{\d_2}(\tilde\k')$, if $\d_1\le \d_2$, for all $\tilde\k'\in\mathbb{S}^{d'}$. Therefore, because of
the nondegeneracy condition \eqref{e5.nondeg} we deduce that $\sup_{|\tilde\k'|=1}  h_d(\tilde\k')\to 0$, as $\d\to0$. Therefore, we may write
\begin{equation}\label{e5.16iv}
 \|v_n^{(2)}\|_{L^2(\R^N)}^2 \le O(\d),
 \end{equation}
 where $O(\d)\to0$ as $\d\to0$, uniformly with respect to $n$.

 Similarly, for $v^{(3)}$, we have 
  \begin{equation*}
 \begin{aligned}
 \|v_n^{(3)}\|_{L^2(\R^N)}^2 &=\|\fF v_n^{(3)}\|_{L^2(\R^N)}^2\\
 &\le \int_{\R^N}\left(\int_{-L}^L 1_{\{\b(v)(\k'',\k'')\le 2\d|\k''|^2 \}}(\xi) \,d\xi\right)\\
 &\qquad\qquad \left(\int_\R \left|\psi^{(1)}(\k)\psi^{(2)}(\k,\xi) \z^{(3)}(\k,\xi)\fF f_n(\k,\xi)\right|^2\,d\xi\right)\, d\k\\
  &\le \sup_{|\k''|=1}\left|\{\xi\in[-L,L] \,:\,\b(\xi)(\k'',\k'') \le  2\d\}\right| \|\phi\|_{L^\infty(\R)}^2 \|f_n\|_{L^2(\R^N\X\R)}^2,
 \end{aligned}
 \end{equation*}  
  where we  have used again Plancherel identity and Cauchy-Schwarz inequality, and again by a reasoning similar to that used for $v_n^{(2)}$ we arrive at 
  \begin{equation}\label{e5.16v}
   \|v_n^{(3)}\|_{L^2(\R^N)}^2 \le O(\d),
 \end{equation}
 where $O(\d)\to0$ as $\d\to0$,  uniformly with respect to $n$.
 
 Let us now consider $v_n^{(4)}$.  Let $\z\in C_c^\infty(\R)$ be such that $\z(\xi)=1$, for $\xi\in[-L,L]$, and $\z(\xi)=0$ for $\xi>L_0$, for some $L_0>L$.  We then have
 \begin{equation}\label{e5.16vi}
 v_n^{(4)}= \int_{\R} \z(\xi) \fF^{-1}\Big( \psi^{(1)}(\k)\psi^{(2)}(\k,\xi)\psi^{(3)}(\k,\xi)\fF f_n\Big) (y,\xi) \, d\xi
 \end{equation}
 Since $\psi^{(2)}\psi^{(3)}$ vanishes on the null set  of the symbol 
 $$
 \LL(\k,\xi):= i\tilde\a'(\xi)\cdot \tilde\k' +\b(\xi)(\k'',\k'')
 $$ 
 we may use the equation \eqref{e5.14} to write
\begin{equation*}
  \psi^{(1)}(\k)\psi^{(2)}(\k,\xi)\psi^{(3)}(\k,\xi)\fF f_n= \frac{\psi^{(1)}(\k)\psi^{(2)}(\k,\xi)\psi^{(3)}(\k,\xi) (1+|\k|^2)^{1/2} }{\LL(\k,\xi)} (1-\po_\xi^2) \fF h_n.
  \end{equation*}  
      
 We now prove that $v_n^{(4)}$ is relatively compact in $L_\loc^1$. We have
 \begin{equation}\label{e5.16vi'}
v_n^{(4)}=  \int_{\R} \z(\xi) \fF^{-1}\Big( \frac{\psi^{(1)}(\k)\psi^{(2)}(\k,\xi)\psi^{(3)}(\k,\xi) (1+|\k|^2)^{1/2} }{\LL(\k,\xi)} (1-\po_\xi^2) \fF h_n\Big) (y,\xi) \, d\xi.
\end{equation} 
Performing an integration by parts in \eqref{e5.16vi'} we obtain
\begin{equation} \label{e5.16vii}
\begin{aligned}
 & v_n^{(4)}= \\
&\int_{\R} (\z(\xi)-\z''(\xi)) \fF^{-1}\Big( \frac{\psi^{(1)}(\k)\psi^{(2)}(\k,\xi)\psi^{(3)}(\k,\xi) (1+|\k|^2)^{1/2}}{\LL(\k,\xi)}  \fF h_n \Big) \,d\xi \\
&+\int_{\R}\z'(\xi) \fF^{-1} \Big(  \po_\xi \left[ \frac{\psi^{(1)}(\k)\psi^{(2)}(\k,\xi)\psi^{(3)}(\k,\xi) (1+|\k|^2)^{1/2}}{\LL(\k,\xi)}\right] \fF h_n \Big) \,d\xi\\
&-\int_{\R}\z(\xi) \fF^{-1} \Big(  \po_\xi^2 \left[ \frac{\psi^{(1)}(\k)\psi^{(2)}(\k,\xi)\psi^{(3)}(\k,\xi) (1+|\k|^2)^{1/2}}{\LL(\k,\xi)}\right] \fF h_n \Big) \,d\xi
\end{aligned}
\end{equation}
So, let us define
$$
\begin{aligned}
m_1(\k,\xi)&:=  \frac{\psi^{(1)}(\k)\psi^{(2)}(\k,\xi)\psi^{(3)}(\k,\xi) (1+|\k|^2)^{1/2}}{\LL(\k,\xi)} \\
m_2(\k,\xi)&:= \po_\xi \left[ \frac{\psi^{(1)}(\k)\psi^{(2)}(\k,\xi)\psi^{(3)}(\k,\xi) (1+|\k|^2)^{1/2}}{\LL(\k,\xi)}\right] \\
m_3(\k,\xi)&:= \po_\xi^2 \left[ \frac{\psi^{(1)}(\k)\psi^{(2)}(\k,\xi)\psi^{(3)}(\k,\xi) (1+|\k|^2)^{1/2}}{\LL(\k,\xi)}\right] 
\end{aligned}
$$ 
We are going to show that $m_1(\k,\xi)$, $m_2(\k,\xi)$  and $m_3(\k,\xi)$ are $L^p$ multipliers in $\R^N$, uniformly in $\xi\in[-L_0,L_0]$. 
For that, we are going to apply the multidimensional extension of Marcinkiewicz multiplier theorem as stated in \cite{Duo}, chapter 8:
Let $m$ be differentiable in all quadrants of $\R^n$ and satisfy
$$
\sup_{i_1,\cdots,i_k}\int_{I_{i_1}\X\cdots\X I_{i_k}}\left|\frac{\po^k m}{\po\k_{i_1}\cdots\po\k_{i_k}}(\k)\right|\,d\k_{i_1}\cdots\,d\k_{i_k}<\infty,
$$
where the $I_j$'s are dyadic intervals in $\R$ and the set $\{i_1,\cdots,i_k\}$ runs over all the subsets $\{1,\cdots, n\}$ containing $k$ elements, $1\le k\le n$.  Clearly, it suffices to show that
$$
\sup_{\k\in\R^n} \k_{i_1}\cdots\k_{i_k} \left|\frac{\po^k m}{\po\k_{i_1}\cdots\po\k_{i_k}}(\k)\right| <\infty,
$$
for all such $\{i_1,\cdots,i_k\}$. 

First, we observe that 
$$
|m_1(\k,\xi)|\le C,
$$
with $C>0$ independent of $(\k,\xi)$. This follows from the fact that in the region where $|\k|>\g$,  $|\tilde\a'(\xi)\cdot \tilde\k'|> \d|\tilde\k'|$ and $\b (\xi)(\k'',\k'')>\d|\k''|^2$, it is not difficult to check that
$|\LL(\k,\xi)|\ge C|\k|$, for some constant $C>0$. So the boundedness for $m_1$ follows from the boundedness of $\psi^{(1)}\psi^{(2)}\psi^{(3)}$. Similarly, using the same reasoning, 
and the fact that the $\xi$-derivatives of $\psi^{(2)}(\k,\xi)$ and $\psi^{(3)}(\k,\xi)$ are uniformly bounded in $(\k,\xi)$, we also deduce that   
$$
|m_2(\k,\xi)|\le C,\quad |m_3(\k,\xi)|\le C,
$$
with $C>0$ independent of $(\k,\xi)$.  
 
 Now, let us analyze $\k_{i} \po_{\k_i}m_1$, for $i\in \{0,1,\cdots,N'+N''\}$. We claim that these expressions are bounded. First, if the derivative hits $\psi^{(1)}$ then the boundedness is clear since $\po_{\k_i}\psi^{(1)}$ has support in $|\k|\le 2\d$. On the other hand,
 if $\k_i\in \{0,\cdots,N'\}$ and the derivative hits $\psi^{(2)}$, then it is easy to see also that the expression is bounded since the derivation of the argument of $\psi^{(2)}$ multiplied by $\k_i$ is bounded. Also, if the derivative $\po_{\k_i}$ hits $1/\LL(\k,\xi)$ then it is clear that 
 the derivative of $1/\LL(\k,\xi)$ multiplied by $\k_i$ is bounded, where  we use that the support of $m_1$ is in a region where 
 $|\k|>\g$,   $|\tilde\a'(\xi)\cdot \tilde\k'|> \d|\tilde\k'|$ and $\b (\xi)(\k'',\k'')>\d|\k''|^2$. Analogously, we see that if $i\in\{N'+1,\cdots, N'+N''\}$,
 $\k_{i} \po_{\k_i}m_1$ is bounded uniformly in $(\k,\xi)\in\R^N\X\R$.  Similarly,  we prove $\k_{i} \po_{\k_i}m_1$, $i=0,\cdots,N$,  is bounded uniformly in $(\k,\xi)\in\R^N\X\R$.  In this way, we may check the hypotheses of the extended Marcinkiewcz multiplier theorem for $m_i$, $i=1,2,3$,  and conclude that they are satisfied uniformly for $\xi\in[-L_0,L_0]$. Hence, we deduce that
 \begin{equation}\label{efinal}
 \|v_n^{(4)}\|_{L^p(\R^N)}\le C\|\z\|_{C^2([-L_0,L_0])} \|h_n\|_{L^p(\R^N\X\R)}.
 \end{equation}
 Therefore, we conclude the proof of the lemma as follows. Given $\ve>0$ we may choose $\gamma>0$ and $\d>0$ such that 
 $\|v_n^{(i)}\|_{L^2(\R^N)}<\ve$, $i=1,2,3$, uniformly in $n\in\N$.  Then, making $n\to\infty$ we get that $v_n^{(4)}$ converges to 0 in $L^p(\R^N)$. Then, since $\ve>0$  is arbitrary, we see that $v^{(i)}\to0$, in $L_\loc^1(\R^N)$, for $i=1,2,3$, which concludes the proof.              
 
 \end{proof}

{\em Conclusion of the Proof of Theorem~\ref{T:5.1}.}
By Lemma~\ref{L:03.4}, it follows that for all $(t_*,\wh y_*)\in\E$, $\ff^\tau(t_*,\wh y_*, \xi)$ is a $\chi$-function a.e.\  in $\Om\X(-L,L)$, and $\E\subset\Sigma$ has total measure, by Lemma~\ref{3.lemma1} and Lemma~\ref{3.lemma2}. Thus, $\ff^\tau(\cdot,\cdot, \cdot)$ is a $\chi$-function a.e.\ in $\Om\X\Sigma\X(-L,L)$.  Hence, from Lemma~\ref{2.criterion} we conclude that $\ff^\tau$ is a strong trace and integrating in $\xi$ we  arrive at the desired conclusion for $u$ on $\cO_0$ and $\Gamma_0$.  Covering $\po\cO'$ with a finite set 
$\{\S_\a'\}_{\a\in I_0}$,  each $\S_\a'$  being the graph of a Lipschitz function, we then  finally deduce \eqref{eT5.1-1}. 

It remains to prove \eqref{eT5.1-2}.  From the essential strong convergence of $ \ff_\psi(\cdot, \cdot, s,\cdot)$ in $ L^1(\Om\X\Sigma\X(-L,L))$,  it follows that, given any sequence $s_n\to0$  in $(0,1)\setminus\mathcal{N}$, with $\mathcal{N}$ of null measure,  we can obtain
a subsequence still denoted $s_n$  such that  $ \ff_\psi(\om, \cdot, s_n,\cdot)\to \ff^\tau(\om,\cdot,\cdot)$ in $ L^1(\Sigma\X(-L,L))$ for $\om$ in a subset of total measure of $\Om$. Thus, by Lemma~\ref{2.criterion0}, we deduce that $\ff^\tau(\om,\cdot,\cdot)$ is
a $\chi$-function for $\om$ in a subset of total measure of $\Om$. By Lemma~\ref{L:3.02} we conclude, using again Lemma~\ref{2.criterion0}, that   $\esslim \ff_\psi(\om, \cdot, s,\cdot)=\ff^\tau(\om,\cdot,\cdot)$ in $L^1(\Sigma\X(-L,L))$  
for $\om$ in a subset of total measure of $\Om$.  Again integrating in $\xi$, we arrive at the desired conclusion for $u$ on $\cO_0$ and $\Gamma_0$, and so by   covering $\po\cO'$ with a finite set 
$\{\S_\a'\}_{\a\in I_0}$ as above, we then  finally deduce \eqref{eT5.1-2}.

\end{proof}

\section{Uniqueness}\label{S:5}

 \begin{theorem}[Uniqueness]\label{T:2.2}  If $u_1$, $u_2$ are kinetic solutions to \eqref{e1.1}--\eqref{e1.4}
 with initial data $u_{1,0}$ and $u_{2,0}$, respectively, then for all $t\in[0,T]$ we have 
 \begin{equation}\label{e2.1}
\bbE \int_{ \cO}|u^\pm(t,x)-v^\pm(t,x)|\,dx\le \bbE \int_{\cO}|u_0(x)-v_0(x)|\,dx,
\end{equation}
for some $C>0$ depending only on the data of the problem.  
\end{theorem}

\begin{proof} The proof combines Theorem~\ref{T:2.1} and Theorem~\ref{T:5.1}. Indeed, by Theorem~\ref{T:2.1} for any nonnegative test functions $\theta\in C_c^\infty((0,T))$ and $\psi\in C_c^\infty(\cO'\times\mathbb{R}^{d''})$, we have a.s. that
\begin{multline}\label{e3.unique1*}
  -\bbE\int_0^T\int_{\cO}\left|u^{\pm}(s,x)-v^{\pm}(s,x)\right|\theta'(s)\psi(x)\, dx\, dy \\
  \le  \bbE\int_0^T\int_{\cO}K_{x''}(u(s,x),v(s,x))\cdot \nabla\psi_1(x)\, dx\, dy\, ds 
 \end{multline} 
Then, it suffices to take a suitable sequence of test functions $\theta_n(s)$ and $\psi_\delta(x)=\zeta_\delta'(x')$ in inequality \eqref{e3.unique1*}, where $\zeta_\delta'(x')$ is a $\cO'$-boundary layer sequence and $\theta_n$ converges to the indicator function of $(0,t)$, so that, by virtue of \eqref{e1.4}, we may use the strong trace property established in Section~\ref{S:4} in order to send $\delta\to 0$ first and then, recalling Proposition~\ref{P;2.1}, also also take $n\to\infty$ to obtain \eqref{e2.1}.
\end{proof}

As in \cite{DV, DHV} we have the following corollary whose proof follows by the same lines as in the mentioned references.

\begin{corollary}[Continuity in time] \label{C:2.1} Let $u$ be a kinetic solution to \eqref{e1.1}--\eqref{e1.3}. Then there exists a representative of $u$ which has almost surely continuous 
trajectories in $L^p(\Om)$, for all $p\in[1,\infty)$.
\end{corollary}

\section{Existence, part one: The first approximate problem}\label{S:6}

In order to tackle the question of existence of solutions to \eqref{e1.1}--\eqref{e1.4}, we approximate it with solve the following  approximation of problem \eqref{e1.1}--\eqref{e1.3},
\begin{align}
& du^\ve+\nabla\cdot \Abf^\ve(u^\ve)\,dt - D_{x''}^2:\Bbf^\ve(u^\ve)\,dt-\ve\Delta u^\ve\,dt = \Phi^\ve(u^\ve)\,dW(t), \label{e6.1}\\
&u^\ve(0,x)=u_0^\ve(x),\qquad x\in\cO,\label{e6.2}\\
& u^\ve=u_b^\ve(t,x),\qquad t>0,\  x\in \cO'\X\po\cO'', \label{e6.3}\\
&\ve \po_\nu  u^\ve=\Abf^\ve(u^\ve)\cdot\nu, \qquad t>0, x\in \po\cO'\X\cO'', \label{e6.3'}
\end{align}
where $u_0^\ve$ is a smooth approximation of $u_0$, $u_0^\ve\in L^\infty(\Om, C_c^\infty(\cO))$, $u_\text{min}\le u_0^\ve\le u_\text{max}$, a.s., $u_b^\ve$ is a smooth approximation of $u_b$, 
$\Phi^\ve$ is a suitable Lipschitz approximation of $\Phi$ satisfying \eqref{e1.4*}  uniformly, with
 $g_ k^\ve$ and $G^\ve$  as in the case $\ve=0$, $g_k^\ve$ smooth satisfying \eqref{e1.8'}.  Moreover, $g_k^\ve\equiv 0$ for $k\ge 1/\ve$. 
Finally, $\Abf^\ve\in C^2(\R;\R^d)$, $\Bbf^\ve\in C^2(\R;\bbM^{d''})$, $\Abf^\ve(u)=\Abf(u)$, $\Bbf^\ve(u)=\Bbf(u)$ for $u\in[u_\text{min},u_{\text{max}}]$,  and, setting $\abf^\ve=(\Abf^\ve)'$, $\bbf^\ve=(\Bbf^\ve)'$ we assume that $\abf^\ve \in L^\infty(\R;\R^d)$, $\bbf^\ve\in L^\infty(\R;\bbM^{d''})$. The latter assumption will be justified later on when we will prove that the solution of \eqref{e6.1}--\eqref{e6.3'},
$u^\ve(t,x)$, satisfies  $u_{\text{min}}\le u^\ve(t,x)\le u_{\text{max}}$, $(t,x)\in (0,T)\X\cO$. 

Nevertheless, the establishment of existence of solutions to \eqref{e6.1}--\eqref{e6.3'} presents also its problems due to its quasilinear character. Thus,  we approximate again \eqref{e6.1}--\eqref{e6.3'} by the following forth order nonlinear SPDE problem
\begin{align}
& du^{\ve,\mu}+\nabla\cdot \Abf^\ve(u^{\ve,\mu})\,dt - D_{x''}^2:\tilde \Bbf^\ve(u^{\ve,\mu})\,dt-\ve\Delta_{x'} u^{\ve,\mu}  +\mu\Delta_{x''}^2 u^{\ve,\mu}  \label{e6.4}\\
&\qquad\qquad\qquad\qquad\qquad= \Phi^\ve(u^{\ve,\mu})\,dW(t),\nonumber\\
&u^{\ve,\mu}(0,x)=u_0^\ve(x),\qquad x\in\cO,\label{e6.5}\\
& u^{\ve,\mu}=u_b^\ve(t,x),\qquad t>0,\  x\in \cO'\X\po\cO'', \label{e6.6}\\
&\ve\po_{\nu}u^{\ve,\mu}=\Abf(u^{\ve,\mu}),\qquad t>0,\ x\in\po\cO'\X\cO'', \label{e6.7}\\
&\po_\nu u^{\ve,\mu}=0,\qquad t>0, \ x\in\cO'\X\po\cO'', \label{e6.8}
\end{align}
where $\mu > 0$ is an artificial viscosity. Notice that we needed to introduce a new boundary condition \eqref{e6.8} due to the fourth order of the biharmonic operator $\mu\Delta_{x''}^2$.

For the sake of clarity, we will call the problem \eqref{e6.4}--\eqref{e6.8} \textit{the first approximate problem}, whereas \eqref{e6.1}--\eqref{e6.3'} \textit{the second approximate problem}.

 We are going to establish the existence and uniqueness of solutions \eqref{e6.4}--\eqref{e6.8} by means of a fixed-point argument envolving the Duhamel formula. More precisely, let $A: D(A) \subset L^2(\cO'\times \cO'') \rightarrow L^2(\cO' \times \cO'')$ be the operator given by
$$
\begin{cases}
D(A) &= \{ u \in L^2(\cO' \times \cO'')\,:\, u\in L^2(\cO''; H^2(\cO')) \cap L^2(\cO'; (H_0^2 \cap H^4)(\cO''))  \\  &\quad\quad\quad\quad\quad\quad\quad\quad\quad \text{ and, in the sense of traces, } \partial_\nu u = 0 \text{ on } \partial \cO' \times \cO'' \}, \\
A &= -\ve\Delta_{x'+x''} + \mu\Delta_{x''}^2,
\end{cases}
$$
and $S(t) = \exp\{ - t A \}$.  

Let $\cP$ be the predictable $\s$-algebra in $\Omega\X[0,T]$. Let
$$
 \E = L^2(\Omega; C([0,T];L^2(\cO)) \cap L^2(\Omega \times [0,T], \cP; L^2(\cO'; H^2(\cO'')) \cap L^2(\cO'';H^1(\cO'))).
 $$
For $v\in\cE$, we define 
\begin{multline}
  K[v](t)=K^{\ve,\mu}[v](t) =\, w^v(t) + \widetilde{u_b^\ve}(t) + S(t)u_0^\ve- \int_0^t S(t-s) \nabla\cdot \Abf^\ve(v(s))\,ds \\
   + \int_0^t S(t-s) D_{x''}^2:\tilde \Bbf^\ve(v(s))\,dt + \int_0^t S(t-s) \Phi^\ve (v(s)) dW(s). \label{e6:uvemu}
\end{multline}
Here $w^v(t,x)$ is the solution  to 
\begin{align}
\frac{\partial w^v}{\partial t}(t,x) &= \ve \Delta_{x',x''} w^v(t,x) - \mu \Delta_{x''}^2 w^v(t,x)\quad \text{ in } \{ 0 < t < T \} \times \cO, \label{e6:1wu} \\
\ve \frac{\partial w^v}{\partial \nu}(t,x) &= \Abf^\ve(v(t,x)) \cdot \nu\qquad\qquad \text{ on } \{ 0 < t < T \} \X \partial \cO' \X \cO'', \label{e6:2wu} \\
w^v(t,x) &= \frac{\partial w^v}{\partial \nu}(t,x) = 0 \qquad\qquad \text{ on }  \{ 0 < t < T \} \X  \cO' \X \partial\cO'', \label{e6:3wu}\\
w^v(0, x) &= 0 \qquad\qquad\qquad\qquad\quad  \text{ on }  \{ 0 = t \} \X  \cO, \label{e6:4wu}
\end{align}
and $\widetilde{u_b^\ve}(t,x',x'')$ is the solution to
\begin{align}
\frac{\partial \widetilde{u_b^\ve}}{\partial t}(t,x) &= \ve \Delta_{x',x''} \widetilde{u_b^\ve}(t,x) - \mu \Delta_{x''}^2 \widetilde{u_b^\ve}(t,x) \quad \text{ in } \{ 0 < t < T \} \times \cO' \times \cO'', \label{e6:1ub} \\\
\frac{\partial \widetilde{u_b^\ve}}{\partial \nu}(t,x) &= 0  \qquad\qquad\text{ on } \{ 0 < t < T \} \X \partial \cO' \X \cO'', \label{e6:2ub} \\\
\widetilde{u_b^\ve}(t,x) &= u_b^\ve(t,x)   \qquad\qquad                      \text{ on }  \{ 0 < t < T \} \X  \cO' \X \partial\cO'',\label{e6:3ub} \\\
\frac{\partial \widetilde{u_b^\ve}}{\partial \nu}(t,x) &= 0   \qquad\qquad        \text{ on }  \{ 0 < t < T \} \X  \cO' \X \partial\cO'',\label{e6:4ub} \\\
\widetilde{u_b^\ve}(0, x) &=0 \qquad\qquad \text{ on }  \{ 0 = t \} \X  \cO. \label{e6:5ub} 
\end{align}
The sense in which the solutions of the problems \eqref{e6:1wu}-\eqref{e6:4wu} and \eqref{e6:1ub}-\eqref{e6:5ub} should be understood is explained subsequently.

First of all, as we have shown in Theorem \ref{thmspectralA}, $A$ is indeed a nonnegative self-adjoint operator, and consequently $S(t)$ is well-defined and possesses several good smoothing properties. 

\subsection{On $w^v(t)$}\label{SS:6.1}

 Henceforth $H_A^{1/2}$ will denote the space $L^2(\cO'; H_0^2(\cO''))$ $\cap$ $L^2(\cO'';$ $H^1(\cO')))$  (see Appendix \ref{appendixA}).

\begin{definition}\label{D:PWS}

Let $v \in L^2(\Om\X(0,T); H_A^{1/2})$ be predictable. We say that $w^v$ is a \textit{pathwise weak solution} of \eqref{e6:1wu}--\eqref{e6:4wu}, if 
$$
w^v \in L^2(\Om\X(0,T); H_A^{1/2}) \cap L^2(\Om;C([0,T];L^2(\cO))) \cap L^2(\Om;H^1((0,T); H_A^{-1/2})),
$$
is predictable and,  for every $v \in H_A^{1/2}$ and  almost every $0<t<T$, a.s.
\begin{align*}
\Big\langle \frac{\partial w^v}{\partial t}&(t), \varphi \Big\rangle_{H_A^{-1/2},H_A^{1/2}} + \ve \int_\cO \nabla w^v(t,x) \cdot \nabla \varphi(x) dx + \mu \int_\cO \Delta_{x''} w^v(t,x) \Delta_{x''} \varphi(x) dx \\&= \int_{\cO''}\int_{\partial \cO'} \Abf^\ve(v(t,x',x'')) \cdot \nu(x',x'') \varphi(x',x'') \,d\mathcal{H}^{d'-1}(x') dx''.
\end{align*}
\end{definition}

 We are going to use the following result by J.-L. Lions, whose proof may be found in \cite{LM} and statement we take from \cite{Br}.

\begin{theorem} \label{6:thmLions}
  Let $H$ be Hilbert space with scalar product $(\quad,\quad)$ and norm $\|\quad\|_H$, and identify $H^*$ with $H$. Let also $V$ be another Hilbert space with norm $\Vert \quad \Vert_V$, for which $V \subset H$ with dense and continuous injection, so that we have the triplet
  $$V \subset H \subset V^*.$$
  
  Let $T>0$ be fixed; and suppose that for a.e.\  $0<t<T$ we are given a bilinear form  $a(t;u,z) : V \times V \rightarrow \R$ satisfying the following properties:
  \begin{enumerate}
    \item For every $u, z \in V$, the function $t \mapsto a(t;u,z)$ is measurable;
    \item $|a(t;u,z)| \leq M \Vert u \Vert_V \Vert z \Vert_V$ for a.e. $0<t<T$, $\forall\,u,z \in V$, and where $M$ is a constant;
    \item $a(t;u,u) \geq \alpha \Vert u \Vert_V^2 - C\Vert u \Vert_H^2$ for a.e. $0<t<T$, $\forall \, u \in V$, and where $\alpha$ and $C$ are positive constants.
  \end{enumerate}
  
  Then for every $F \in L^2(0,T;V^*)$ and $u_0 \in H$, there exists a unique function 
  $$
  u \in L^2(0,T; V) \cap C([0,T];H) \cap H^1(0,T;V^*),
  $$
  such that $u(0) = u_0$ and
  \begin{equation}\label{e.bilform}
  \Big\langle \frac{du}{dt}(t), z \Big\rangle_{V^*,V} + a(t;u(t),z) = \langle F(t), z\rangle_{V^*, V} \text{ for a.e. } 0<t<T \text{ and } \forall z \in V.
  \end{equation}
\end{theorem}

We apply this theorem as follows. Let $H = L^2(\cO)$ and $V = H_A^{1/2}$, so that $V^* = H_A^{-1/2}= L^2(\cO'; H^{-2}(\cO'')) + L^2(\cO''; H^1(\cO')^*)$. Our bilinear form $a(t,u,z)$ will be 
$$
a(t; u, z) = \ve \int_\cO \nabla u \cdot \nabla z\,dx + \mu \int_\cO \Delta_{x''} u \Delta_{x''} z \, dx,
$$
which is actually time-independent. Finally, let $F(t) \in L^2(0,T; H_A^{-1/2})$ be
\begin{equation}\label{e.f100}
\langle F(t), \varphi\rangle_{H_A^{-1/2}, H_A^{1/2}} = \int_{\cO''} \int_{\partial \cO'} \Abf^\ve(v(t,x',x'')) \cdot \nu(x',x'') \varphi(x',x'') \,d\mathcal{H}^{d'-1}(x') dx'', 
\end{equation}
which makes sense because of the trace properties of Sobolev functions. 

For each fixed $\om\in\Om$,  the conditions in Theorem \ref{6:thmLions} are immediately verified, so  existence and uniqueness of $w^v$ for each fixed $\om\in\Om$ follows.
Now the proof of Theorem~\ref{6:thmLions} can be made by the Garlerkin method.  So, since $v \in L^2(\Omega \times [0,T], \cP; H_A^{1/2})$ is predictable, $f$ given by \eqref{e.f100} is also predictable and so are its finite dimensional projections. Therefore, the Galerkin approximations, which are solutions of finite dimensional ODEs obtained as projections of \eqref{e.bilform},
are also predictable. The convergence of the Galerkin approximations is obtained alongside a uniform estimate in  $L^2(\Om\X(0,T), \cP; H_A^{1/2})$ so that in the limit we obtain a pathwise weak solution in the sense of Definition~\ref{D:PWS}.  

Moreover, if $v_1$ and $v_2 \in L^2(\Omega \times [0,T], \cP; H_A^{1/2})$ are predictable, then, a.s. and for a.e. $0 < t < T$,
\begin{align*}
  \Big\langle& \frac{\partial w^{v_1}}{\partial t}(t) - \frac{\partial w^{v_2}}{\partial t}(t), w^{v_1}(t) - w^{v_2}(t) \Big\rangle_{H_A^{-1/2},H_A^{1/2}} \\&+ \ve \int_\cO |\nabla w^{v_1}(t,x) - \nabla w^{v_2}(t,x)|^2 \,dx + \mu \int_\cO | \Delta_{x''} w^{v_1}(t,x) -  \Delta_{x''} w^{v_2}(t,x) |^2\, dx \\&= \int_{\cO''}\int_{\partial \cO'} (\Abf^\ve(v_1(t,x)) - \Abf^\ve(v_2(t,x)))\cdot \nu(x) \\ &\quad\quad\quad\quad\quad\quad\quad\quad\quad\quad\quad\quad\quad\quad\quad\quad\quad\quad(w^{v_1}(t,x) - w^{v_2}(t,x)) \,d\mathcal{H}^{d'-1}(x')\, dx''.
\end{align*}
Thus, a.s.,
\begin{align*}
  \frac{1}{2} \sup_{0\leq t \leq T} \| &w^{v_1}(t) - w^{v_2}(t) \|_{L^2(\cO)}^2 + \ve \int_0^T   \| \nabla w^{v_1}(t) - \nabla w^{v_2}(t) \|_{L^2(\cO)}^2\, dt \\ + &\mu\int_0^T   \| \Delta_{x''} w^{v_1}(t) - \Delta_{x''} w^{v_2}(t) \|_{L^2(\cO)}^2\, dt \\ &\quad\quad \leq \Lip(\Abf^\ve) \int_0^T\int_{\partial \cO' \times \cO''} |w^{v_1}(t,x) - w^{v_2}(t,x) | \\ &\quad\quad\quad\quad\quad\quad\quad\quad\quad\quad\quad |v_1(t,x) - v_2(t,x)| \, d\mathcal{H}^{d'-1}(x')dx'' dt.
\end{align*}
As $\Vert f \Vert_{L^2(\partial \cO')} \leq C \Vert f \Vert_{H^1(\cO')}$ for some $C>0$, Young and Gronwall inequalities yield
$$\bbE\Vert w^{v_1} - w^{v_2} \Vert_{C([0,T];L^2(\cO)) \cap L^2(0,T; H_A^{1/2})}^2 \leq C \bbE \Vert v_1 - v_2 \Vert_{L^2(0,T;H_A^{1/2})}^2.$$

In conclusion, we summarise the results above as follows.

\begin{lemma} \label{6:lemmawv1}
For any predictable $v \in L^2(\Omega \times [0,T], \cP; H_A^{1/2})$, there exists a unique pathwise weak solution $w^v \in L^2(\Omega; C([0,T]; L^2(\cO))) \cap L^2(\Omega \times [0,T], \cP; H_A^{1/2})$ of \eqref{e6:1wu}--\eqref{e6:4wu}. Moreover,  the mapping $v \mapsto w^v$ is continuous.
\end{lemma}

\subsection{On $\widetilde{u_b^\ve}$} \label{SS:6.2} 

Let us now analyze $\widetilde{u_b^\ve}$. At this stage of approximation, we may and will assume that $u_b^\ve \in L^2(\Omega; C_c^\infty((0,T) \times \cO' \times \partial\cO''))$ is predictable. Thus, straightening out the boundary, using partition of unity,  etc., one can extend $u_b^\ve$ to some predictable $f^\ve \in L^2(\Omega; C_c^4((0,T) \times \cO' \times \R^{d''}))$, in a way that a.s.
$$\partial_\nu f^\ve  = 0 \text{ on } (0,T) \times \partial \cO' \times \cO'' \text{ and on }  (0,T) \times  \cO' \times \partial\cO''.$$
Hence, $z = \widetilde{u_b^\ve} - f^\ve$ should satisfy
\begin{equation}\label{e.z100}
\begin{aligned}
\frac{\partial z}{\partial t} - \ve \Delta z + \mu \Delta_{x''}^2 z = -\Big(\frac{\partial}{\partial t} - \ve \Delta + \mu \Delta_{x''}^2 \Big) f^\ve  &\text{ in } (0,T) \times \cO, \\
\frac{\partial z}{\partial \nu}(t,x) = 0 &\text{ on } (0,T)\X \partial \cO' \X \cO'', \\
z(t,x) = \frac{\partial z}{\partial \nu}(t,x) = 0 &\text{ on }  (0,T)\X  \cO' \X \partial\cO'',\\
z(0, x) =  0 &\text{ on }  \{ t= 0 \} \X  \cO.
\end{aligned}
\end{equation}

Existence and uniqueness of a solution to  problem \eqref{e.z100} is obtained from Theorem~\ref{6:thmLions} with $H=L^2(\cO)$, $V=H_A^{1/2}$, 
$$
a(t; u, z) = \ve \int_\cO \nabla u \cdot \nabla z\,dx + \mu \int_\cO \Delta_{x''} u \Delta_{x''} z \, dx,
$$
as before, and  
$$
\langle F(t),\varphi\rangle_{H_A^{-1/2},H_A^{1/2}} := \langle -\Big(\frac{\partial}{\partial t} - \ve \Delta + \mu \Delta_{x''}^2 \Big) f^\ve , \varphi\rangle_{L^2(\cO)}.
$$
 Therefore, we have proved the following result, where, the assertion about the predictability, as before, is a consequence of  the proof of Theorem~\ref{6:thmLions} through Garlerkin method.
 
 \begin{lemma} \label{6:lemmau_b}
 Assume that $u_b^\ve \in L^2(\Omega; C_c^\infty((0, T) \times \cO' \times \partial\cO''))$ is predictable. Then there is a unique solution of problem \eqref{e6:1ub}--\eqref{e6:5ub}, 
 $\widetilde{u_b^\ve} \in L^2(\Omega \times [0,T], \cP;$ $L^2(\cO';H^2(\cO''))$ $\cap$ $L^2(\cO'';H^1(\cO'))$.
 \end{lemma}

 \subsection{On the mapping $K[v]$.} \label{SS:6.3}
 
In order to avoid cumbersome notation,  in this subsection we drop the superscripts $\ve$ and $\mu$. 

\begin{lemma} \label{6:Kcontinuity}
   $K : \E \to \E$ is well-defined and continuous.
 \end{lemma}
 
 \begin{proof}  For $v\in\cE$, let us decompose  $Kv$, defined in \eqref{e6:uvemu}, into six parts:
$$
(Kv)(t) = (K_0v)(t)+ (K_1 v)(t) + (K_2 v)(t) + (K_3 v)(t) - (K_4 v)(t) + (K_5 v)(t),
$$
where 
\begin{align*}
  (K_0v)(t) &= w^v(t), \\
  (K_1v)(t) &= \widetilde{u_b}(t), \\
  (K_2v)(t) &= S(t) u_0,\\
  (K_3v)(t) &= \int_0^t S(t-s) D_{x''}^2 : \Bbf(v(s)) \, ds,\\
  (K_4v)(t) &= \int_0^t S(t-s) \nabla \cdot \Abf (v(s)) \, ds, \text{ and } \\
  (K_5v)(t) &= \int_0^t S(t-s) \Phi (x,v(s)) \, dW(s).
\end{align*}
Observe that $K_0$ and $K_1$ were already addressed  in Lemmas~\ref{6:lemmawv1} and \ref{6:lemmau_b}. Furthermore,  $K_2v$ is independent  of $v$ and,   from the basic theory of semigroups of linear operators, $K_2v$ is clearly seen to be an element of $\E$. The other terms will now be investigated individually.

\textit{The analysis of $K_3$}. Let us verify that $K_3 : \E \rightarrow \E$ is continuous. Note that,  if $h \in L^2(\Omega;$ $L^2(0,T; L^2(\cO))$, and if $d'+1 \leq j, k \leq d$, then, as distributions,
    \begin{align*}
      \frac{\partial h}{\partial x_j} &\in L^2\big(\Omega; L^2(0,T;L^2(\cO';H^{-1}(\cO'')))\big),  \\
      \frac{\partial^2 h}{\partial x_j \partial x_k} &\in L^2\big(\Omega; L^2(0,T;L^2(\cO';H^{-2}(\cO'')))\big),
    \end{align*}
with
    \begin{align*}
    \mathbb{E}\int_{\cO'} \int_0^T \Big\Vert \frac{\partial h}{\partial x_j } (s) \Big\Vert_{H^{-1}(\cO'')}^2 ds\,dx'  &\leq \mathbb{E} \int_0^T \Vert h (s) \Vert_{L^2(\cO)}^2 \,ds, \\
    \mathbb{E} \int_{\cO'}\int_0^T \Big\Vert \frac{\partial^2 h}{\partial x_j \partial x_k} (s) \Big\Vert_{H^{-2}(\cO'')}^2 ds\,dx'  &\leq \mathbb{E} \int_0^T \Vert h (s) \Vert_{L^2(\cO)}^2 \,ds.
\end{align*}  
 In particular, $D_{x''}^2: \Bbf(v(s))\in H_A^{-1/2} = L^2(\cO';H^{-2}(\cO'')) + L^2(\cO''; [H^1(\cO')]^*)$, so we have that a.s.
    \begin{align}
        \Big\Vert \int_0^t S(t-s) D_{x''}^2 : \Bbf(v(s))\, ds & \Big\Vert_{L^2(\cO)} \nonumber \\ &\leq \int_0^t \Vert S(t-s) D_{x''}^2 : \Bbf(v(s)) \Vert_{L^2(\cO)} \, ds \label{e6.1000} \\
        &\leq C \int_0^t \Big( 1 + \frac{1}{(t-s)^{1/2}} \Big) \Vert  \Bbf(v(s)) \Vert_{L^2(\cO)} \, ds  \nonumber \\
        &\leq C \sup_{0\leq t \leq T} \Vert \, \Bbf(v(s)) \, \Vert_{L^2(\cO')}. \nonumber
    \end{align}
    where we have used Proposition~\ref{B:reg0} in the second inequality above. 
  Taking the $\sup_{0\le t\le T}$, squaring and taking expectation  in  \eqref{e6.1000},  we conclude that $v \in \E \mapsto K_3 v \in \cE_0:= L^2(\Omega;$ $C([0,T];$ $L^2(\mathcal{O})))$ is well-defined and 
  \begin{equation}\label{e6.1001}
 \| K_3v\|_{\cE_0}\le C \|\Bbf(v)\|_{\cE_0}.
 \end{equation}
Similarly,
\begin{equation}\label{e6.1002}
 \| K_3v_1-K_3v_2\|_{\cE_0}\le C \|\Bbf(v_1)-\Bbf(v_2)\|_{\cE_0}.
 \end{equation}

  Likewise, by Proposition \ref{B:reg1}, we see that almost surely
    \begin{align}
      \int_0^T \Big\Vert \int_0^t S(t-s) D_{x''}^2 : \Bbf(v(s))\, ds  \Big\Vert_{H_A^{1/2}}^2 &\, dt\nonumber\\
        &\leq C \,  \int_0^T \Vert \, D_{x''}^2: \Bbf(v(s)) \, \Vert_{H_A^{-1/2}}^2 \, ds  \nonumber\\
        &\leq C \,  \int_0^T \Vert \, D_{x''}^2: \Bbf(v(s)) \, \Vert_{L^2(\cO';H^{-2}(\cO''))}^2\, ds  \nonumber\\
        &\leq C \, \int_0^T \Vert \, \Bbf(v(s)) \, \Vert_{L^{2}(\cO)}^2 \, ds. \label{6:K3est2}
    \end{align}
    Thus we see that $K_3 v \in \cE_1:= L^2(\Omega;L^2(0,T; H_A^{1/2}))$, recalling that $H_A^{1/2}=L^2(\cO';H_0^2(\mathcal{O}''))\cap L^2(\cO''; H^1(\cO'))$, with
   \begin{equation}\label{e6.1003}
   \| K_3v\|_{\cE_1}\le C \|\Bbf(v)\|_{\cE_0}.
 \end{equation}
Similarly,
\begin{equation}\label{e6.1004}
 \| K_3v_1-K_3v_2\|_{\cE_1}\le C \|\Bbf(v_1)-\Bbf(v_2)\|_{\cE_0}.
 \end{equation}
 Therefore, we conclude that  $K_3:\cE\to \E$  with
     \begin{equation}\label{e6.1005}
   \| K_3v\|_{\cE}\le C \|\Bbf(v)\|_{\cE_0},
 \end{equation}
and
\begin{equation}\label{e6.1006}
 \| K_3v_1-K_3v_2\|_{\cE}\le C \|\Bbf(v_1)-\Bbf(v_2)\|_{\cE_0}.
 \end{equation}

     \textit{The analysis of $K_4$}. Given $v \in \E$, the contractivity of $S(t)$ gives
    $$\sup_{0 \leq t \leq T} \Vert K_4 v(t) \Vert_{L^2(\cO)}^2 \leq C_T \int_0^T \Vert \nabla \cdot \Abf(v(s)) \Vert_{L^2(\cO)}^2 \, ds,$$
whereas Proposition \ref{B:reg1} gives
    $$
    \int_0^T \Vert K_4 v(s) \Vert_{H_A^{1}}^2 \, ds \leq C \int_0^T \Vert \nabla \cdot \Abf(v(s)) \Vert_{L^2(\cO)}^2 \, ds.
    $$
  From this, one can easily show that $K_3 v \in \E$. Moreover, if $w$ is another element of $\E$, we see that
\begin{equation}
     \Vert K_4 v - K_4 w \Vert_{\E} \leq C \bbE \int_0^T \Vert \abf(v(s)) \cdot \nabla v(s) - \abf(w(s))\cdot \nabla w(s)  \Vert_{L^2(\cO)}^2 \, ds. \label{6:K4est1}
\end{equation}
    We claim that the inequality above implies that $K_4$ is continuous. Indeed, assume, by contradiction, that $K_3$ is not continuous. In this case, there would exist a sequence $(v_n)$ in $\E$     such that $v_n \rightarrow v$ in $L^2(\Omega \times [0,T]; H^1(\cO))$, but
    \begin{equation}
          \Vert K_4(v_n) - K_4(v) \Vert_{\cE} \geq \delta, \label{6:k4.2}
    \end{equation}
for some $\delta > 0$. On the other hand,  passing to a subsequence if necessary (see, e.g., theorem~4.9 in \cite{Br}), we could assume that     
$$\begin{cases}
v_n \rightarrow v               &\text{for almost every } x\in \cO, 0<t<T \text{ and } \omega \in \Omega, \\
\nabla v_n \rightarrow \nabla v &\text{for almost every } x\in \cO, 0<t<T \text{ and } \omega \in \Omega, \text{ and} \\
|v_n| + |\nabla_x v_n| \leq g   &\text{for some } g \in L_{t,x,\omega}^2.
\end{cases}$$
Consequently, since $\abf$ is bounded by assumption,  applying the dominated convergence to \eqref{6:K4est1}, we would see that
$$ \Vert K_4(v_n) - K_4(v) \Vert_{\cE} \rightarrow 0,$$
contradicting \eqref{6:k4.2}.

  \textit{Analysis of $K_5$}. Finally, let us study the stochastic term $K_5$. Fix $v \in \mathcal{E}$. That $(K_5 v)(t) \in L^2(\Omega; C([0,T];L^2(\mathcal{O})))$ is a consequence of the contractivity of the semigroup $S(t)$ and the well known maximal inequality for stochastic convolutions (see \cite{Tu, Ko1, Ko2}), from which we obtain
  \begin{equation*}
  \mathbb{E} \sup_{0\leq t \leq T}  \Big\Vert \int_0^t S(t-s) \Phi(v(s)) \,dW(s) \Big\Vert_{L^2(\cO)}^2\leq C\, \mathbb{E} \, \int_0^T \Vert \Phi(v(t)) \Vert_{L_2^0(\mathfrak{U}; L^2(\cO))}^2\,dt, 
  \end{equation*}
  so, using  \eqref{e1.4*}, we have
  \begin{equation}\label{e4.1001}
  \|K_5v\|_{\cE_0}^2 \le C(1+\|v\|^2_{L^2(\Om\X[0,T];L^2(\cO))}).
  \end{equation}
Similarly, using \eqref{e1.5*},
\begin{equation}\label{e4.1002}
 \|K_5v_1-K_5v_2\|_{\cE_0}^2 \le C(\|v_1-v_2\|^2_{L^2(\Om\X[0,T];L^2(\cO))}).    
 \end{equation}
 Moreoever, by Theorem \ref{B:reg2} (with $\alpha = 0$), and Theorem \ref{thmspectralA},
    \begin{align*}
        \mathbb{E} \int_0^T \Big\Vert \int_0^t S(t-s) & \Phi(v(s)) \, dW(s) \Big\Vert_{H_A^{1/2}}^2 \,ds\\ &\leq C \, \mathbb{E} \int_0^T \Vert \Phi(v(s)) \Vert_{L_2(\mathfrak{U}, L^2(\mathcal{O}))}^2 \,ds,
    \end{align*}
    so, using again \eqref{e1.4*}, we get
   \begin{equation}\label{e4.1003}
  \|K_5v\|_{\cE_1}^2 \le C(1+\|v\|^2_{L^2(\Om\X[0,T];L^2(\cO))}).
  \end{equation}
  and, similarly, using \eqref{e1.5*}
 \begin{equation}\label{e4.1004}
 \|K_5v_1-K_5v_2\|_{\cE_1}^2 \le C(\|v_1-v_2\|^2_{L^2(\Om\X[0,T];L^2(\cO))}).    
 \end{equation}
 Hence, \eqref{e4.1001}--\eqref{e4.1004}  imply that  $K_5:\cE\to\cE$ is continuous.

\textit{Conclusion.} Since $K_0, K_1, \ldots, K_5$ are all continuous mappings  $\E\to\cE$, the lemma is proven.
\end{proof}

\begin{lemma}\label{L:Kv1000} For $v\in\cE$, $Kv$ satisfies $Kv-\tilde u_b\in L^2(\Om\X[0,T],\cP;H_A^{1/2}(\cO))$, and 
\begin{equation}\label{e6.Kv}
Kv = u_0 + \int_0^t G(s)\, ds + \int_0^t \Phi(v(s))\, dW(s),
\end{equation}
where $G\in L^2(\Omega \times [0,T], \cP; H_A^{-1/2})$ is almost surely given by
\begin{equation}\label{e6.Kv2}
\begin{aligned}
  \big\langle G(s), \varphi \big\rangle_{H_A^{-1/2}, H_A^{1/2}} =\,& - \ve \int_\cO \nabla Kv(s,x) \cdot \nabla \varphi(x) \, dx \nonumber \\
                                                                     & - \mu  \int_\cO  \Delta_{x''} Kv(s,x) \Delta_{x''} \varphi(x) \, dx \nonumber \\
                                                                     & + \int_\cO  \Bbf(v(s,x)) : D^2_{x''} \varphi(x) \, dx \nonumber \\
                                                                     & + \int_\cO  \Abf(v(s,x)) \cdot \nabla \varphi(x) \, dx \nonumber.
\end{aligned}
\end{equation}
\end{lemma}
\begin{proof}  The fact that $Kv-\tilde u_b\in L^2(\Om\X[0,T],\cP;H_A^{1/2}(\cO))$ follows from the proof of Lemma~\ref{6:Kcontinuity}. Clearly, $K_{_{0,1,2}} v:=K_0v+K_1v+K_2v$ satisfies
\begin{equation*}
\begin{aligned}
\int_{\cO} K_{_{0,1,2}}  v\, \varphi\,dx=&-\int_{\cO}u_0\varphi\,dx-\ve\int_0^t\int_{\cO}\nabla K_{_{0,1,2}}v\cdot \nabla\varphi\,dx\,dt\\
&-\mu \int_0^t\int_{\cO}\Delta_{x''} K_{_{0,1,2}}v \Delta_{x''} \varphi\,dx\, ds,
\end{aligned}
\end{equation*}
for all $\varphi\in H_A^{1/2}(\cO)$, a.s.\ and for all \ $t\in[0,T]$.  As for $K_3v, K_4v$, they are convolutions of the semigroup, having the form
$$
K_jv(t)=\int_0^t S(t-s)\psi(s)\,ds,\quad j=3,4,
$$
with $\psi=D^2\Bbf(v)$ and $\psi=\nabla\cdot \Abf(v)$, respectively.  Therefore, proceeding similarly to the proof of Proposition~\ref{B:reg1}, through the use of the spectral theorem, we deduce that $K_{_{3,4}}v:=K_3v+K_4v$ satisfies, for all $\varphi\in H_A^{1/2}(\cO)$, a.s.\ and for all \ $t\in[0,T]$,
\begin{align*}
  \int_{\cO} K_{_{3,4}}v\, \varphi \,dx =\,& - \ve \int_0^t\int_\cO \nabla K_{_{3,4}}v(s,x) \cdot \nabla \varphi(x) \, dx\,ds \nonumber \\
                                                                     & - \mu\int_0^t  \int_\cO  \Delta_{x''} K_{_{3,4}}v(s,x) \Delta_{x''} \varphi(x) \, dx\,ds \nonumber \\
                                                                     & + \int_0^t\int_\cO  \Bbf(v(s,x)) : D^2_{x''} \varphi(x) \, dx\,ds \nonumber \\
                                                                     & +\int_0^t \int_\cO  \Abf(v(s,x)) \cdot \nabla \varphi(x) \, dx\,ds \nonumber.
\end{align*}
Similarly, $K_5v$ is a stochastic convolution of the semigroup and proceeding as in the proof of Proposition~\ref{B:reg2}, again using the spectral theorem,  we deduce that, for all $\varphi\in H_A^{1/2}(\cO)$, a.s.\ and for all \ $t\in[0,T]$, 
\begin{align*}
  \int_{\cO} K_5v\, \varphi \,dx =\,& - \ve \int_0^t\int_\cO \nabla K_5v(s,x) \cdot \nabla \varphi(x) \, dx\,ds \nonumber \\
                                                                     & - \mu\int_0^t  \int_\cO  \Delta_{x''} K_5v(s,x) \Delta_{x''} \varphi(x) \, dx\,ds \nonumber \\
                                                                     & + \int_0^t\int_{\cO} \Phi(v(s)) \varphi(x) \, dx\,dW(s). \nonumber                                                                      
\end{align*}
Adding up the equations obtained for $K_{_{0,1,2}} v$, $K_{_{3,4}}v$  and $K_5v$ we get \eqref{e6.Kv}, as desired.

\end{proof}

\begin{lemma}[Relative energy identity] \label{6:LemmaRelative0} If $v_1,v_2\in\E$, then, a.s., and for all $0 \leq t \leq T$,
\begin{align}
  \frac{1}{2}&\int_\cO | Kv_1(t,x) - Kv_2(t,x) |^2 \, dx \nonumber   \\
                                               +\, & \ve \int_0^t \int_\cO     |\nabla Kv_1(s,x)  - \nabla Kv_1(s,x) |^2  \, dx\, ds \nonumber \\
                                               +\, & \mu \int_0^t \int_\cO \big(\Delta_{x''} Kv_1(s,x) - \Delta_{x''} Kv_2(s)\big)^2\, dx\, ds \nonumber \\
                                               =\, &     \int_0^t \int_\cO  \big( \Bbf(v_1(s,x)) - \Bbf(v_2(s,x)) \big) : \big(D^2_{x''} Kv_1(s,x) - D^2_{x''} Kv_2(s,x) \big)  \, dx\, ds \nonumber\\
                                                & +    \int_0^t \int_\cO \big(\Abf(v_1(s,x)) - \Abf(v_2(s,x))\big) \cdot \big(\nabla Kv_1(s,x) - \nabla Kv_2(s,x) \big)  \, dx\, ds\nonumber \\
                                                & +    \int_0^t \int_\cO \big( \Phi(v_1(s)) - \Phi(v_2(s)) \big) \big( Kv_1(s,x) -  Kv_2(s,x) \big) \, dx\,dW(s) \nonumber \\
                                                & + \frac{1}{2} \int_0^t \int_\cO \sum_{k=1}^\infty | g_k(v_1(s,x)) - g_k(v_2(s,x)) |^2  \, dx\,ds. \label{6:Ito0}
\end{align} 
\end{lemma}
\begin{proof}
By Lemma~\ref{L:Kv1000}, we deduce  that  $Kv_j$, $j=1,2$,  can be written as
$$
Kv_j = u_0 + \int_0^t G_j(s)\, ds + \int_0^t \Phi(v_j(s))\, dW(s),
$$
where $G_j \in L^2(\Omega \times [0,T], \cP; H_A^{-1/2})$ is almost surely given by
\begin{align*}
  \big\langle G_j(s), \varphi \big\rangle_{H_A^{-1/2}, H_A^{1/2}} =\,& - \ve \int_\cO \nabla Kv_j(s,x) \cdot \nabla \varphi(x) \, dx \nonumber \\
                                                                     & - \mu  \int_\cO  \Delta_{x''} Kv_j(s,x) \Delta_{x''} \varphi(x) \, dx \nonumber \\
                                                                     & + \int_\cO  \Bbf(v_j(s,x)) : D^2_{x''} \varphi(x) \, dx \nonumber \\
                                                                     & + \int_\cO  \Abf(v_j(s,x)) \cdot \nabla \varphi(x) \, dx \nonumber.
\end{align*}

Let us introduce the approximations of the identity map $I_\lambda = (I + \lambda A)^{-1}$. We observe that $I_\lambda$ has the following properties:
\begin{enumerate}
  \item $I_\lambda \in L(H_A^{\alpha};H_A^{\alpha+1})$ with norm $\leq 1/\lambda$; and
  \item for any $f \in H_A^\alpha$, $I_\lambda f \rightarrow f$ in $H^\alpha_A$ as $\lambda \rightarrow 0$.
\end{enumerate}
Consequently, if $w_j = Kv_j$, and $w_j^\lambda = I_\lambda w_j$, then
\begin{align*}
  w_1^\lambda(t) - w_2^\lambda(t) &= \int_0^t  \big( I_\lambda G_1(s) - I_\lambda G_2(s)\big)\, ds \\ &+ \int_0^t \big( I_\lambda\Phi( v_1(s)) - I_\lambda \Phi(v_2(s)))\, dW(s) \text{ in } L^2(\cO).
\end{align*}
Thus the Itô's formula applied to the function $f \mapsto \frac{1}{2}\Vert f \Vert_{L^2(\cO)}^2$ gives, after some manipulation,
\begin{align}\label{6:Ito0.5}
& \frac{1}{2} \Vert w_1^\lambda(t) - w_2^\lambda(t) \Vert_{L^2(\cO)}^2 \\ 
 & +\, \ve \int_0^t \int_\cO \big( \nabla w_1(s,x) - \nabla w_2(s,x) \big) \cdot \nabla I_\lambda \big( w_1^\lambda(s,x) - w_2^\lambda(s,x) \big)\, dx\,ds \nonumber  \\ 
  &+ \mu \int_0^t \int_\cO \big( \Delta_{x''} w_1(s,x) - \Delta_{x''} w_2(s,x) \big) \cdot \Delta_{x''} I_\lambda \big( w_1^\lambda(s,x) - w_2^\lambda(s,x) \big)\, dx\, ds \nonumber \\
 &=  \int_0^t \int_\cO \big(\Bbf(v_1(s,x)) - \Bbf(v_2(s,x)) \big) : D_{x''}^2 I_\lambda \big( w_1^\lambda(s,x) - w_2^\lambda(s,x) \big)\, dx\, ds \nonumber\\
 & + \int_0^t \int_\cO \big(\Abf(v_1(s,x)) - \Abf(v_2(s,x)) \big) \cdot \nabla I_\lambda\big( w_1^\lambda(s,x) - w_2^\lambda(s,x) \big)\, dx\, ds \nonumber \\
                                                                         & + \sum_{k=1}^\infty \int_0^t \int_\cO \big( I^\l g_k (x,v_1(s,x)) - I^\l g_k(v_2(s,x))\big) \big(w_1^\lambda(s,x) - w_2^\lambda(s,x) \big) \,dx\, d\beta_k(s) \nonumber \\
                                                                         & +\frac12 \sum_{k=1}^\infty \int_0^t \int_\cO \big( I^\l g_k (x,v_1(s,x)) - I^\l g_k(v_2(s,x))\big)^2  \,dx\, ds. \nonumber
\end{align}  
At last, we notice that $w_1 - w_2 \in L^2(\Omega;$ $C([0,T];$ $L^2(\cO))$ $\cap$ $L^2(\Omega;$ $L^2(0,T; H_A^{1/2}))$, implying that we can pass $\lambda \rightarrow 0$ in \eqref{6:Ito0.5} to deduce \eqref{6:Ito0}.
\end{proof}

\begin{lemma} [A second relative energy estimate] \label{6:LemmaRelative}

There exists a constant $C_* > 0$, such that for any two elements $v_1$ and $v_2$ of $\E$, and for any $0\leq t \leq T$,
\begin{align}
  \bbE \sup_{0\leq \tau \leq t}\Big(  &\int_\cO | Kv_1(\tau,x) - Kv_2(\tau,x) |^2 \, dx \nonumber \\ & + \ve\, \int_0^\tau \int_\cO     |\nabla Kv_1(s,x)  - \nabla Kv_1(s,x) |^2  \, dx\, ds \nonumber \\
  &+\mu\, \int_0^\tau \int_\cO     \big( \Delta_{x''} Kv_1(s,x)  - \Delta_{x''} Kv_1(s,x) \big)^2  \, dx\, ds \Big) \nonumber \\&\quad\quad\quad\quad\quad\leq C_* \, \bbE \int_0^t \int_\cO | v_1(s,x) - v_2(s,x) |^2\, dx\,ds. \label{6:relativeenergy0}
\end{align} 
\end{lemma}
\begin{proof}
Let us take the expectation of the supremum of \eqref{6:Ito0} between $0 \leq \tau \leq t$, so that
\begin{align}
  \bbE \sup_{0\leq \tau \leq t}\Big(  \int_\cO \frac{1}{2}| &Kv_1(\tau,x) - Kv_2(\tau,x) |^2 \, dx \nonumber \\ 
          & + \ve\, \int_0^\tau \int_\cO     |\nabla Kv_1(s,x)  - \nabla Kv_1(s,x) |^2  \, dx\, ds            \nonumber \\
          & + \mu\, \int_0^\tau \int_\cO     \big( \Delta_{x''} Kv_1(s,x)  - \Delta_{x''} Kv_1(s,x) \big)^2  \, dx\, ds \Big) \nonumber \\
                                               \leq\, & \,     \bbE \sup_{0\leq \tau \leq t} \int_0^\tau \int_\cO  \big( \Bbf(v_1(s,x)) - \Bbf(v_2(s,x)) \big) \nonumber \\&\quad\quad\quad\quad\quad\quad\quad\quad\quad\quad: \big(D^2_{x''} Kv_1(s,x) - D^2_{x''} Kv_2(s,x) \big)  \, dx\, ds \nonumber\\
                                                & +   \bbE \sup_{0\leq \tau \leq t} \int_0^\tau \int_\cO \big(\Abf(v_1(s,x)) - \Abf(v_2(s,x))\big) \nonumber \\&\quad\quad\quad\quad\quad\quad\quad\quad\quad\quad\cdot \big(\nabla Kv_1(s,x) - \nabla Kv_2(s,x) \big)  \, dx\, ds\nonumber \\
                                                & +   \bbE \sup_{0\leq \tau \leq t} \int_0^\tau \int_\cO \big( \Phi(v_1(s)) - \Phi(v_2(s)) \big) \nonumber \\&\quad\quad\quad\quad\quad\quad\quad\quad\quad\quad\big( Kv_1(s,x) -  Kv_2(s,x) \big) \, dx\,dW(s) \nonumber \\
                                                & + \bbE \sup_{0\leq \tau \leq t} \frac{1}{2} \int_0^\tau \int_\cO \sum_{k=1}^\infty | g_k(v_1(s,x)) - g_k(v_2(s,x)) |^2  \, dx\,ds \nonumber \\
                                                = \, & I_1+ \ldots + I_4 \label{6:relativeenergy0.1}
\end{align} 
Let us look at each term separately. Using the well-known fact that
$$
\Vert f \Vert_{L^2(\cO';H_0^2(\cO'')}^2 \leq C\, \int_\cO (\Delta_{x''} f(x))^2 \, dx
$$
for some constant $C>0$ independent of  $f$, we, therefore, have, after an integration by parts and routine estimates
\begin{align}
I_1&\leq C\,\bbE \int_0^t \int_\cO  \big( v_1(s,x) - v_2(s,x) \big)^2 \, dx\, ds \nonumber\\ 
    &\quad\quad + \frac{1}{6}\bbE  \sup_{0\leq \tau \leq t}\Big( \int_\cO  \frac{1}{2} | Kv_1(\tau,x) - Kv_2(\tau,x) |^2 \, dx \nonumber \\ 
    &\quad\quad\quad\quad\quad+ \ve\, \int_0^\tau \int_\cO     |\nabla Kv_1(s,x)  - \nabla Kv_1(s,x) |^2  \, dx\, ds            \nonumber \\
    &\quad\quad\quad\quad\quad+ \mu\, \int_0^\tau \int_\cO     \big( \Delta_{x''} Kv_1(s,x)  - \Delta_{x''} Kv_1(s,x) \big)^2  \, dx\, ds \Big). \label{6:(I)}
\end{align}

Likewise, one has that
\begin{align}
I_2 &\leq C\,\bbE \int_0^t \int_\cO  \big( v_1(s,x) - v_2(s,x) \big)^2 \, dx\, ds \nonumber\\ 
     &\quad\quad + \frac{1}{6}\bbE \sup_{0\leq \tau \leq t} \Big( \int_\cO  \frac{1}{2} | Kv_1(\tau,x) - Kv_2(\tau,x) |^2 \, dx \nonumber \\ 
     &\quad\quad\quad\quad\quad+ \ve\, \int_0^\tau \int_\cO     |\nabla Kv_1(s,x)  - \nabla Kv_1(s,x) |^2  \, dx\, ds            \nonumber \\
     &\quad\quad\quad\quad\quad+ \mu\, \int_0^\tau \int_\cO     \big( \Delta_{x''} Kv_1(s,x)  - \Delta_{x''} Kv_1(s,x) \big)^2  \, dx\, ds \Big). \label{6:(II)}
\end{align}

To estimate  $I_3$, we apply  the Burkholder inequality (see, e.g., \cite{O}), to get
  \begin{align}\label{6:(III)}
   I_3 &= \mathbb{E}  \operatornamewithlimits{sup}_{0\leq \tau \leq t} \Bigg| \sum_{k=1}^\infty \int_0^\tau \int_{\cO} \big(g_k(v_1(s,x)) - g_k(v_2(s,x))  \big)\\ 
   &\quad\quad\quad\quad\quad\quad\quad\quad \big( (Kv_1)(s,x) - (Kv_2)(s,x) \big) \,dx \,d\beta_k(s) \Bigg| \nonumber\\
    &\leq C \,  \mathbb{E}\Big(   \int_0^t \sum_{k=1}^\infty \Big\{ \int_{\cO} (g_k( v_1(s,x)) - g_k(v_2(s,x)) ) \nonumber \\ 
    &\quad\quad\quad\quad\quad\quad\quad\quad\quad\quad\quad\quad\quad\quad ((Kv_1)(s,x) - (Kv_2)(s,x) ) \, dx \Big\}^2\, ds \Big)^{1/2}\nonumber \\
       &\leq \frac{1}{6} \mathbb{E}  \operatornamewithlimits{sup}_{0 \leq \tau \leq t} \Big( \frac{1}{2}\Vert (Kv_1)(t) - (Kv_2)(t) \Vert_{L^2(\cO)}^2 \nonumber    \\ 
    &\quad\quad\quad\quad\quad\quad + \mu \int_0^t \Vert \Delta_{x''}  (K v_1) (s) - \Delta_{x''} (K v_2) (s) \Vert_{L^2(\mathcal{O})}^2 \, dx ds \nonumber  \\
    &\quad\quad\quad\quad\quad\quad + \ve \int_0^t \Vert \nabla  (K v_1) (s)  - \nabla (K v_2) (s) \Vert_{L^2(\mathcal{O})}^2 \, dx ds  \Big) \nonumber  \\  
    & \quad\quad\quad\quad\quad\quad\quad\quad\quad\quad  + C \, \mathbb{E} \int_0^\tau \Vert v_1(s) - v_2(s) \Vert_{L^2(\mathcal{O})}^2 \,ds. \nonumber 
  \end{align}
  
  Finally, it is clear that
  \begin{equation}
    I_4 \leq D\,\bbE \int_0^t \int_{\cO} | v_1(s,x) - v_2(s,x) |^2\, dx\, ds. \label{6:(IV)}
  \end{equation}
  
  Combining \eqref{6:(I)}--\eqref{6:(IV)} and \eqref{6:relativeenergy0.1} we arrive at \eqref{6:relativeenergy0}, as desired.
\end{proof}

\begin{theorem}
$K: \E \rightarrow \E$ has a unique fixed point $u^{\mu,\ve}$. Moreover, $u^{\mu,\ve}$ is a weak solution of \eqref{e6.4}--\eqref{e6.8} in the sense that $u^{\mu,\ve}-\tilde u_b\in L^2(\Om\X[0,T], \cP; H_A^{1/2}(\cO))$, and
\begin{equation}\label{e6.Kv'}
u^{\mu,\ve} = u_0 + \int_0^t G(s)\, ds + \int_0^t \Phi(u^{\mu,\ve}(s))\, dW(s),
\end{equation}
where $G(s)$ is as in Lemma~\ref{L:Kv1000} with $u^{\mu,\ve}$ instead of $v$. 

\end{theorem}

\begin{proof}
Let us introduce the the equivalent norm in $\mathcal{E}$
\begin{align*}
  \Vert u \Vert_{*\E}^2 &\,= \operatornamewithlimits{sup}_{0\leq t \leq T} e^{-C_* t/ \alpha} \,  \mathbb{E} \operatornamewithlimits{sup}_{0 \leq \tau \leq t} \Big( \frac{1}{2} \Vert u(\tau) \Vert_{L^2(\cO)}^2 \\ &\quad\quad + \mu \int_0^\tau \Vert \Delta_{x''} u (s) \Vert_{L^2(\mathcal{O})}^2 \, ds + \ve \int_0^\tau \Vert \nabla u (s) \Vert_{L^2(\mathcal{O})}^2 \, ds \Big),
  \end{align*}
  where $0 < \alpha < 1$ may be arbitrarily chosen. {}From   \eqref{6:relativeenergy0}, it follows that
  $$
  \Vert Kv_1 - Kv_2 \Vert_{*\E}^2 \leq   C_* \mathbb{E} \int_0^t \Vert v_1 (s) - v_2(s) \Vert_{L^2(\cO)}^2 ds. 
  $$
  Now, we have
  \begin{align*}
    C_* \mathbb{E} \int_0^t \Vert v_1 (s) &- v_2(s) \Vert_{L^2(\cO)}^2 ds  \leq C_* \mathbb{E} \int_0^t \operatornamewithlimits{sup}_{0\leq \tau \leq s} \Vert v_1 (\tau) - v_2(\tau) \Vert_{L^2(\cO)}^2 \,ds \\
    &= C_* \,  \int_0^t e^{C_* s/ \alpha}\, e^{-C_* s/ \alpha} \,\mathbb{E} \operatornamewithlimits{sup}_{0\leq s \leq t} \Vert v_1 (\tau) - v_2(\tau) \Vert_{L^2(\cO)}^2\, dt \\
    &\leq C_*  \int_0^t e^{C_* s/ \alpha} \Vert v_1 - v_2 \Vert_{*\E}^2\, dt \\
    &\leq \alpha \, e^{C_*t/\alpha} \Vert v_1 - v_2 \Vert_{*\E}^2.
  \end{align*}
 Hence, we get
  $$
  \Vert Kv_1 - Kv_2 \Vert_{*\E}^2 \leq \alpha\,\Vert v_1 - v_2 \Vert_{*\E}^2. 
  $$
  This proves that $K$ is a contraction and so the first part of the theorem follows from Banach's fixed point theorem. The second part of the statement  follows immediately from Lemma~\ref{L:Kv1000}.
\end{proof}

\begin{lemma} [An uniform energy estimate] \label{6:uniformlemma} Let $u^{\mu,\ve}$ be a weak solution of \eqref{e6.4}--\eqref{e6.8}.
Then, there exists a constant $C>0$  such that for $0 < \mu < 1$,
\begin{align}
\bbE \Big( \sup_{0\leq t \leq T} \int_\cO & | u^{\mu, \ve}(t,x)-\tilde u_b(t) |^2 \, dx   + \ve\, \int_0^t \int_\cO     |\nabla u^{\mu,\ve}(s,x) |^2  \, dx\, ds \nonumber \\
  &\quad\quad\quad\quad\quad\quad\quad\quad+\mu\, \int_0^t \int_\cO     \big( \Delta_{x''}  u^{\mu,\ve}(s,x) \big)^2  \, dx\, ds \Big) \leq C. \label{6:uniformestimate1}
\end{align} 
\end{lemma}

\begin{proof}
 First,  consider the function $\tilde u_b(t,x)$ solution of \eqref{e6:1ub}--\eqref{e6:5ub}. Applying  It\^o formula to $\frac12\|u^{\ve,\mu} - \tilde u_b\|_{L^2(\cO)}^2$, we see that a.s.
\begin{align*}
  \frac{1}{2} &  \Vert u^{\mu, \ve}(t) -\tilde u_b(t) \Vert_{L^2(\cO)}^2 =  \frac{1}{2}\Vert u_0^\ve  \Vert_{L^2(\cO)}^2  \nonumber \\
              &- \int_0^t \int_\cO \mu(  \Delta_{x''} u^{\mu, \ve} - \Delta_{x''} \tilde u_b )( \Delta_{x''} u^{\mu, \ve} - \Delta_{x''}\tilde u_b ) \, dx\, ds \nonumber \\
              &- \int_0^t \int_\cO \ve( \nabla u^{\mu, \ve} - \nabla \tilde u_b)\cdot(\nabla u^{\mu, \ve} - \nabla \tilde u_b)\, dx\, ds \nonumber \\
              &- \int_0^t \int_\cO (\bbf(u^{\mu,\ve}) \nabla_{x''} u^{\mu, \ve}) \cdot (\nabla_{x''} u^{\mu,\ve} - \nabla_{x''} \tilde u_b) \, dx\, ds \nonumber \\
              &+ \int_0^t \int_\cO \Abf(u^{\mu,\ve}) \cdot (\nabla u^{\mu,\ve} - \nabla \tilde u_b ) \, dx\, ds \nonumber \\
              &+ \int_0^t \int_\cO (u^{\mu,\ve} - \tilde u_b) \Phi(u^{\mu,\ve}) \, dx \, dW(s) \nonumber \\
              &+ \frac{1}{2} \int_0^t \int_\cO \sum_{k=1}^\infty | g_k(u^{\mu,\ve}) |^2 \, dx\, ds\nonumber,
\end{align*}
and so
\begin{align*}
  \bbE \Big( \frac12\sup_{0\leq \tau \leq t} &\int_\cO | u^{\mu,\ve}(\tau,x)-\tilde u_b(\tau,x) |^2 \, dx  + \ve\, \int_0^t \int_\cO     |\nabla u^{\mu,\ve}(s,x)-\nabla \tilde u_b(s,x) |^2  \, dx\, ds  \\
  &\quad+\mu\, \int_0^t \int_\cO     \big( \Delta_{x''}  u^{\mu,\ve}(s,x) -\Delta_{x''} \tilde u_b(s,x)\big)^2  \, dx\, ds  \\
  &+ \int_0^t \int_\cO \bbf(u^{\mu,\ve}) (\nabla_{x''} u^{\mu, \ve}-\nabla_{x''}\tilde u_b) \cdot (\nabla_{x''} u^{\mu,\ve} - \nabla_{x''} \tilde u_b) \, dx\, ds\Big)\\
  &\le \frac 12\|u_0\|^2+ \bbE \int_0^t \int_\cO |(\bbf(u^{\mu,\ve}) \nabla_{x''} \tilde u_b) \cdot (\nabla_{x''} u^{\mu,\ve} - \nabla_{x''} \tilde u_b)| \, dx\, ds \nonumber\\
   &+\bbE \int_0^t \int_\cO |(\Abf(u^{\mu,\ve})-\Abf(\tilde u_b)) \cdot (\nabla u^{\mu,\ve} - \nabla \tilde u_b )| \, dx\, ds \nonumber \\
   &+\bbE \int_0^t \int_\cO |\Abf(\tilde u_b) \cdot (\nabla u^{\mu,\ve} - \nabla \tilde u_b )| \, dx\, ds \nonumber \\
              &+\bbE\sup_{0\le \tau\le t} \Big|\int_0^t \int_\cO (u^{\mu,\ve} - \tilde u_b) \Phi(u^{\mu,\ve}) \, dx \, dW(s)\Big| \nonumber \\
              &+ \bbE\frac{1}{2} \int_0^t \int_\cO \sum_{k=1}^\infty | g_k( u^{\mu,\ve}) |^2 \, dx\, ds\nonumber\\
              &=I_1+\cdots+I_6. 
\end{align*} 

Here,  $I_1$ is trivial, $I_2$  is estimated using Young's inequality with $\ve$ to get a term that can be absorbed in the left-hand side and another term which  is bounded  by a multiple of the squared $L_{\om,t,x}^2$-norm of $\nabla_{x''}\tilde u_b$, using the boundedness of $\Bbf'$. $I_3, I_4$ are similarly estimated using 
Young's inequality with $\ve$ and the boundedness of $\Abf'$. $I_5$ is estimated using Burkholder inequality as in the proof of Lemma \ref{6:LemmaRelative}. As for $I_6$, we use \eqref{e1.4*}.   We then use  Gr\"onwall's inequality to see that
\begin{align*}
  \bbE \Big( \sup_{0\leq t \leq T} &\int_\cO | u^{\mu,\ve}(t,x) - \tilde u_b(t,x) |^2 \, dx \nonumber + \ve\, \int_0^T \int_\cO     |\nabla u^{\mu,\ve}(t,x) - \nabla \tilde u_b(t,x) |^2  \, dx\, dt \nonumber \\
  &\quad\quad\quad\quad\quad\quad\quad\quad+\mu\, \int_0^T \int_\cO     \big( \Delta_{x''}  u^{\mu,\ve}(t,x) - \Delta_{x''} \tilde u_b(t,x)\big)^2  \, dx\, dt \Big) \leq C,
\end{align*} 
uniformly for $0 < \mu < 1$, from which \eqref{6:uniformestimate1} immediately follows.
\end{proof}

\section{Existence, part two: The second approximate problem}\label{S:7}

\subsection{Some a priori properties of weak solutions}\label{SS:7.1}

Since our limiting argument when $\mu \rightarrow 0$ involves the Yamada-Watanabe method \cite{YW}, with application of Gy\"ongy-Krylov criterion \cite{GK}, it is necessary that we first establish the uniqueness of weak solutions to \eqref{e6.1}--\eqref{e6.3}. This subsection will be concerned with this and other important facts about weak solutions. 

First, let us state the concept of solution that we will be employing throughout this section.

\begin{definition}
A process $u$ is said to be a \text{weak solution} to \eqref{e6.1}--\eqref{e6.3'} provided that

\begin{enumerate}
  \item $u \in L^2(\Omega; C([0,T];L^2(\cO)))$ $\cap$ $L^2(\Omega \times [0,T], \cP; H^1(\cO))$;
  \item $u = u^b$ almost surely in the sense of traces on $(0,T) \times \cO' \times \po\cO''$; 
  \item Almost surely, and for any $0 \leq t \leq T$ and any $\theta \in L^2(\cO', H_0^1(\cO''))$ $\cap$ $L^2(\cO'';H^1(\cO'))$, 
  \begin{align}
  (u(t), \theta)_{L^2(\cO)} =\,&  (u_0, \theta)_{L^2(\cO)} - \ve \int_0^t \int_\cO \nabla_{x'} u \cdot \nabla_{x'} \theta \,dx\, ds \nonumber \\
                               &- \int_0^t \int_\cO \operatorname{div}_{x''} \Bbf^\ve(u) \cdot \nabla_{x''} \theta \, dx\, ds + \int_0^t \int_\cO \Abf(u) \cdot \nabla \theta \, dx\, ds \nonumber \\
                               &+ \int_0^t \int_\cO \theta(x) \, \Phi(u) \, dx\, dW(s). \label{6':1}
  \end{align}
\end{enumerate}
\end{definition}

We will now deduce some estimates based on It\^o's formula for weak solutions of \eqref{e6.1}--\eqref{e6.3'}. Our arguments are similar to those we have already applied in the proof of Lemma \ref{6:LemmaRelative0}.

Let us introduce the operator $T: D(T) \subset L^2(\cO) \rightarrow L^2(\cO)$ by
$$
\begin{cases}
  D(T) = \{ f \in H^2(\cO); &\partial_\nu f = 0 \text{ on } \partial \cO' \times \cO'' \\ &\text{and } f = 0 \text{ on }  \cO' \X \partial \cO''  \}, \\
    Tf = -\Delta f.
\end{cases}
$$
According to Theorem~\ref{thmspectralT}, $T$ is also a nonnegative self-adjoint operator. Again, if we introduce its ``approximation of the identity'' $J_\lambda = (I + \lambda T)^{-1}$ for $\lambda > 0$, we have that \begin{enumerate}
  \item If $H_T^\alpha$ denotes the intermediate spaces of $T$, $J_\lambda \in L(H_T^{\alpha};H_T^{\alpha+1})$ with norm $\leq 1/\lambda$, and
  \item for any $f \in H_T^\alpha$, $J_\lambda f \rightarrow f$ in $H^\alpha_T$ as $\lambda \rightarrow 0$.
\end{enumerate}

The key element for the proof of our Itô's formula is the nearly obvious constatation that, for $H_T^{1/2} = L^2(\cO';H_0^1(\cO))$ $\cap$ $L^2(\cO'';H^1(\cO))$, a weak solution $u$ of \eqref{6':1} satisfies
\begin{equation}
  u(t) = u_0 + \int_0^t B(s) \, ds + \int_0^t \Phi(s,u(s)) \, dW(s) \quad \text{ in } H_T^{-1/2} = (H_T^{1/2})^*, \nonumber
\end{equation}
where $B \in L^2(\Omega \times [0,T], \cP; H_T^{-1/2})$ is almost surely the functional
\begin{align*}
  \big\langle B(s), \theta \big\rangle_{H^{-1/2}_T, H_T^{1/2}} =\,& -\int_\cO \big(\ve \nabla_{x'} u \cdot \nabla_{x'} \theta + \operatorname{div}_{x''} \Bbf^\ve(u) \cdot \nabla_{x''} \theta \big) \, dx \\
  &+   \int_\cO \Abf(u) \cdot \nabla \theta \, dx.
\end{align*}

Hence, $u^\lambda = J_\lambda u$ satisfies
\begin{equation}
  u^\lambda (t) = J_\lambda u_0 + \int_0^t J_\lambda B(s) \, ds + \int_0^t J_\lambda \Phi(s) \, dW(s) \quad \text{ in } L^2(\cO). \label{e6':2}
\end{equation}
Consequently, given any $\eta \in C^2(-\infty,\infty)$ with $\frac{d^2\eta}{du^2} \in L^\infty(-\infty, \infty)$, and any $\theta \in C^1(\overline{\cO})$ which vanishes near $\cO' \X \partial \cO''$ (i.e., $ \theta \in C_c^1(\R^{d'} \times \cO'')$), the classic It\^o's formula asserts that
\begin{align}
  \int_\cO \theta \, \eta(u^\lambda) \, dx =\,& \int_\cO \theta \, \eta(J_\lambda u_0)  \, dx + \int_0^t \int_\cO  \theta \, \eta'(u^\lambda) J_\lambda B(s) \, dx \, ds \nonumber \\
                     &+ \int_0^t \int_\cO  \theta \, \eta'(u^\lambda) J_\lambda \Phi(u) \, dx \, dW(s) \nonumber\\  &+ \frac{1}{2} \int_0^t \int_\cO \theta \, \eta''(u^\lambda) \sum_{k=1}^\infty | J_\lambda g_k(u) |^2  \, dx \, ds, \label{e6':3}
\end{align}
a.s. for any $0 \leq t \leq T$. Here and henceforth, for simplicity, we denote $\eta'=\frac{d\eta}{du}$ and $\eta''=\frac{d^2\eta}{du^2}$. Next we observe that 
\begin{align*}
  (J_\lambda B(s), \theta \, \eta'(u^\lambda) )_{L^2(\cO)} &= \big\langle B(s), J_\lambda ( \theta \,\ \eta'(u^\lambda) ) \big\rangle_{H^{-1/2}_T, H_T^{1/2}},
\end{align*}
so that \eqref{e6':3} means that
\begin{align}
  \int_\cO \theta \, \eta(u^\lambda) \, dx =\,& \int_\cO \theta \, \eta(J_\lambda u_0)  \, dx - \ve \int_0^t \int_\cO \nabla_{x'} u \cdot \nabla_{x'} J_\lambda \big( \theta \, \eta'(u^\lambda) \big)  \, dx \, ds \nonumber \\
                                              &- \int_0^t \int_\cO \operatorname{div}_{x''} \Bbf^\ve(u) \cdot \nabla_{x''} J_\lambda \big( \theta \, \eta'(u^\lambda)  \big) \,dx\,ds \nonumber\\
                                              &+ \int_0^t \int_\cO \Abf(u) \cdot \nabla J_\lambda \big( \theta \, \eta'(u^\lambda)  \big) \,dx\,ds \nonumber\\ 
                                              &+ \int_0^t \int_\cO  \theta \, \eta'(u^\lambda) J_\lambda \Phi(u) \, dx \, dW(s) \nonumber \\  &+ \frac{1}{2} \int_0^t \int_\cO \theta \, \eta''(u^\lambda) \sum_{k=1}^\infty | J_\lambda g_k(u) |^2  \, dx \, ds \label{e6':4}
\end{align}
Passing the limit $\lambda \rightarrow 0$, we obtain the following theorem.

\begin{theorem} \label{6':Ito1}
  If $u$ is a weak solution of \eqref{e6.1}--\eqref{e6.3'}, the following chain rule is valid: For every $0\leq t \leq T$, every $\eta \in C^2(-\infty, \infty)$ with $\eta'' \in L^\infty$, and every $\theta \in H_T^{1/2} \cap C^1(\overline{\cO})$,
  \begin{align}
  \int_\cO \theta(x) \, \eta(u(t,x)) \, dx =\,& \int_\cO \theta(x) \, \eta(u_0(x))  \, dx \nonumber \\&- \ve \int_0^t \int_\cO \nabla_{x'} u(s,x) \cdot \nabla_{x'} \big( \theta(x) \, \eta'(u(s,x)) \big)  \, dx \, ds \nonumber \\
                                              &- \int_0^t \int_\cO \operatorname{div}_{x''} \Bbf^\ve(u(s,x)) \cdot \nabla_{x''} \big( \theta(x) \, \eta'(u(s,x))  \big) \,dx\,ds \nonumber\\
                                              &+ \int_0^t \int_\cO \Abf(u(s,x)) \cdot \nabla \big( \theta(x) \, \eta'(u(s,x))  \big) \,dx\,ds \nonumber \\ 
                                              &+ \int_0^t \int_\cO  \theta(x) \, \eta'(u(s,x)) \Phi(u(s,x)) \, dx \, dW(s) \nonumber \\  &+ \frac{1}{2} \int_0^t \int_\cO \theta(x) \, \eta''(u(s,x)) \sum_{k=1}^\infty | g_k(u(s,x)) |^2  \, dx \, ds, \text{ a.s.}  \label{e6':Ito1}
\end{align}
\end{theorem}

\begin{proof}
Evidently, it suffices to prove the theorem for the simpler case we were dealing before, which was with $\theta$ vanishing near $\cO' \times \partial \cO''$. The passage to the limit $\lambda \rightarrow 0$ in \eqref{e6':4} has simple terms, not so complicated terms, and others that really require some attention. The simple terms are all that do not involve any derivatives on $u^\lambda$; these ones clearly converge to the corresponding limit.

On the other hand, by means of the Spectral Theorems \ref{tensorialspectralthm} and \ref{thmspectralT}, one can easily see that $u^\lambda \rightarrow u$ in $L^2(\Omega \times [0,T], \cP;$ $L^2(\cO''; H^1(\cO')))$. Hence the term with derivatives with respect to $x'$ also converge nicely. 

However, for the terms involving the derivatives of $u^\lambda$ with respect to $x''$, which require more attention, we have the following. First, let us show that $u^\lambda \rightarrow u$ in $L^2(\Omega \times [0,T, \cP;$ $L^2(\cO';H_{\text{loc}}^1(\cO'')))$. Fixing $\omega \in \Omega$ and $0\leq t \leq T$, $u^\lambda$ is characterized by
\begin{equation}\label{e.ulambda}
\lambda \int_\cO \nabla u^\lambda \cdot \nabla \varphi \, dx + \int_\cO u^\lambda \varphi \, dx = \int_\cO u \varphi \, dx
\end{equation}
for any $\varphi \in C^\infty(\overline{\cO})$ vanishing near $\cO' \X \partial \cO''$. Thus, if $\partial_{x''}$ denotes any partial derivative with respect to one of the $x''$--variables, we have
$$\lambda \int_\cO \nabla u^\lambda \cdot \nabla \partial_{x''}\varphi \, dx + \int_\cO u^\lambda \partial_{x''}\varphi \, dx = \int_\cO u \partial_{x''}\varphi \, dx,$$
which imples (as $u^\lambda \in H^2$),
$$\lambda \int_\cO \nabla \partial_{x''} u^\lambda \cdot \nabla \varphi \, dx + \int_\cO \partial_{x''} u^\lambda \varphi \, dx = \int_\cO \partial_{x''} u \varphi \, dx.$$
By means of smooth approximations, the identity above remains true for $\varphi = \phi^2 \partial_{x''} u^\lambda $, where $\phi$ is also in $C^1(\overline{\cO})$ and vanishes near $\cO' \X \partial \cO''$. Therefore,
\begin{align}
  \lambda \int_\cO |\nabla \partial_{x''} u^\lambda|^2 \phi^2 \, dx + 2 \lambda \int_\cO &\phi \partial_{x''} u^\lambda \nabla \partial_{x''} u^\lambda \cdot \nabla \phi dx \nonumber \\ & +  \int_\cO ( \partial_{x''} u^\lambda )^2 \phi^2 \, dx = \int_\cO \partial_{x''} u \partial_{x''} u^\lambda \phi^2 \, dx. \label{6}
\end{align}
Consequently, as $\lambda \int_\cO |\nabla u^\lambda|^2 \, dx$ is bounded by $\frac14 \int_\cO u^2 dx$, which follows easily from \eqref{e.ulambda}, we get
$$\int_\cO \phi^2 (\partial_{x''} u^\lambda)^2 \, dx \leq C(\phi) \int_\cO (u^2 + |\nabla u|^2) \, dx.$$
It is established thus that $u^\lambda \rightharpoonup u$ weakly in $L^2(\Omega \times [0,T], \cP; L^2(\cO'; H_{\text{loc}}^1(\cO'')))$, once that we already knew that $u^\lambda \rightarrow u$ strongly in $L^2(\Omega \times [0,T]\times \cO)$. On the other hand, returning to \eqref{6}, we see that
$$\limsup_{\lambda \rightarrow 0} \int_\cO \phi^2 (\partial_{x''} u^\lambda)^2 \, dx \leq \int_\cO \phi^2 (\partial_{x''} u)^2 \, dx,$$
implying, as asserted, that $u^\lambda \rightarrow u$ strongly in $L^2(\Omega \times [0,T], \cP;$ $L^2(\cO';H_{\text{loc}}^1(\cO'')))$.

 Consequently, it is not hard to see that there exists a sequence $\lambda_n \rightarrow 0+$ such that a.s.
\begin{align*}
  \int_0^t \int_\cO \operatorname{div}_{x''} \Bbf^\ve&(u(s,x)) \cdot \nabla_{x''} J_{\lambda_n} \big( \theta(x) \, \eta'(u^{\lambda_n}(s,x))  \big) \,dx\,ds \\ &\rightarrow \int_0^t \int_\cO \operatorname{div}_{x''} \Bbf^\ve(u(s,x)) \cdot \nabla_{x''} \big( \theta(x) \, \eta'(u(s,x))  \big) \,dx\,ds, \text{ and} \\
  \int_0^t \int_\cO \Abf&(u(s,x)) \cdot \nabla J_{\lambda_n} \big( \theta(x) \, \eta'(u^{\lambda_n}(s,x))  \big) \,dx\,ds \\ &\rightarrow \int_0^t \int_\cO \Abf(u(s,x)) \cdot \nabla \big( \theta(x) \, \eta'(u(s,x))  \big) \,dx\,ds.
\end{align*}
This concludes the proof. 
\end{proof}

Regarding now an Itô's formula for the difference of two solutions $u$ and $v$, let us just mention the reprisal of the previous arguments yields the following theorem.

\begin{theorem} \label{6':Ito2}
  Let $u$ and $v$ be two weak solutions of \eqref{e6.1}--\eqref{e6.3'} with possibly different initial and boundary datum. Then, almost surely, for any $\eta \in C^2(-\infty,\infty)$ with $\eta'' \in L^\infty$, any $0\leq t \leq T$, and any $\theta \in H_T^{1/2} \cap C^1(\overline{\cO})$,
    \begin{align}
  \int_\cO \theta(x) \, &\eta(u(t,x) - v(t,x)) \, dx =\, \int_\cO \theta(x) \, \eta(u_0(x) - v_0(x))  \, dx \nonumber \\&- \ve \int_0^t \int_\cO \big(\nabla_{x'} u(s,x) - \nabla_{x'} v(s,x) \big)  \cdot \nabla_{x'} \big( \theta(x) \, \eta'(u(s,x)-v(s,x)) \big)  \, dx \, ds \nonumber \\
                                              &- \int_0^t \int_\cO \big( \operatorname{div}_{x''} \Bbf^\ve(u(s,x)) - \operatorname{div}_{x''} \Bbf^\ve(v(s,x)) \big) \nonumber \\ &\quad\quad\quad\quad\quad\quad\quad\quad\quad\quad\quad\quad\quad\,\,\,\cdot \nabla_{x''} \big( \theta(x) \, \eta'(u(s,x)- v(s,x))  \big) \,dx\,ds \nonumber\\
                                              &+ \int_0^t \int_\cO (\Abf(u(s,x)) -\Abf(v(s,x)) \cdot \nabla \big( \theta(x) \, \eta'(u(s,x)-v(s,x))  \big) \,dx\,ds \nonumber \\ 
                                              &+ \int_0^t \int_\cO  \theta(x) \, \eta'(u(s,x)-v(s,x))(\Phi(u(s,x)) - \Phi(v(s,x))) \, dx \, dW(s) \nonumber \\  &+ \frac{1}{2} \int_0^t \int_\cO \theta(x) \, \eta''(u(s,x)-v(s,x))  \nonumber \\ &\quad\quad\quad\quad\quad\quad\quad\quad\quad\quad\quad\quad\quad\sum_{k=1}^\infty | g_k(u(s,x)) - g_k(v(s,x)) |^2  \, dx \, ds. \label{e6':Ito2}
\end{align}
\end{theorem}

The previous formula has an important extension which follows from the elementary theory on Sobolev spaces (see, e.g., \cite{Ev}). We state it  as a corollary for future reference. 

\begin{corollary} \label{6':corolIto2}
If $\eta'(u - v) \in L^2(\Omega \times [0,T], \cP;$ $L^2(\cO';H_0^1(\cO'')))$, then $\theta(x)$ can be chosen  to be only in $C^1(\overline{\cO})$ in the equation \eqref{6':Ito2}.
\end{corollary}

We now establish our comparison principle for weak solutions to \eqref{e6.1}--\eqref{e6.3'}.
\begin{lemma}
Assume that $u$ and $v$ are weak solutions to \eqref{e6.1}--\eqref{e6.3'}, and assume that
\begin{equation}
  u^b \leq v^b \text{ a.s. on } (0,T) \times \cO' \times  \partial \cO'', \label{6':boundarydatum}
\end{equation}
where $u^b$ and $v_b$ denote, respectively, the boundary data of $u$ and $v$. Then, for any $0 \leq t \leq T$,
\begin{equation}
  \bbE \int_\cO (u(t,x) - v(t,x))_+\, dx \leq \bbE \int_\cO (u_0(x) - v_0(x))_+\, dx, \label{6':precomparisonprinciple}
\end{equation}
with likewise $u_0$ and $v_0$ standing for the initial data of, respectively, $u$ and $v$.
\end{lemma}

\begin{proof}

Let $\psi \in C_c^\infty(-\infty, \infty)$ be such that $\psi \geq 0$, $\operatorname{supp.} \psi \subset (-1,1)$ and $\int_{-\infty}^\infty \psi(t)\,dt = 1$. If $\psi_\delta(t) = \frac{1}{\delta} \psi(\delta^{-1}t)$ ($\delta > 0$), then put

$$\sign_\delta^+(t) = \int_{-\infty}^t \psi_\delta(k-\d) \,dk = \int_{-\infty}^t \psi\Big( \frac{k-\delta}{\delta} \Big) \,\frac{dk}{\delta}.$$
Define also $\eta_\delta(t) = \int_{-\infty}^t \sign_\delta^+(k) \, dk.$ Notice that $\eta_\delta$ is a smooth convex approximation of the ``positive part'' function $t \mapsto t_+$.

Because of \eqref{6':boundarydatum}, $\eta_\delta'(u-v) \in L^2(\Omega \times [0,T], \cP;$ $L^2(\cO';H_0^1(\cO'')))$, hence, by Corollary \ref{6':corolIto2} and \eqref{6':Ito2},

\begin{align}
  \int_\cO &\eta_\delta(u(t,x) - v(t,x)) \, dx =\, \int_\cO \theta(x) \, \eta_\delta(u_0(x) - v_0(x))  \, dx \nonumber \\&- \ve \int_0^t \int_\cO \eta_\delta''(u(s,x) - v(s,x))|\nabla_{x'} u(s,x) - \nabla_{x'} v(s,x) |^2  \, dx \, ds \nonumber \\
                                              &- \int_0^t \int_\cO \eta_\delta''(u(s,x) - v(s,x)) \big(  \bbf^\ve(u(s,x)) \nabla_{x''} u(s,x) - \bbf^\ve v(s,x)) \nabla_{x''} v(s,x) \big) \nonumber \\ &\quad\quad\quad\quad\quad\quad\quad\quad\quad\quad\quad\quad\quad\,\,\,\cdot \big( \nabla_{x''} u(s,x) - \nabla_{x''} v(s,x))  \big) \,dx\,ds \nonumber\\
                                              &+ \int_0^t \int_\cO \eta_\delta''(u(s,x) - v(s,x)) (\Abf(u(s,x)) -\Abf(v(s,x)) \nonumber \\ &\quad\quad\quad\quad\quad\quad\quad\quad\quad\quad\quad\quad\quad\,\,\,\cdot  \big( \nabla (u(s,x) - \nabla v(s,x))  \big) \,dx\,ds \nonumber \\ 
                                              &+ \int_0^t \int_\cO  \eta_\delta'(u(s,x)-v(s,x)) (\Phi(u(s,x)) - \Phi(v(s,x))) \, dx \, dW(s) \nonumber \\  &+ \frac{1}{2} \int_0^t \int_\cO \eta_\delta''(u(s,x) - v(s,x)) \sum_{k=1}^\infty | g_k(u(s,x)) - g_k(v(s,x)) |^2  \, dx \, ds, \text{ a.s.}  \nonumber
\end{align}

We now pass $\delta \rightarrow 0$ in the identity above. First, we observe that, taking the expected value of the expressions, the term involving the stochastic integral will vanish; additionally, we also observe that the third term is $\leq 0$. 

On the other hand, the fourth term can be written as 
\begin{align*}
  -\int_0^t \int_\cO \eta_\delta''(u(s,x) - v(s,x)) &\big(\nabla_{x''} u(s,x) - \nabla_{x''} v(s,x) \big) \nonumber \\ &\cdot \bbf^\ve(u(s,x))\big( \nabla_{x''} u(s,x) - \nabla_{x''} v(s,x))  \big) \,dx\,ds\\
  -\int_0^t \int_\cO \eta_\delta''(u(s,x) - v(s,x)) &\big(\nabla_{x''} u(s,x) - \nabla_{x''} v(s,x) \big) \nonumber \\ &\cdot \big(\bbf^\ve(u(s,x)) - \bbf^\ve (v(s,x)) \big) \nabla_{x''} v(s,x)\,dx\,ds \\
  &= - \text{(i)} + \text{(ii)}.  
\end{align*}
Of course $\text{(i)} \geq 0$, whereas
\begin{align*}
  \text{(ii)} &\leq \Vert \bbf_\ve \Vert_\infty \int_0^t \int_\cO \psi\Big( \frac{u(s,x) - v(s,x)-\delta}{\delta} \Big) \Big| \frac{u(s,x) - v(s,x)}{\delta} \Big| \\
              &\quad\quad\quad |\nabla_{x''} u(s,x) - \nabla_{x''} v(s,x)| \, |\nabla_{x''} v(s,x) | \, dx\, ds. 
\end{align*}
The integrand above is uniformly bounded by an $L^1$--function and converges pointwisely to $0$. Hence,
$$\bbE \text{(ii)} = o(1).$$

Likewise the hyperbolic term is $o(1)$. Finally, for the last term, we notice that
\begin{align*}
  \frac{1}{2} \int_0^t \int_\cO \eta_\delta''&(u(s,x) - v(s,x)) \sum_{k=1}^\infty | g_k(u(s,x)) - g_k(v(s,x)) |^2  \, dx \, ds \\
  &\leq D \delta \int_0^t \int_\cO \psi \Big( \frac{u(s,x) - v(s,x)-\delta}{\delta} \Big) \Big| \frac{u(s,x) - v(s,x)}{\delta} \Big|^2 \, dx\, ds \\ &= O(\delta).
\end{align*}

For
\begin{align*}
  \bbE \int_\cO \eta_\delta(u(t,x) - v(t,x)) \, dx  &\rightarrow \bbE \int_\cO (u(t,x) - v(t,x))_+ \, dx \text{ and}\\
  \bbE \int_\cO \eta_\delta(u_0(x) - v_0(x)) \, dx  &\rightarrow \bbE \int_\cO (u_0(x) - v_0(x))_+ \, dx,
\end{align*}
the proof of the theorem is thus complete.
\end{proof}

The following theorem then follows immediately from the result just proven. Note, in particular, that by virtue of \eqref{e1.f} and \eqref{e1.8'}, the functions $u\equiv u_\text{min}$ and $u\equiv u_\text{max}$ solve equations \eqref{e6.1} and \eqref{e6.3'}.
\begin{theorem} \label{6':compunimax}
  Let $u$ and $v$ be weak solutions to \eqref{e6.1}--\eqref{e6.3'} with initial and boundary datum being, respectively, $u_0$ and $u^b$, and $v_0$ and $v^b$. The following assertions hold true.
  \begin{enumerate}
    \item The comparison principle: If $u_0 \leq v_0$ and $u^b \leq v^b$ almost surely, then $u \leq v$ almost surely.
    \item Uniqueness: If $u_0 = v_0$ and  $u^b = v^b$ almost surely, then $u = v$ almost surely.
    \item The maximum princple: If $u_\text{min} \leq u_0 \leq u_\text{max}$ and  $u_\text{min} \leq u^b \leq \text{max}$ almost surely, then $u_\text{min} \leq u \leq u_\text{max}$ almost surely.
  \end{enumerate}
\end{theorem}

Let us close this section with an energy estimate which will be required later on. Its proof follows rather easily from the ideas in the proof of Lemma~\ref{6:uniformlemma} and so we omit it here for the sake of brevity. 

\begin{lemma}[An uniform energy estimate]\label{L:7.100}
Let $u^\ve$ be a weak solution of \eqref{e6.1}--\eqref{e6.3'}. There exists a constant $C > 0$ such that for $0 < \ve < 1$
\begin{equation}
\bbE\sup_{0\le t\le T}\|u^\ve(t)\|^2_{L^2(\cO)} +  \bbE \int_0^T \int_\cO \Big( |\nabla_{x''} \bbf (u^\ve) |^2 +  \ve\, |\nabla_{x'} u^\ve(s,x) |^2 \Big) \, dx\, ds \leq C. \label{6':uniformenergyestimate}
\end{equation} 
\end{lemma}

\subsection{Existence of solutions to \eqref{e6.1}--\eqref{e6.3'}}\label{SS:7.2}

The proof of the existence of weak solutions to \eqref{e6.1}--\eqref{e6.3'} now follows through the same arguments employed in section~4 of \cite{DHV}, step by step, with minor adaptations. First, Lemma~\ref{6:uniformlemma} assures the uniform boundedness of the solutions of \eqref{e6.4}--\eqref{e6.8} in $L^2(\Om;L^2([0,T];H^1(\cO)))$. Using Kolmogorov's continuity theorem as in proposition~4.4 of \cite{DHV} we get that 
\begin{equation}\label{eSS7.2_1} 
\bbE\|u^{\ve,\mu}\|_{C^\l([0,T];H^{-3}(\cO))}\le C,
\end{equation}
for any $\l\in(0,1/2)$ for some $C>0$, independent of $\mu$.  Using Lemma~\ref{6:uniformlemma} and \eqref{eSS7.2_1} we 
prove the tightness of the laws of $u^{\mu,\ve}$ in $\mathcal{X}=L^2([0,T];L^2(\cO))\cap C([0,T];H^{-3}(\cO))$, and so its relative compactness by Prokhorov's theorem. Then we apply Skorokhod representation theorem to infer the convergence a.e.\ of a subsequence $u^{\mu_n,\ve}$ in a new probability space $(\tilde \Om, \tilde {\mathcal F}, \tilde\bbP)$. Then, as in section~4 of \cite{DHV}, we prove that the limit is a  weak  martingale solution of \eqref{e6.1}--\eqref{e6.3'}. Finally we use the uniqueness of the solutions of \eqref{e6.1}--\eqref{e6.3'}, Theorem~\ref{6':compunimax}, and apply Gy\"ongy-Krylov's criterion
to conclude that the whole sequence $u^{\mu,\ve}$ converges to the unique solution $u^\ve$ of \eqref{e6.1}--\eqref{e6.3'}, which concludes the prove of the existence of solutions to \eqref{e6.1}--\eqref{e6.3'}.

\section{Existence, part three: Degenerate case}\label{S:8}

We finally discuss the existence of a kinetic solution to problem \eqref{e1.1}--\eqref{e1.4}. Here, we follow the compactness argument in \cite{Ha}, with the decisive help of the space regularity result established in \cite{GH}.  Again we use the Yamada-Watanabe method \cite{YW}, with application of Gy\"ongy-Krylov's criterion \cite{GK}. Concerning the latter, we recall that the uniqueness of the kinetic solution to problem~\eqref{e1.1}--\eqref{e1.4}  has been established in Theorem~\ref{T:2.2}.  Now, let $u^\ve$ be the solution to the problem~\eqref{e6.1}--\eqref{e6.3'} and let $\ff^\ve(t,x,\xi):=\chi^\ve (t, x,\xi)=1_{(-\infty, u^\ve(t,x))}(\xi)-1_{(-\infty,0)}(\xi)$. We can prove, by an argument similar to the one in \cite{DHV}, that $\ff^\ve$ satisfies
\begin{multline}\label{e8.1}
\po_t \ff^\ve + \abf(\xi)\cdot \nabla_{x} \ff^\ve + \bbf^\ve(\xi):D_{x''}^2 \ff^\ve \\= \left(m^\ve-\frac12G^2(\xi)\,\d_{u^\ve(t,x)}(\xi)\right)_\xi 
+\sum_{k=1}^\infty g_k(\xi) \dot\beta_k(t)\, \d_{u^\ve(t,x)}(\xi)\\
=-\abf''(\xi)\cdot\nabla_{x''}\ff^\ve+ q_\xi^\ve-\sum_{k=1}^\infty g_k(\xi)(\po_\xi\ff^\ve) \dot\beta_k(t) +\sum_{k=1}^\infty\d_0(\xi) g_k(\xi)\dot\beta_k(t),
\end{multline}
where $\bbf^\ve(\xi):=\bbf(\xi)+\ve I_{d''\X d''}$ and  $q^\ve=m^\ve-\frac12G^2(\xi)\,\d_{u^\ve(t,x)}(\xi)$,
$$
dm^\ve(t,x,\xi)=|\s^\ve(u^\ve)\nabla u^\ve|^2\, d\d_{u^\ve=\xi}\,dx\,d\xi,
$$
and $\s^\ve(u)\in {\mathbb M}^{d''}$ is such that $\s^{\ve}(u)^2=\bbf(u)+\ve I_{d''\X d''}$, 
and because of Lemma~\ref{L:7.100} the right-hand side of the above equation is of the form 
$\po_\xi n^\ve(t,x,\xi)$ where $n^\ve$ is a measure on $(0,T)\X\cO\X(-L_0,L_0)$ with total variation $|n^\ve|$ satisfying 
$\bbE |n^\ve|\le C$, for some $C>0$ independent of $\ve$. Reasoning as in \cite{GH}, we see that the symbol 
$$
\LL^\ve(i\tau,i \k, \xi):=i(\tau+\abf(\xi)\cdot \k) + {\k''}^\top \bbf^\ve(\xi)\k'',
$$  
$(\tau,\k)=(\tau,\k',\k'')\in\R\X\R^{d'}\X\R^{d''}$ satisfies condition \eqref{e1.4'''}, uniformly in $\ve$. Moreover, given any  
$V\Subset \cO$ and $\phi\in C_c^\infty(\cO)$, with $\phi(x)=1$, for $x\in V$,  we see that $\ff^{\phi,\ve}:=\phi\ff^\ve$ satisfies  
\begin{multline}\label{e8.2}
\po_t \ff^{\phi,\ve} + \abf(\xi)\cdot \nabla_{x} \ff^{\phi,\ve} + \bbf^\ve(\xi):D_{x''}^2 \ff^{\phi,\ve} \\
= q_\xi^{\phi,\ve}+\s^\ve(\xi)\nabla_{x''}\phi \cdot \s^\ve(\xi)\nabla_{x''}\ff^\ve+ \sum_{k=1}^\infty \phi(x) g_k(\xi)(\po_\xi\ff^\ve) \dot\beta_k(t) \\+\sum_{k=1}^\infty\d_0(\xi) \phi(x) g_k(\xi)\dot\beta_k(t),
\end{multline}
where 
$$
q^{\phi,\ve}=\phi (m^\ve-\frac12G^2(\xi)\,\d_{u^\ve(t,x)}(\xi))-\ff^\ve\abf(\xi)\cdot\nabla_{x}\phi.
$$
Now, we also have that  $n^{\phi,\ve}:=\s^\ve(\xi)\nabla_{x''}\phi \cdot \s^\ve(\xi)\nabla_{x''}\ff^\ve$ is a.s.\ a finite total variation measure on $(0,T)\X\cO\X(-L_0,L_0)$ such that 
$\bbE|n^{\phi,\ve}|\le C$,   because of  \eqref{e1.4''}, \eqref{e1.4'''} and Lemma~\ref{L:7.100}.  After extending $\ff^{\phi,\ve}$ periodically in the space variable $x$ with a period $\Pi\supset \supp\phi$, we can apply the averaging lemma by Gess and Hofmanov\'a in \cite{GH} to deduce that 
\begin{equation}\label{e8.3}
 \| u^\ve\|_{L^r(\Om\X(-T_0, T_0); W^{s,r}(V))}\le C_{V},
 \end{equation}
 for some $C_V$ independent of $\ve$ and some $1<r<2$, $0<s<1$, for any $V\Subset\cO$. Again, using Kolmogorov's continuity theorem as in proposition~4.4 of \cite{DHV} we get that 
\begin{equation}\label{e8.4} 
\bbE\|u^{\ve}\|_{C^\l([0,T];H^{-2}(\cO))}\le C,
\end{equation}
for any $\l\in(0,1/2)$ for some $C>0$, independent of $\ve$. Define, $\mathcal{X}_u=L^2([0,T];L^2(\cO))\cap C([0,T];H^{-3}(\cO))$, $\mathcal{X}_W=C([0,T];\frak{U}_0)$ and $\mathcal{X}=\mathcal{X}_u\X\mathcal{X}_W$. Let $\mu_{u^\ve}$ be the law of $u^\ve$ in $\mathcal{X}_u$, $\mu_W$ be the law of $W$ in $\mathcal{X}_W$, and $\mu_\ve$ be the joint law of $(\mu_{u^\ve},\mu_W)$ in $\mathcal{X}$.  {}From \eqref{e8.3} and \eqref{e8.4}, as in \cite{FL1}, we conclude the tightness of the $\mu_\ve$ in $\mathcal{X}$, and so the pre-compactness of these laws in $\mathcal{X}$. Then one applies Skorokhod's representation theorem to obtain a new probability space $(\tilde \Om; \tilde \bbP)$ and a subsequence of random variables  $(\tilde u_{\ve_j}, \tilde W):\tilde \Om\to \mathcal{X}$, whose laws $\tilde \mu_{\ve_j}$ are equivalent to $\mu_{\ve_j}$ such that $\tilde u_{\ve_j}$ converges in measure to some $\tilde u:\tilde \Om\to\mathcal{X}$. In particular, $\tilde u_{\ve_j}$  converges a.s.\ in $L^2((0,T)\X\cO\X (-L_0,L_0))$ to a certain $\tilde u:\tilde \Om\to \mathcal{X}_u$.  Then, one can reason as in \cite{Ha,FL1}, to prove that $\tilde u$ is a martingale solution to \eqref{e1.1}--\eqref{e1.4}, that is, $\tilde u$ is a kinetic solution of \eqref{e1.1}--\eqref{e1.4}, with $\tilde \Om$ and $\tilde W$ instead of $\Om$ and $W$. Observe that the verification of the Neumann condition follows directly form the convergence in $L^2((0,T)\X\cO\X(-L,L))$. On the other hand, the verification of the Dirichlet  conditions may be performed using the arguments of Section~4 in \cite{FL2} with slight adaptations, to which we refer for the details. Hence,  because of the uniqueness of the kinetic solution of \eqref{e1.1}--\eqref{e1.4} established by Theorem~\ref{T:2.2}, we may apply Gy\"only-Krylov's criterion to conclude that the whole sequence $u^\ve$
converges to a kinetic solution of \eqref{e1.1}--\eqref{e1.4}, which concludes the prove of the existence of a kinetic solution to \eqref{e1.1}--\eqref{e1.4}.

\appendix

\section{Spectral analysis of an elliptic operator} \label{appendixA}

On this additional section, we will provide a detailed spectral analysis of the operator $A: D(A) \subset L^2(\cO'\times \cO'') \rightarrow L^2(\cO' \times \cO'')$ given by
$$\begin{cases}
D(A) &= \{ u \in L^2(\cO' \times \cO''); L^2(\cO''; H^2(\cO')) \cap u \in L^2(\cO'; (H_0^2 \cap H^4)(\cO'))  \\  &\quad\quad\quad\quad\quad\quad\quad\quad\quad \text{ and, in the sense of traces, } \partial_\nu u = 0 \text{ on } \partial \cO' \times \cO'' \}, \\
A &= -\ve\Delta_{x',x''} + \mu\Delta_{x''}^2.
\end{cases}$$
Our investigation of this operator is motivated by the approximate problem \eqref{e6.4}--\eqref{e6.8}, in which $A$ arises naturally. To better understand  $A$  we  next recall  the theory of the  tensorial product of Hilbert spaces. 

\subsection{Definition of the tensor product of two Hilbert spaces}

Let $H_1$ and $H_2$ be two (complex) Hilbert spaces, whose scalar products are, respectively,  $(\quad,\quad)_{H_1}$ and $(\quad,\quad)_{H_2}$.  The rigorous construction of tensor product space $H_1 \widehat{\otimes} H_2$ can be performed as follows (see, e.g., \cite{We}).

Let us first define the so-called algebraic tensor product between two (complex) linear spaces $E_1$ and $E_2$, not necessarily endowed with any kind of topology. Consider the set 
\begin{align*}
F(E_1, E_2) =  \Big\{ \sum_{j \in J} a_j [x_j, y_j] ;\, &J \text{ is finite}, a_j \in \mathbb{C}, x_j \in E_1,\\&\text{and } y_j \in E_2  \text{ for all } j \in J \Big\},
\end{align*}
 i.e., the set of all formal linear combinations of elements $[x,y] \in E_1 \times E_2$.
Notice that $F(E_1, E_2)$ is again another linear space. Thus, we can also consider the subspace of $N \subset F(E_1, E_2)$ consisting of sums of elements of the form
$$\sum_{j\in J} \sum_{k \in K} a_j b_k [x_j, y_k] + (-1) \Big[ \sum_{j\in J} a_j x_j, \sum_{k\in K} b_k y_k \Big],$$
where $J$ and $K$ are finite index sets, and $a_j, b_k \in \C$, $x_j \in E_1$, and $y_k \in E_2$ for every $j \in J$ and $k \in K$. 

\begin{definition}
With the notations of the preceding paragraph, the quotient space
$$E_1 \otimes E_2 = F(E_1, E_2) / N$$
is the so-called \textit{algebraic tensor product of $E_1$ and $E_2$}.
\end{definition}

Observe that for every $[x, y] \in H_1 \times H_2$ we may define the \textit{simple tensor}
$$x \otimes y = [x,y] + N \in H_1 \otimes H_2.$$ 
As a result, we may understand the algebraic tensor product $E_1 \otimes E_2$ as the space of the linear combination of simple tensors. 

Let us now assume that $E_1$ and $E_2$ are pre-Hilbert spaces, possessing, respectively, the scalar products $(\quad,\quad)_{E_1}$ and $(\quad,\quad)_{E_2}$. It is easy to verify that the formula
$$\Big( \sum_{j\in J} a_j u_j \otimes v_j, \sum_{k \in K} a_k' u_k' \otimes v_k' \Big)_{E_1 \otimes E_2} = \sum_{j\in J} \sum_{k\in K} a_j \overline{a_k'} (u_j, u_k')_{E_1} (v_j, v_k')_{E_2}$$
gives rise to a well-defined scalar product in $E_1 \otimes E_2$. However, even if $E_1 = H_1$ and $E_2 = H_2$ (i.e., $E_1$ are $E_2$ Hilbert spaces), in general $\langle H_1 \otimes H_2, (\quad,\quad)_{H_1 \otimes H_2} \rangle$ fails to be complete; indeed, this would be true if, and only if, $H_1$ or $H_2$ were of finite dimension. Nevertheless, we may always complete such space, thus obtaining  a Hilbert space. 

\begin{definition}
With the same hypothesis and notations of the previous parapragh, the completion of $H_1 \otimes H_2 $ under $(\quad,\quad)_{H_1 \otimes H_2}$, denoted by $H_1 \,\widehat{\otimes}\, H_2 $, is the so-called (complete) \textit{tensor product of $H_1$ and $H_2$}.
\end{definition}

Henceforth, we will identify the algebraic tensor product $H_1 \otimes H_2$ as being a subset of the complete tensor product $H_1 \,\widehat{\otimes}\, H_2$. From $H_1 \otimes H_2$ being dense in $H_1 \,\widehat{\otimes}\, H_2$, the following simple yet useful fact can be easily deduced (see \cite{We}, theorem 3.12).

\begin{proposition} \label{propA.1}
Let $H_1$ and $H_2$ be two Hilbert spaces. 
\begin{enumerate}
  \item If $M_1$ and $M_2$ are total subsets of, respectively, $H_1$ and $H_2$, then the set $\{ u \otimes v \}_{u \in M_1, v \in M_2}$ is total in $H_1 \,\widehat{\otimes}\, H_2$.
  \item If $\{ e_\alpha \}_{\alpha \in A}$ and $\{ f_\beta \}_{\beta \in B}$ are Hilbert bases of, respectively, $H_1$ and $H_2$, then $\{ e_\alpha \otimes f_\beta \}_{\alpha \in A, \beta \in B}$ is a Hilbert basis of $H_1 \,\widehat{\otimes}\, H_2$.
\end{enumerate}

\end{proposition}
\subsection{Operators in tensor product spaces}

Let $H_1$, $H_2$, $V_1$ and $V_2$ be Hilbert spaces. 

\begin{definition}
Given two possibily unbounded linear operators $T_1 : D(T_1) \subset H_1 \rightarrow V_1$ and $T_2 : D(T_2) \subset H_2 \rightarrow V_2$, we will define its tensor product $T_1 \otimes T_2 : D(T_1 \otimes T_2) \subset H_1 \widehat\otimes H_2 \rightarrow V_1 \widehat{\otimes} V_2$ as the linear operator by the formulae
\begin{equation*}
\begin{cases}
  D(T_1 \otimes T_2) = D(T_1) \otimes D(T_2), \text{ and} \\
  (T_1 \otimes T_2) \big( \sum_{j \in J} b_j u_j \otimes v_j \big) = \sum_{j \in J} b_j\, T_1 u_j \otimes T_2 v_j,
\end{cases}
\end{equation*}
where $J$ is a finite index set, and $b_j \in \C$, $u_j \in H_1$ and $v_j \in H_2$ for every $j \in J$.
\end{definition}

Concerning the above definition, it can be shown that the value of $(T_1\otimes T_2)(\sum b_j u_j \otimes v_j)$ independs of the representation of $\sum b_j u_j \otimes v_j$; this can be easily proven from the linearity of $T_1 \otimes T_2$ (see \cite{We}, section 8.5).

Next we recall an important theorem in the theory of tensor product of operators in tensor product of Hilbert spaces (see \cite{We}, theorem 8.33).

\begin{theorem} \label{thmA.1}
Let $T_1 : D(T_1) \subset H_1 \rightarrow H_1$ and $T_2 : D(T_2) \subset H_2 \rightarrow H_2$ be two linear operators, and define $A : D(A) \subset H_1 \widehat{\otimes} H_2 \rightarrow H_1 \widehat{\otimes} H_2$ by
$$\begin{cases}
  D(A) = D(T_1) \otimes D(T_2), &\text{ and} \\
  A = T_1 \otimes I_{H_2} + I_{H_1} \otimes T_2,&
\end{cases}$$
  with $I_{H_1}$ stands for the identical operator in $H_1$, etc. 
 If $T_1$ and $T_2$ are essentially self-adjoint operators on $H_1$ and $H_2$ respectively, then $T_1 \otimes T_2$ and $A$ are essentially self-adjoint on $H_1 \,\widehat{\otimes}\, H_2$.
\end{theorem}

Let us explore this result in light of the Spectral Theorem. For this purpose, let us recall this well known  result in its multiplicative operator form, whose statement we reproduce from \cite{RS}:

\begin{theorem} \label{spectralthm}
Let $T$ be a self-adjoint operator on a separable Hilbert space $H$ with domain $D(T)$. Then there is a measure space $(M, \mu)$ with a finite measure $\mu$, a unitary operator $U: H \rightarrow L^2(M, d\mu)$, and a real-valued function $f$ on $M$ which is finite a.e. such that
\begin{enumerate}
  \item  $u \in D(T)$ if, and only if, $ f( \> \cdot \>) (U u) (\> \cdot \>) \in L^2(M, d\mu) $;
  \item If $\psi \in U(D(T))$, then $(U\, T\, U^{-1} \psi)(m) = f(m) \psi(m)$.
\end{enumerate}
\end{theorem}

Combining these two theorems, we can deduce the following tensorial Spectral Theorem, which follows  in a standard way from the classical spectral theorem~\ref{spectralthm}, and so we omit its proof here.  

\begin{theorem} \label{tensorialspectralthm}
  Let $H_1$ and $H_2$ be two separable Hilbert spaces, and consider two self-adjoint operators $T_1 : D(T_1) \subset H_1 \rightarrow H_1$ and $T_2 : D(T_2) \subset H_2 \rightarrow H_2$. Let also $U_1 : H_1 \rightarrow L^2(M_1, \mu_1)$ and $U_2 : H_2 \rightarrow L^2(M_2, \mu_2)$ be two unitary maps, with $(M_1, \mu_1)$ and $(M_2, \mu_2)$ being finite measure spaces, such that for $j = 1, 2$:
\begin{itemize}
  \item  $u \in D(T_j)$ if, and only if, $ f_j( \, \cdot \,) (U_j u) (\, \cdot \,) \in L^2(M_j, d\mu_j) $, for some measurable function $f_j : M_j \rightarrow \R$;
  \item If $\psi \in U_j (D(T_j))$, then $(U_j\, T_j\, U_j^{-1} \psi)(m) = f_j(m) \psi(m)$.
\end{itemize}

  If $A = T_1 \otimes I_{H_2} + I_{H_1} \otimes T_2$, $A$ is essentially self-adjoint. Writing $\Lambda = \overline{A}$ and $U = \overline{U_1 \otimes U_2} : H_1 \,\widehat{\otimes}\, H_2 \rightarrow L^2(M_1 \times M_2, \mu_1 \times \mu_2)$, $U$ defines an unitary map such that
  \begin{enumerate}
  \item  $u \in D(\Lambda)$ if, and only if, $ \big( f_1(m_1) + f_2 (m_2) \big) (U u) (m_1, m_2) \in L^2(M_1 \times M_2, d\mu_1 \times d\mu_2) $;
  \item If $\psi \in U(D(\Lambda))$, then $(U\, \Lambda\, U^{-1} \psi)(m_1, m_2) = \big(f_1(m_1) + f_2(m_2) \big) \psi(m_1, m_2)$.
\end{enumerate}   
\end{theorem}
\begin{remark}
Under the same conditions, an analogous theorem can be deduced for $T_1 \otimes T_2$: it is an essentially  self-adjoint operator, and its closure is equivalent to the multiplication operator $\psi(m_1, m_2) \mapsto f_1(m_1)f_2(m_2)\psi(m_1,m_2)$.
\end{remark}

\subsection{The nonnegative case} \label{A:nonnegative}

Preserving the notations and hypoteheses of Theorem \ref{tensorialspectralthm}, let us briefly investigate the situation in which both operators $T_1$ and $T_2$ are nonnegative.

 Let $H$ be a separable Hilbert space and $T : D(T) \subset H \rightarrow H$ is a self-adjoint operator, and consider the unitary operator $U$, the measure space $(M,\mu)$ and the measurable function $f : M \rightarrow \R$ given by Theorem \ref{spectralthm}. If $T$ is nonnegative, then the function given $f : M \rightarrow \mathbb{R}$ is nonnegative as well (and conversely). Thence,
\begin{itemize}
  \item if $S(t) = \exp \{ -t T \}$ is the semigroup associated with $-T$, then $(U S(t) \psi)(m)$ $=$ $\exp \{ -t f(m) \} (U\psi)(m)$ for any $\psi \in H$, $t > 0$ and $m \in M$, and
  \item for $\alpha \geq 0$, we may characterize the intermediate spaces $D(T^\alpha)$ as
  $$D(T^\alpha) =: H_T^\alpha = U^{-1} \big( L^2(M, (1 + f(m))^{2\alpha} d\mu) \big) =: U^{-1}(X_f^\alpha).$$
  When $0\leq \alpha \leq 1$, these are precisely the interpolation spaces $[H, D(T)]_\alpha$ studied in \cite{LM}, vol. 1. We can also extend this definition for $\alpha < 0$, by setting $H_T^\alpha = (H_T^{-\alpha})^*$, for these spaces can still be naturally identified with $X_f^\alpha =  L^2(M, (1 + f(m))^{2\alpha} d\mu)$. Except when mentioned, the spaces $H_T^\alpha$ are equiped with the induced norms of $X_f^\alpha$ (which are equivalent to the graph norms when $\alpha \geq 0$), and thus are also Hilbert. Notice then that 
\begin{equation}
  H_T^\alpha \supset H_T^\beta \text{ with dense and continuous injection} \label{eqA.3}
\end{equation}
provided that $\alpha \leq \beta$.  
\end{itemize}
With this in mind, it is not hard to verify the following consequence of Theorem \ref{tensorialspectralthm}. We omit 
its proof  for the sake of brevity. 

\begin{corollary} \label{coroltensorspectralthm}
Under the same hypotheses and notations of Theorem \ref{tensorialspectralthm}, assume that $T_1$ and $T_2$ are positive, and denote by $H_{T_1}^\alpha$ and $H_{T_2}^\alpha$ the intermediate spaces of $T_1$ and $T_2$ respectively. Then $\Lambda = \overline{T_1 \otimes I_{H_2} + I_{H_1} \otimes T_2}$ is a nonnegative self-adjoint operator. Regarding its intermediate spaces $H_{\Lambda}^\alpha$, if $\alpha \geq 0$,
\begin{equation}
  H_\Lambda^\alpha = D(\Lambda^\alpha) = H_1 \, \widehat{\otimes} \, H_{T_2}^\alpha \, \cap \, H_{T_1} ^\alpha \, \widehat{\otimes} \, H_2 \text{ with equivalent norms, } \label{eqA.4}
\end{equation}
and, if $\alpha < 0$,
\begin{equation}
H_\Lambda^\alpha = H_1 \,\widehat{\otimes}\, H_{T_2}^\alpha \, + \, H_{T_1}^\alpha \, \widehat{\otimes} \, H_2 \text{ with equivalent norms}. \label{eqA.5}
\end{equation}
\end{corollary}

\subsection{An application to a differential operator}\label{A4}

At last, we will analyze the operator stated on the very first paragraph of this appendix. The proof of the following result follows from the theory developed so far in this section and we omit it here for the sake of brevity.

\begin{theorem} \label{thmspectralA}
Let $\cO' \subset \R^{d'}$ and $\cO'' \subset \R^{d''}$ be smooth nonempty bounded open sets, and let $A: D(A) \subset L^2(\cO'\times \cO'') \rightarrow L^2(\cO' \times \cO'')$ be the unbounded operator
$$\begin{cases}
D(A) &= \{ u \in L^2(\cO' \times \cO''); L^2(\cO''; H^2(\cO')) \cap u \in L^2(\cO'; (H_0^2 \cap H^4)(\cO''))  \\  &\quad\quad\quad\quad\quad\quad\quad\quad\quad \text{ and, in the sense of traces, } \partial_\nu u = 0 \text{ on } \partial \cO' \times \cO'' \}, \\
A &= -\ve\Delta_{x' + x''} + \mu\Delta_{x''}^2.
\end{cases}$$
Then $A$ defines a nonnegative self-adjoint operator, which can be understood as 
$A = \overline{I_{H_1} \,\widehat \otimes\, T_2 +  T_1 \,\widehat \otimes\, I_{H_2}}$
where $T_j : D(T_j) \subset H_j \rightarrow H_j$ ($j=1, 2$) are given by
$$\begin{matrix}
&H_1 = L^2(\cO'),  &D(T_1) = \{f \in H^2(\cO'); \partial_\nu f = 0 \text{ on } \partial \cO' \}, &T_1 = -\varepsilon \Delta_{x'}, \\ 
&H_2 = L^2(\cO''), &D(T_2) = H_0^2 (\cO'') \cap H^4(\cO''), \text{ and }   &T_2 = \mu \varepsilon \Delta_{x''} - \ve \Delta_{x''}.
\end{matrix}$$
Moreover, denoting by $H_A^\alpha$ its intermediate spaces (see the previous subsection), we have that with equivalent norms
\begin{align*}
  H_A^{1/2} &= L^2(\cO'; H_0^2(\cO'')) \cap L^2(\cO''; H^1(\cO')), \\
  H_A^{1/4} &= L^2(\cO'; H_0^1(\cO'')) \cap L^2(\cO''; H^{1/2}(\cO')), \\
  H_A^{-1/4} &= L^2(\cO'; H^{-1}(\cO'')) + L^2(\cO''; H^{1/2}(\cO')^*), \text{ and} \\
  H_A^{-1/2} &= L^2(\cO'; H^{-2}(\cO'')) + L^2(\cO''; H^1(\cO')^*).
\end{align*}
\end{theorem}

\begin{remark}
For $j=1,2$, the spectrum of $T_j$ is a sequence $0 \leq \lambda_1^{(j)} \leq \lambda_2^{(j)} \leq \lambda_3^{(j)} \leq \ldots$ with $\lambda_n^{(j)} \rightarrow \infty$ as $n \rightarrow \infty$. In addition, there exists a Hilbert basis $\{e_1^{(j)}, e_2^{(j)}, \ldots \}$ of $H_j$ composed eigenfunctions of $T_j$ such that $T_j e_n^{(j)} = \lambda_n^{(j)} e_n^{(j)}$. By Proposition \ref{propA.1}, $\{ e_m^{(1)} \otimes e_n^{(2)} \}_{m \in \N, m \in \N}$ defines a Hilbert basis of $H_1 \, \widehat{\otimes} \, H_2$. Since $e_m^{(1)} \otimes e_n^{(2)} \in D(T_1) \otimes D(T_2)$ and $A(e_m^{(1)} \otimes e_n^{(2)}) = (\lambda_m^{(1)} + \lambda_n^{(2)}) e_m^{(1)} \otimes e_n^{(2)}$, we arrive at the following conclusion: $A$ is also diagonalizable and its eigenvalues are $\{\lambda_m^{(1)} + \lambda_n^{(2)}\}_{m \in \N, n \in \N}$. Of course, the same finding could have been reached through Theorem \ref{tensorialspectralthm}.
\end{remark}

Let us also state the following theorem which also follows from the theory developed in this appendix and will be needed in the study of the second approximate problem.

\begin{theorem} \label{thmspectralT}
Let $\cO' \subset \R^{d'}$ and $\cO'' \subset \R^{d''}$ be smooth nonempty bounded open sets, and let $T: D(T) \subset L^2(\cO'\times \cO'') \rightarrow L^2(\cO' \times \cO'')$ be the unbounded operator
$$\begin{cases}
D(T) = \{ u \in H^2(\cO); &\text{in the sense of traces, } u = 0 \text{ on } \cO' \X \partial \cO'', \text{ and }\\ &  \partial_\nu u = 0 \text{ on } \partial \cO' \times \cO'' \}, \\
T = -\ve\Delta .
\end{cases}$$
Then $T$ defines a nonnegative self-adjoint operator, which can be understood as 
$T = \overline{I_{H_1} \,\widehat \otimes\, T_2 +  T_1 \,\widehat \otimes\, I_{H_2}}$
where $T_j : D(T_j) \subset H_j \rightarrow H_j$ ($j=1, 2$) are given by
$$\begin{matrix}
&H_1  &= L^2(\cO'),  & D(T_1) = \{f \in H^2(\cO'); \partial_\nu f = 0 \text{ on } \partial \cO' \}, &T_1 = -\varepsilon \Delta_{x'}, \\ 
&H_2  &= L^2(\cO''), & D(T_2) = \{g \in H^2(\cO''); g = 0 \text{ on } \partial \cO'' \}, & \text{ and }   T_2 = -\varepsilon \Delta_{x''}.
\end{matrix}$$
Consequently, with regard to its intermediate spaces $H_T^\alpha$, we have that with equivalent norms
\begin{align*}
  H_T^{1/2}  &= L^2(\cO'; H_0^1(\cO'')) \cap L^2(\cO''; H^1(\cO'') , \\
  H_T^{-1/2} &= L^2(\cO'; H^{-1}(\cO'')) + L^2(\cO''; H^1(\cO')^*).
\end{align*}
\end{theorem}

\begin{remark}
The fact that $D(T_2)^{1/2} = H_0^1(\cO'')$ can be proven just as we had shown that $D(T_2) \in H^1(\cO')$ in \cite{FL1}; actually the computations are simpler in this case.
\end{remark}

\section{Some convolution estimates on Hilbert spaces}\label{B}

For future reference, we will state here some convolution estimates on Hilbert spaces. Let $H$ be a separable Hilbert space and $A : D(A) \subset H \rightarrow H$ be a nonnegative self-adjoint operator; let us also denote by $H_A^\alpha$ its intermediate spaces (see Subsection \ref{A:nonnegative}).

We state the following three propositions that are used in the text, for whose  proof  we refer to \cite{FL1} and \cite{FL2}.

\begin{proposition} \label{B:reg0}
    For any real number $s \geq 0$ and $t>0$, the expression $A^s S(t)$ defines a \textit{bounded} linear operator in $H$. Moreover,
    \begin{equation}
        \Vert \, A^s S(t) \, \Vert_{\mathscr{L}(H)} \leq c_s/t^s. \label{ineq1}
    \end{equation}
    Consequently, for $\alpha < \beta$ and $t > 0$, $S(t)$ is a bounded linear operator from $H_A^\alpha$ into $H_A^\beta$ whose norm may be majorized by
    $$\Vert \,  S(t) \, \Vert_{\mathscr{L}(H^\alpha_A; H^\beta_A)} \leq c_{\alpha,\beta} \Big( 1 + \frac{1}{t^{\beta - \alpha}} \Big).$$
\end{proposition}
 
\begin{proposition} \label{B:reg1}
  For any $-\infty < \alpha < \infty$, define the Duhamel convolution operator
$$
(\mathcal{I} h)(t) = \int_0^t S(t-s) h(s) ds
$$
for $h \in L^2(0,T; H_A^\alpha)$.

Then $\mathcal{I}$ maps $L^2(0,T;H_A^\alpha)$ into $L^2(0,T;H_A^{\alpha+1})$ and
  \begin{equation}
    \int_0^T \Vert \mathcal{I} h(s) \Vert_{H_A^{\alpha+1}}^2 ds \leq C \int_0^T \Vert h(s) \Vert_{H_A^\alpha}^2 ds, \label{regduhamel}
  \end{equation}
  for some absolute constant $C$ depending only on $T$.
\end{proposition}

\begin{proposition}  \label{B:reg2}
Let $(\Omega, \mathscr{F}, (\mathscr{F})_{t\geq 0}, \mathbb{P})$ be stochastic basis with a complete and right-continuous filtration, and let $W$ be a cylindrical Wiener process; i.e.,
$$
W(t) = \sum_{k=1}^\infty \beta_k(t) e_k
$$
where the $\beta_k$'s are mutually independent real-valued standard Wiener processes relative to $(\mathscr{F}_t)_{t\geq 0}$, and $(e_k)$ is an orthonormal basis of another separable Hilbert space $\mathfrak{U}$.  

For some $-\infty< \alpha <\infty$, assume that $\Psi \in L^2((0,T)\times\Omega;L_2(\mathfrak{U};H_A^\alpha))$ is predictable. Then, if $(\mathcal{I}_W \Psi)(t)$ is the stochastic convolution 
$$
(\mathcal{I}_W \Psi)(t) = \int_0^t S(t-s) \Psi(s)\,dW(s),
$$
then $\mathcal{I}_W  \Psi\in L^2(\Omega\X[0,T];H_A^{\alpha+1/2})$ and
\begin{equation}\label{regduhamel2}
\Vert \mathcal{I}_W  \Psi \Vert_{L^2(\Omega;L^2(0,T;H_A^{\alpha+1/2}))} \leq C \Vert \Psi \Vert_{L^2(\Omega;L^2(0,T;L_2(\mathfrak{U};H_A^\alpha))},
\end{equation}
for some $C>0$ depending only on $T$. 
\end{proposition}

\end{document}